\documentclass[final,10pt,reqno]{amsart}
\usepackage{comment} 

\usepackage[left=1.0in, right=1.0in, top=1.5in, bottom=1.5in]{geometry}

\usepackage[foot]{amsaddr} 
\usepackage{stmaryrd}

\usepackage{verbatim}
\usepackage{amssymb}
\usepackage{mathrsfs}
\usepackage{amsthm}
\usepackage{amsmath}
\usepackage{amsfonts}
\usepackage{latexsym}
\usepackage{mathtools} 
\usepackage{enumitem} 
\usepackage[english]{babel} 
\usepackage{mathabx}
\usepackage{scalerel,stackengine}

\usepackage{float}

\usepackage[pdftex,usenames,dvipsnames]{xcolor}  

\definecolor{dkblue}{RGB}{30,90,140} 
\usepackage[colorlinks=true, pdfstartview=FitV, linkcolor=dkblue, citecolor=dkblue, urlcolor=dkblue]{hyperref}    
\definecolor{mydarkbluett}{RGB}{12,111,174}

\usepackage[colorinlistoftodos,prependcaption,textsize=tiny]{todonotes}

\stackMath
\newcommand\reallywidehat[1]{%
\savestack{\tmpbox}{\stretchto{%
  \scaleto{%
    \scalerel*[\widthof{\ensuremath{#1}}]{\kern-.6pt\bigwedge\kern-.6pt}%
    {\rule[-\textheight/2]{1ex}{\textheight}}
  }{\textheight}%
}{0.5ex}}%
\stackon[1pt]{#1}{\tmpbox}%
}

\emergencystretch=\maxdimen
\theoremstyle{plain}
\newtheorem{teor}{Theorem}
\newtheorem{obs}[teor]{Remark}
\newtheorem{prop}[teor]{Proposition}
\newtheorem{coro}[teor]{Corollary}
\newtheorem{lema}[teor]{Lemma}
\newtheorem{defi}[teor]{Definition}
\newtheorem{asu}[teor]{Assumption}

\numberwithin{equation}{section}
\numberwithin{teor}{section}

\newcommand{\norma}[1]{{\left\vert\kern-0.25ex\left\vert\kern-0.25ex\left\vert #1 
    \right\vert\kern-0.25ex\right\vert\kern-0.25ex\right\vert}}
\def\R{\mathbb{R}}

\def\N{\mathbb{N}_0}
\def\T{\mathbb{T}}

\def\Z{\mathbb{Z}}
\def\z{\mathcal{Z}}
\def\fou{e^{i(r-\zeta)\cdot\xi}}
\def\foul{e^{ir\cdot\xi}}

\def\p{\text{Op}}
\def\f{\frac{1}{(2\pi)^2}}
\def\dive{\text{div}\,}

\addto\extrasenglish{
  
}

\newcommand{\norm}[1]{\left\lVert#1\right\rVert} 
\DeclarePairedDelimiter\abs{\lvert}{\rvert}%

\begin{document}

	\title[The 3D MHS equation with Grad-Rubin boundary conditions]{The three dimensional magneto-hydrostatic equations with Grad-Rubin boundary value}
	\author[D. Alonso-Or\'an]{Diego Alonso-Or\'an}
	\address{Departamento de An\'{a}lisis Matem\'{a}tico and Instituto de Matem\'aticas y Aplicaciones (IMAULL), Universidad de La Laguna, C/Astrof\'{i}sico Francisco S\'{a}nchez s/n, 38271, La Laguna, Spain.}
	\email{\href{mailto:dalonsoo@ull.edu.es}{dalonsoo@ull.edu.es}}
 
	\author[D.S\'anchez-Sim\'on del Pino]{Daniel S\'anchez-Sim\'on del Pino}
 \address{University of Bonn, Institute for Applied Mathematics Endenicher Allee 60, D-53115, Bonn, Germany.}
	\email{\href{mailto:sanchez@iam.uni-bonn.de}{sanchez@iam.uni-bonn.de}}
	
 \author[J. J.L. Vel\'azquez]{Juan J.L. Vel\'azquez }
	\address{University of Bonn, Institute for Applied Mathematics Endenicher Allee 60, D-53115, Bonn, Germany.}
	\email{\href{mailto:velazquez@iam.uni-bonn.de}{velazquez@iam.uni-bonn.de}}

	\setcounter{tocdepth}{1}
	\begin{abstract}
    In this work, we study the well-posedness of the three dimensional magneto-hydrostatic equation under Grad-Rubin boundary value conditions. The proof relies on a fixed point argument to construct solutions to an elliptic-hyperbolic problem in a perturbative regime by means of pseudo-differential operators with symbols with limited regularity in H\"older spaces. As a byproduct, the employed technique in this work is more flexible and simplifies the arguments of the proof in \cite{Alo-Velaz-2022} for the two-dimensional setting.
        \end{abstract}
	\maketitle
	\tableofcontents
	\section{Introduction and prior results}\label{section1:intro}
Plasma is an ionized gas consisting of freely moving positively charged ions, electrons and neutrals. It is by far the most common phase of ordinary matter present in the Universe. Plasma processes can be described mathematically by the so called magnetohydrodynamic equations (MHD). These equations, first  proposed by H. Alfv\'{e}n,  describe electrically conducting fluids in which a magnetic field is present. The applications of plasma physics are many and are divided in two main areas: laboratory thermonuclear fusion and astrophysical plasmas. An important class of equations in magnetohydrodynamics are the so called magneto-hydrostatic (MHS) equations which describe steady magnetic fields with trivial flow and are given by
\begin{equation}\label{mhs1}
    \left\{
\begin{array}{ll}
    j\times B=\nabla p & \text{in }\Omega,\\
    \nabla\times B=j &\text{in }\Omega, \\
    \nabla\cdot B=0 &\text{in }\Omega
\end{array}
    \right.
\end{equation}
where $B$ denotes the magnetic field of a perfectly conducting plasma, $j=\nabla\times B$ is the current density and $p$ is the plasma pressure on a suitable three-dimensional smooth bounded domain $\Omega$.
 
Equations \eqref{mhs1} are particularly relevant to understand coronal field structures, stellar winds as well as in the study of equilibrium configurations in plasma confinement fusion \cite{Goedbloed-Poedts-2010,Goedbloed-Poedts-2010-2}. It should be noted that these equations also describe stationary solutions to the 3D Euler equations in fluid mechanics, with $B$ playing the role of the velocity field $v$ of the fluid, $j$ the role of the vorticity and $p$ being the negative of the Bernoulli function. Indeed, using the vectorial identity $j\times B= B\cdot\nabla B -\frac{1}{2}\nabla (|B|^{2})$, we have that equations \eqref{mhs1} read
\begin{equation}\label{mhs1:eul}
    \left\{
\begin{array}{ll}
    B\cdot\nabla B=\nabla (p+\frac{1}{2}|B|^{2}), & \text{in }\Omega,\\
    \nabla\cdot B=0, &\text{in }\Omega,
\end{array}
    \right.
\end{equation}
which is equivalent to the steady 3D Euler equation for incompressible fluids after the identification of variable $B\leftrightarrow v$ and $p+\frac{1}{2}|B|^{2}\leftrightarrow -p$. Subsequently, we will elaborate on the boundary conditions considered for the MHS equations \eqref{mhs1}. From a mathematical perspective, systems \eqref{mhs1} and \eqref{mhs1:eul} are identical, allowing us to apply analogous methods for the 3D stationary incompressible Euler equations \eqref{mhs1:eul}.  

In this manuscript, our focus relies on examining particular boundary value problems associated with \eqref{mhs1} where components of the magnetic field are prescribed in different parts of the boundary $\partial\Omega$. The study of boundary value problems for the magneto-hydrostatic equations (and also steady incompressible Euler equations) can be traced back at least to the seminal work of H. Grad and H. Rubin, cf.~\cite{Grad-Rubin-1958}. In this paper, the authors presented several meaningful boundary value problems related to the magneto-hydrostatic equations \eqref{mhs1} in two dimensional and three dimensional cases. As it is pointed out in \cite{Grad-Rubin-1958}, the structure of the equations \eqref{mhs1} suggest that the natural domains $\Omega$ to solve boundary value problems have the property that
$\partial\Omega=\partial\Omega_{+}\cup \partial\Omega_{-}$ where
\begin{equation}\label{domain:decom}
\partial\Omega_{-}=\{ x\in\partial\Omega: (B\cdot n)(x)<0\}, \quad \partial\Omega_{+}=\{ x\in\partial\Omega: (B\cdot n)(x)>0\}.
\end{equation}
Here $n$ denotes the outer normal to $\partial\Omega$. In \cite{Grad-Rubin-1958}, the authors proposed to prescribe in addition to the normal component of the magnetic field $B$ in the whole $\partial\Omega$ one of the following additional boundary conditions
\begin{align}
& p \ \mbox{given on } \partial\Omega_{-} , \quad p  \mbox{ given on }  \partial\Omega_{+}, \label{first:problem} \\
&  p \ \mbox{given on } \partial\Omega_{-} , \quad j\cdot n  \mbox{ given on }  \partial\Omega_{-}, \label{second:problem} \\
&  p \ \mbox{given on } \partial\Omega_{-} , \quad j\cdot n  \mbox{ given on }  \partial\Omega_{+}, \\
 &B_{\tau}=B-(B\cdot n)n \ \mbox{given on } \partial\Omega_{-}. \label{our:problem}
\end{align}
Notice that $B\cdot n$ has to be chosen in such a way that \eqref{domain:decom} holds. The role of the boundaries $\partial\Omega_{-},\partial\Omega_{+}$ can be exchanged by using a transformation $B \leftrightarrow -B$. The rigorous construction of solutions to the previous boundary value problems was not addressed in \cite{Grad-Rubin-1958}. 

To the best of our knowledge, solutions to boundary value problems for \eqref{mhs1} were first constructed by D. Lortz in \cite{Lortz-1970}, although the boundary conditions used in that article are different from \eqref{first:problem}-\eqref{our:problem}. Besides Lortz approach, there are two main techniques to construct solutions to \eqref{mhs1}, namely, the Grad-Shafranov method \cite{Grad-1967,Safranov-1966} and the vorticity transport method introduced by Alber \cite{Alber-1992}. The core principle of the Grad-Shafranov method is to transform the steady Euler or the MHS equations into an elliptic equation for which a plethora of techniques are available. The Grad-Shafranov approach is restricted to two-dimensional problems or 3D problems with some additional symmetry.  See for instance, \cite{Alo-Velaz-2021, CDG-2020,CDG-2020-2,CLV,Hamel-Nadirashvili-2017}. Recently, this technique has been also used to tackle the question of controllability of the full magneto-hydrodynamic equations, cf. \cite{KukaNovVic22}.

On the other hand, in \cite{Alber-1992}, Alber introduced an alternative method to derive solutions with non-vanishing vorticity for the 3D steady incompressible Euler equation. Essentially, Alber constructed solutions wherein the velocity field  $v=v_{0}+V$ is decomposed as a base flow and a suitable small perturbation.  The boundary value problem for the Euler equations is reduced to a fixed point problem for a function $V$ combining the fact that the vorticity satisfies a suitable transport equation and that the velocity can be recovered from the vorticity using the Biot-Savart law. In particular, using this method, Alber showed the well-posedness for the system \eqref{mhs1} together with boundary conditions \eqref{second:problem}. Several extensions adapting Alber's idea to deal with similar boundary value problems have appeared in the literature, cf. \cite{Molinet-1999,Seth-2020,Tang-Xin-2009}.

In \cite{Alo-Velaz-2021}, using the methods of both Alber and Grad-Shafranov, the first and last authors constructed two-dimensional $C^{2,\alpha}$ solutions to \eqref{mhs1} satisfying different types of boundary value conditions, cf. \cite[Table 1]{Alo-Velaz-2021} to see the class of boundary value conditions considered. Furthermore, in \cite{Alo-Velaz-2022}, introducing a novel idea based on solving an integral non-local equation the same authors studied the solvability of \eqref{mhs1} with boundary value conditions \eqref{our:problem} in the two-dimensional case. To that purpose, the authors derived H\"older estimates for a class of non-convolution singular integral operators.

The main objective of this article is to study the three dimensional boundary value problem extending the methods initiated in \cite{Alo-Velaz-2022}, i.e., our goal is to study the system \eqref{mhs1} with boundary value conditions \eqref{our:problem} in the three-dimensional setting. This issue has been addressed in \cite{Sanchez-2023}. We will delve later (c.f. Subsection \ref{sec:noveltiesandstrategy}) into the novelties and key ideas needed to extend the previous result to the three-dimensional setting. 

The time dependent analogue of the boundary value problem \eqref{mhs1} (or equivalently \eqref{mhs1:eul}) has been studied in \cite{monographrusos,Gieetal,Gieetal2,Petcu06}. 
The strategy followed in those articles consists on writing an elliptic equation for the pressure $p$ that allows to obtain the pressure in terms of the magnetic field B (or equivalently the velocity field $V$). After computing the pressure $p$, using a fixed point problem for small times $t$ one can directly compute the magnetic field (velocity field). This allows to show local existence of solutions to the time dependent boundary value problem. Notice that the existence of solutions to the time dependent problem (even globally in time solutions) does not imply the existence of a steady state solution. This particular issue has actually been remarked in the monograph \cite[Chapter 4, Section 7.4]{monographrusos}.

 \subsection{Notation}
In this subsection, we collect the different notations that will be used throughout the manuscript. 
\subsubsection*{Geometrical setting} The equations are studied on the manifold $\Omega=\mathbb{T}^{2}\times [0,L]$ with $L>0$. We also use the notation $\partial\Omega$ to denote the boundary of $\Omega$. The two components of the manifold are denoted by $\partial\Omega_{-}=\mathbb{T}^{2}\times\{0\}$ and $\partial\Omega_{+}=\mathbb{T}^{2}\times\{L\}$. We denote by $n$ the normal in $\partial\Omega$. More precisely, $n$ equals the outer normal in $\partial\Omega_+=\T^2\times \{L\}$, and the inner normal in $\partial\Omega_-=\T^2\times \{0\}$. We write the subscript $B_{\tau}$ to denote the tangential component of a generic vector field $B$ on $\partial \Omega$.
 
 \subsubsection*{Functional spaces and Fourier transform}
 We define the space of H\"older continuous functions as 
	$$ C^{\alpha}(\Omega)=\{ f\in C_{b}(\Omega): \norm{f}_{C^\alpha(\Omega)}< \infty\},$$
    where $C_{b}(\Omega)$ denotes the set of bounded continuous functions on $\Omega$. The H\"older space is equipped with the norm
 \[ \norm{f}_{C^{\alpha}(\Omega)}=\displaystyle \sup_{x\in \overline{\Omega}}\abs{f(x)}+  \left[ f \right]_{\alpha,\Omega}, \quad \left[ f \right]_{\alpha,\Omega}= \displaystyle\sup_{x\neq y; x , y\in \overline{\Omega}} \frac{\abs{f(x)-f(y)}}{\abs{x-y}^\alpha}.\]
   Similarly, for any non-negative integer $k$ we define the H\"older spaces $C^{k,\alpha}(\Omega)$ as 
	$$ C^{\alpha}(\Omega)=\{ f\in C^{k}_{b}(\Omega): \norm{f}_{C^{k,\alpha}(\Omega)}< \infty\},$$
	equipped with the norm
	$$ \norm{f}_{C^{k,\alpha}(\Omega)}=\displaystyle\max_{\abs{\beta}\leq k }\sup_{x\in \overline{\Omega}}\abs{\partial^{\beta}f(x)}+\displaystyle\sum_{\abs{\beta}=k}\left[\partial^{\beta}f \right]_{\alpha,\Omega},$$
	where $\beta=(\beta_1,\beta_2)\in \N^{2}$ and $\N=\{0,1,2,\ldots\}$. We will sometimes write the shortcut $\norm{\cdot}_{k,\alpha}$ instead of $\norm{\cdot}_{C^{k,\alpha}(\Omega)}$. Similarly, for the supremum norm of function $\norm{\cdot}_{L^{\infty}(\Omega)}$ we just write $\norm{\cdot}_{\infty}.$ 
 Notice that in the definitions above the H\"older regularity holds up to the boundary, i.e in $\overline{\Omega}$.
	We omit in the functional spaces whether we are working with scalars or vectors fields, this is $C^{k,\alpha}(\Omega,\mathbb{R})$ or $C^{k,\alpha}(\Omega,\mathbb{R}^3)$, and instead just write $C^{k,\alpha}(\Omega)$. The specific type of functional space (scalar or vector) will be clear from the context. Moreover, we will denote H\"older spaces on the boundary of the manifold, namely on $\partial\Omega$, $\partial\Omega_{+}$ and $\partial\Omega_{-}$ by $C^{k,\alpha}(\partial\Omega)$, $C^{k,\alpha}(\partial\Omega_{+})$ and $C^{k,\alpha}(\partial\Omega_{-})$ respectively.  We identify functions $f\in C^{k,\alpha}(\T\times \T)$, $k=0,1,\ldots$ and $\alpha\in (0,1)$ with the subspace of functions in $C^{k,\alpha}(\mathbb{R}\times \mathbb{R})$ such that $f(x+2\pi,y+2\pi)=f(x,y)$. 
	
The Fourier transform of an integrable $2\pi-periodic$ function $f$ in variables $x,y$ is defined as
	\[\reallywidehat{f}(m,n)=\int_{0}^{2\pi}\int_{0}^{2\pi} f(x,y){\rm e}^{-{\rm i}(x,y)\cdot (m,n)}\, dx dy.\]
 Moreover, if the function $f$ is sufficiently regular we can reconstruct it out of these Fourier coefficients according to the formula
	\[f(x,y)=\frac{1}{(2\pi)^2}\sum_{(m,n)\in{\mathbb Z}^2}\reallywidehat{f}(m,n){\rm e}^{{\rm i}(x,y)\cdot (m,n)}.\] 
To shorten the notation, we will use $\xi=(m,n)$ and $r=(x,y).$ 
We define the average of a two periodic function $f$ as
 \[ \langle f \rangle = \frac{1}{(2\pi)^{2}}\int_{\mathbb{T}^{2}} f(x,y) \, dx dy. \]
 We will write $\widetilde{f}$ to denote the mean free function $\widetilde{f}=f- \langle f \rangle$.
\subsubsection*{Operators and constants}

 Let $X$ and $Y$ be  Banach spaces. We say that $T$ is a bounded operator from $X$ to $Y$ if there exists a constant $c\geq 0$ such that $\norm{Tu}_{Y}\leq c\norm{u}_{X}$, $\forall u\in E.$ The norm of the bounded operator $T$ is defined and denoted as
\[\norm{T}_{\mathcal{L}(X,Y)}= \displaystyle\sup_{u\neq 0}\frac{\norm{Tu}_{Y}}{\norm{u}_{X}}. \]
Moreover, if $X=Y$, we just write $\mathcal{L}(X)$ instead of $\mathcal{L}(X,X)$. We will also use the brackets $\big[ \cdot \big]$, in order to denote the dependence of an operator on the bracketed function, namely $T\big[ f\big]$ denotes that the operator $T$ depends in a certain way on the function $f$. 

 We will denote with $C$ a positive generic constant that depends only on fixed parameters. More precisely, they will depend on the parameter $L$ and the H\"older exponent $\alpha\in (0,1)$. Note also that this constant might differ from line to line.

  \subsection{The main result: novelties and strategy towards the proof}\label{sec:noveltiesandstrategy}
The main result in this manuscript deals with the well-posedness (existence and uniqueness) of a the three-dimensional problem \eqref{mhs1} with boundary conditions \eqref{our:problem}, i.e,  we prescribe the normal component of the magnetic field $B\cdot n$ on $\partial \Omega$ as well as the tangential component of $B$, this is $B_{\tau}$, in one part on the boundary. The precise statement of the main result reads
\begin{teor}\label{mainteor}Let $\alpha\in (0,1)$, $\Omega=\T^2\times [0,L]$, and denote $\partial\Omega_-=\T^2\times \{0\}$. There exists a (small) constant $M=M(\alpha,L)$ such that  for every $f\in C^{2,\alpha}(\partial\Omega)$ and $g\in C^{2,\alpha}(\partial\Omega_-)$ satisfying 
$$\|f\|_{C^{2,\alpha}(\partial\Omega)}+\|g\|_{C^{2,\alpha}(\partial\Omega_-)}\leq M,$$
and 
\begin{equation}\label{integrability:condition}
\int_{\Omega_-}fdx=\int_{\Omega_+}fdx,
\end{equation}
there exists a solution $(B,p)\in C^{2,\alpha}(\Omega)\times C^{2,\alpha}(\Omega) $ for the problem 
\begin{equation}\label{mhs2}
    \left\{
\begin{array}{ll}
    j\times B=\nabla p,& \text{in }\Omega,\\
    \nabla\times B=j,&\text{in }\Omega,\\
    \nabla\cdot B=0,&\text{in }\Omega,\\
    B\cdot n=1+f, &\text{on }\partial\Omega,\\
    B_{\tau}=g,&\text{on }\partial \Omega_{-},
\end{array}
    \right.
\end{equation}
where $n$ equals the outer normal in $\partial\Omega_+=\T^2\times \{L\}$, and the inner normal in $\partial\Omega_-=\T^2\times \{0\}$. Furthermore, this is the unique solution that satisfies 
$$\|B-(0,0,1)\|_{C^{2,\alpha}(\Omega)}\leq M.$$
\end{teor}
\subsubsection*{Novelties and strategy towards the proof}

The strategy of the proof of Theorem \ref{mainteor} follows closely the ideas used in \cite{Alo-Velaz-2022} to study the two-dimensional case of Theorem \ref{mainteor}. The key difficulty to solve the problem is to obtain the value of $j(x)$ in $\partial\Omega_{-}$, i.e., $j_0=j|_{\partial\Omega_{-}}$.  To fix ideas, let us assume that $(B,p)$ is a sufficiently smooth solution to \eqref{mhs1}. Taking the curl of \eqref{mhs1} we obtain the following transport equation for the current density $j$ that reads
\begin{equation}\label{transport:1}
    (B\cdot \nabla)j=(j\cdot \nabla)B \mbox{ in } \Omega.
\end{equation}
This equation allows, using characteristics, to compute $j$ on the full domain $\Omega$ if $j_0$ and $B$ are known. Using this value of $j$, we can obtain a unique magnetic field $B$ by solving the div-curl type problem
\begin{align}
    \nabla\times B=j, \mbox{ in } \Omega \label{div:curl}\\
    \nabla\cdot B=0,\mbox{ in } \Omega \label{div:curl2}
\end{align}
if we prescribe the normal component of the magnetic field on the full boundary, i.e., 
\begin{equation}\label{normal:component}
B\cdot n=f \ \mbox{ on } \partial\Omega,
\end{equation} 
as well as some horizontal fluxes of the magnetic field, i.e.
\begin{equation}\label{topology:constraint}
 \int_{\{x=0\}}Bd\vec{S}=J_1, \quad \int_{\{y=0\}}Bd\vec{S}=J_2.
\end{equation}
We remark that in the Biot-Savart problem \eqref{div:curl}-\eqref{normal:component} the values $J_{1},J_{2}$ are two degrees of freedoms that can be understood as part of the source term in the same way as the density current $j$.  These degrees of freedoms are due to the non-trivial topology of the domain $\Omega=\mathbb{T}^{2}\times[0,L]$ and correspond to the two linearly independent harmonic vector fields tangential to the boundaries. Notice that the equations \eqref{div:curl}-\eqref{normal:component} with $j=0$ and $f=0$ have one trivial solution satisfying \eqref{topology:constraint} given by $\frac{L}{2\pi} (J_{1},J_{2},0)$.

In order to obtain the value of $j_0(x)$ we notice that we have an additional boundary condition, namely, the tangential component of the magnetic field $B_{\tau}$ on $\partial\Omega_{-}$. Imposing that the constructed magnetic field $B$ satisfies this tangential condition yields a non-local integral equation for the function $j_0(x)$ on $\partial\Omega_{-}$. More precisely, we derive an integral equation for $j_0$ of the following type
\begin{equation}\label{current:j_0}
j_{0}(x)= \mathsf{F}[j_0,B[j_0], f, g],\mbox{ on } \partial\Omega_{-},
\end{equation}
where $\mathsf{F}$ is a non-local functional. Obtaining $j_0$ solving \eqref{current:j_0}, we find at the same time the magnetic field $B$ which solves \eqref{mhs2}. 

In order to make precise the previous ideas, we will define an operator $B\to \Gamma[B]$ whose fixed point $B$ leads to a solution $(B,p)$ to \eqref{mhs2}, see Section \ref{sec:linearized} and Section \ref{sec:6} for details.

One of the main technical challenges in this manuscript is to solve \eqref{current:j_0} in suitable H\"older spaces. One of the main novelties of this article in comparison with \cite{Alo-Velaz-2022}, is that we will use the theory of pseudo-differential operators with symbols of limited regularity while in \cite{Alo-Velaz-2022},  H\"older estimates were derived for a class of singular integral operators. This new approach has two major advantages: firstly, it simplifies and unifies the tedious \textit{hand-made} computations obtained in \cite{Alo-Velaz-2022} and secondly, it seems robust enough to provide similar results as the ones in Theorem \ref{mainteor} but in more complicated and physically relevant geometrical settings. 

It is also important to notice that, compared to the two-dimensional situation, Theorem \ref{mainteor} has new difficulties intrinsic to the three-dimensional nature of the problem. Indeed, in the two-dimensional case, the current density $j$ is simply a scalar quantity while in our case $j$ is a three-dimensional vector which in addition has to be divergence free. These constraints increase the difficulty to derive \eqref{current:j_0}. The vectorial nature of $j$ is also apparent in the transport equation \eqref{transport:1}. In this case, we have an extra stretching term on the right hand side of \eqref{transport:1} which trivially cancels in the two-dimensional case. Although this issue makes the analysis more cumbersome, it does not affect the regularity of the solution of the equation. 

As noted in the statement of Theorem \ref{mainteor}, we will restrict ourselves to a particular geometric setting, namely, $\Omega=\T^2\times [0,L]$. Similarly, as in the two-dimensional case, such choice of domain ensures there is no pathological behaviour of the characteristic curves associated to \eqref{transport:1} resulting from the fact that $B\cdot n$ might be zero in $\partial \Omega$. Indeed, for $\Omega=\T^2\times [0,L]$ we can choose $B\cdot n$ in such a way that $\Omega_-\cap\Omega_+\neq 0$ and hence $B\cdot n\neq 0$ at all points $x\in\partial\Omega$. Whether Theorem \ref{mainteor} holds true for domains where some singular behaviour for $B$ can arise, is an intriguing open problem even in the two-dimensional case.

To conclude this introduction, let us remark that in the analysis of the time-dependent problem studied in \cite{monographrusos,Gieetal,Gieetal2}, one key problem is to calculate the value of $j_0$ on the entering part of the domain $\partial\Omega_{-}$. As discussed previously, this is also a crucial issue of the problem considered in this article. Nevertheless, the value of $j_0$ can be obtained in $\partial\Omega_{-}$ integrating in time, which allows to formulate a fixed point problem for the velocity for small times. However, in the stationary case considered in this paper, to compute $j_{0}$ we need to study an integral type equation \eqref{current:j_0} which has source terms that depend on the values of $B\cdot n$ in $\partial\Omega$ as well as the tangential component $B_{\tau}$ in $\partial\Omega_{-}$. This rigidity is one of the main reasons why we have to work with magnetic fields which are small perturbations of trivial magnetic fields.  A question that would be interesting to study is whether it is possible to construct perturbed solutions  around more general magnetic fields or more general geometries.

\subsection{Plan of the paper} In Section \ref{sec:2} we first collect some important functional results about Littlewood-Paley theory and Besov spaces. Moreover, we present several key auxiliary results regarding estimates in Besov spaces for pseudo-differential operators, existence and uniqueness results for the two building block of the fixed-point operator: a div-curl system and a transport type problem. We devote Section \ref{sec:linearized} to present the construction of the method to find the solutions $(B,p)$ of the magnetohydrostatic equations satisfying the boundary value conditions \eqref{mhs2} in the linearized setting. In Section \ref{sec:4} we extend the analysis presented in Section \ref{sec:linearized} to the full non-linear setting. In particular, we derive an integral equation for the current $j$ on $\partial\Omega_{-}$. In Section \ref{sec:5} 
we provide a rigorous approach to the formal computations introduced in the previous section and show the existence of a solution in H\"older spaces to the integral equation for the current $j$ on $\partial\Omega_{-}$. More precisely, in Subsections \ref{subsec:51} and Subsections \ref{subsec:52} we derive a priori H\"older estimates for the main operators involved in the integral equation. Subsection \ref{subsec:5:3} - Subsection \ref{subsec:5:5} are devoted to establish H\"older estimates for the remaining terms of the integral equations. Later, in Subsection \ref{subsec:5:6} we show the existence and uniqueness of a solution to the integral equation by using previous derived estimates. Finally, in Section \ref{sec:6} we provide the precise definitions of the operators required to formulate the fix-point argument and prove Theorem \ref{mainteor} as a direct application of the fixed point theorem. To conclude, we gather in Appendix \ref{appendix1} complementary details of some previous stated results of Section \ref{sec:2}.

 \section{Preliminaries: functional setting and auxiliary results}\label{sec:2}
 In this section we collect the main tools needed in our analysis. More precisely, we start by recalling some well-known facts about Littlewood-Paley theory and Besov spaces. Then we present, into some extent, the theory of pseudo-differential operators in periodic H\"older spaces (and Besov spaces) together with some estimates for symbols of pseudo-differential operators that will be important in the subsequent arguments of the present article. To conclude, we establish estimates for the two building block of the fixed-point operator: a div-curl system and a transport type problem. For the sake of completeness we have also included in the Appendix \ref{appendix1} the proof of one of the stated results contained in this section. 
\subsection{Littlewood-Paley theory and Besov spaces}\label{subsec:21}
We recall here the main ideas of Littlewood-Paley theory in $\mathbb{R}^{d}$ and the relation with Besov spaces. These are classical results which can be found with full detail in \cite{Chemin-Danchin-Bahouri-2011,Triebel-1983}.

We begin by defining the so called Littlewood-Paley projections, which rely on a technical construction: the dyadic partition of unity. We have the following result
\begin{prop}
    Let $\mathcal{C}$ be the annulus $\{\xi\in \R^d\,:\, 3/4\leq |\xi|\leq 8/3\}.$ There exist radial functions $\chi$, $\varphi$ such that $\chi\in C^\infty_c(B_{4/3}(0);[0,1])$ and $\varphi\in C^\infty_c(\mathcal{C};[0,1])$ satisfying
\[ \forall \xi\in \R^d,\quad \chi(\xi)+\sum_{j\geq 0}\varphi(2^{-j}\xi)=1, \quad \sum_{j\in \Z}\varphi(2^{-j}\xi)=1.\]
Furthermore, their support satisfy 
\[ |j-j'|\geq 2\Rightarrow \text{supp }\varphi(2^{-j}\cdot)\cap \text{supp }\varphi(2^{-j'}\cdot)=\emptyset,\]
\[j\geq 0\Rightarrow \text{supp }\chi\cap \text{supp }\varphi(2^{-j}\cdot)=\emptyset.\]
\end{prop}

 By means of the previous functions $\chi$ and $\varphi$, we can define the non-homogeneous dyadic blocks $(\Delta_{j})_{j\in\mathbb{Z}}$ as
\begin{equation}\label{dyadicbloks:def1}
\Delta_j=0\quad \mbox{if } j\leq -2, \quad  \Delta_{-1}=\chi(D), \mbox{ and } \Delta_{j}=\varphi(2^{-j}D) \quad \mbox{if } j\geq 0.
\end{equation}
Here, $f(D)$ stands for the operator $u\mapsto \mathcal{F}^{-1}\left(f(\xi)\reallywidehat{u}(\xi)\right)$. Therefore, 
\begin{equation*}
        \Delta_ju=2^{jd}\widecheck{\varphi}(2^{j}\cdot)*u\, \text{ for all } j\geq 0,
\end{equation*}
and hence $\|\Delta_j\|_{L^p\rightarrow L^p}<\infty$ for every $p\in [1,\infty]$, and that these norms are independent of $j$. 
We also introduce the following low frequency cut-off operators
\begin{equation}\label{lowfreq:def1}
S_j=\chi(2^{-j}D)=\sum_{k\leq (j-1)}\Delta_{k}, \quad j\geq 0.
\end{equation}
As expected, this cut-offs operators approximate distributions in the following sense
\[ \forall u\in \mathcal{S}'(\R^d), \quad \quad  u=\displaystyle\sum_{j\geq -1}\Delta_{j} u \quad \mbox{ in the sense of  } \mathcal{S}'(\R^d). \]
By means of the Littlewood-Paley decomposition, we can define the following class of non-homogeneous Besov spaces
\begin{defi}
    Let $s\in \R$ and $p,q\in[1,\infty]$. We say that a tempered distribution $u\in \mathcal{S}'(\R^d)$ belongs to the non-homogeneous Besov space $B^{s}_{p,q}$ if
    \begin{equation}\label{besov}\|u\|_{B^s_{p,q}}=\left\|\left(2^{js}\|\Delta_j u\|_{L^p}\right)_{j\in \Z}\right\|_{\ell ^q}<\infty.
    \end{equation}
\end{defi}
Besov spaces are complete normed spaces $(B^s_{p,q},\|\cdot\|_{B^s_{p,q}})$ and are independent of the partition of the unity used to construct them. Moreover, non-homogeneous Besov spaces have nice properties of duality. More precisely, for $p\in [1,\infty)$ we have that the dual space of $B^{s}_{p,q}$, denoted by $(B^{s}_{p,q})'$, coincides with
\[ (B^{s}_{p,q})'=B^{-s}_{p',q'},\]
where $p',q'$ are the conjugate exponents of $p,q$, respectively, cf. \cite{Triebel-1983}.

Besov spaces are interpolation spaces between Sobolev spaces $W^{k,p}(\R^d)$. In fact, an immediate application of Littlewood-Paley theory gives us that the spaces $W^{k,p}(\R^d)$ coincides with the Besov spaces $B^k_{p,p}$ and their norms are equivalent for $p\in (1,\infty)$. In then endpoint case $p=\infty$, Besov spaces can be identified with H\"older spaces. In particular, the H\"older space $C^{m,s}(\R^d)$ corresponds with the space $B^{m+s}_{\infty,\infty}(\R^d)$ for any $s\in (0,1)$. Moreover, the norms in both spaces are equivalent. We will use the following identification (cf. \cite{Triebel-1983}):
   \begin{equation}
B^s_{\infty,\infty}(\R^d)\sim C^{\lfloor s\rfloor, s-\lfloor s\rfloor}(\R^d), \mbox{ for } s\notin\N.
    \end{equation}
In the sequel we will assume that $s\notin \N$.

\subsection{Symbols, pseudo-differential operators and adjoints}\label{subsec:22}
In this section, we collect some classical results on pseudo-differential operators. Our analysis will be restricted to pseudo-differential operators with symbols of limited regularity. We will make stress on the notion of adjoint of a pseudo-differential operator and the boundedness of those operators in H\"older spaces. For a thorough and systematic review of smooth pseudo-differential operators we refer the reader to \cite{Abels-2012, Jurgen-87, Taylor-1981}. 

Let us start by defining the class of symbols we will deal with throughout the article:
\begin{defi}\label{symbol:class:def}
Let $s\in\mathbb{R}_+\setminus\N$, and $m\in\R$. We define the symbol class $S^{m}(s)$ consisting on  functions $a(\cdot,\cdot)$ such that for every multi-index $\gamma$ 
\begin{equation}\label{symbol}
\|\partial_\xi^\gamma a(\cdot,\xi)\|_{C^s}\leq C_{\gamma,s}(1+|\xi|)^{m-|\gamma|}.
\end{equation}
\end{defi}

Therefore, if $a\in S^{m}(s)$ is a symbol, then
\begin{equation}\label{op}
\p(a)u(x)=\frac{1}{(2\pi)^d}\int_{\R^d} e^{ix\cdot \xi}a(x,\xi)\reallywidehat{u}(\xi)d\xi, \quad \mbox{for all }  x\in\mathbb{R}^{d}
\end{equation}
defines the associated pseudo-differential operator. If $u\in \mathcal{S}(\mathbb{R}^{d})$, then $\reallywidehat{u}\in \mathcal{S}(\mathbb{R}^{d}) $ and therefore $a(x,\xi)\reallywidehat{u}(\xi)\in \mathcal{S}(\mathbb{R}^{d})$ for every $x\in\mathbb{R}^{d}$.  As a result, the integral \eqref{op} exists and is well-defined. 

Moreover, using Fubini's theorem, we can calculate the formal adjoint of $\p(a)u(x)$
\begin{align*}
\left(\p(a)u, v\right)_{L^{2}(\mathbb{R}^{d})}&=\frac{1}{(2\pi)^d}\int \int e^{ix\cdot \xi}a(x,\xi)\reallywidehat{u}(\xi) d\xi \overline{v(x)} \ dx \\
&=\frac{1}{(2\pi)^d}\int \int e^{ix\cdot \xi}a(x,\xi)\overline{v(x)} \ dx\reallywidehat{u}(\xi) d\xi \\ 
&=\frac{1}{(2\pi)^d}\int  \reallywidehat{u}(\xi)  \overline{\int e^{-ix\cdot \xi} \overline{a(x,\xi)}v(x) \ dx} d\xi. 
\end{align*}
Since $v,\reallywidehat{u}\in\mathcal{S}(\mathbb{R}^{d})$, one can check that $e^{ix\cdot\xi}a(x,\xi)\reallywidehat{u}(\xi)v(x)\in L^{1}(\mathbb{R}^{d}\times\mathbb{R}^{d}).$ Using the fact that $\left(\reallywidehat{u}, v\right)_{L^{2}}=\left(u,\widehat{v}\right)_{L^{2}}$ we obtain that
\begin{equation}\label{op:adjoint}
\p(a)^{\star}v(x)=\frac{1}{(2\pi)^d}\int\int e^{i(x-y)\cdot \xi} \overline{a(y,\xi)}v(y) \ dy d\xi. 
\end{equation}
We refer the interested reader to the complete monograph \cite{Taylor-1981} for further details.

It will be also convenient when estimating the norms of pseudo-differential operators \eqref{op} to define the following semi-norms
\begin{defi}
   Given $s\in\mathbb{R}_+\setminus\N$ and $l\in \N$, we define the following family of semi-norms in the symbol class $S^m(s)$:

    $$\|a\|_{m,s,l}=\sup_{|\gamma|\leq l}\sup_{\xi\in\R^n}(1+|\xi|)^{|\gamma|-m}\|\partial_\xi^\gamma a(\cdot,\xi)\|_{C^s}.$$
\end{defi}

Next, we collect an important property regarding the boundednesss of a pseudo-differential operator acting on Besov spaces. The result can be found in \cite[Lemma 4.5]{Jurgen-87}. However, we include a more detailed and complete version of the proof in Appendix \ref{appendix1}.

\begin{teor}\label{boundedness:Besov}
    Let $a\in S^m(s)$, then $\p(a)$ defined in \eqref{op} extends to a bounded operator 
    \begin{equation}
        \p(a):B^{m-s}_{1,1}(\mathbb{R}^{d})\longrightarrow B^{-s}_{1,1}(\mathbb{R}^{d}).
    \end{equation}
    More precisely, we have that
        \begin{equation}
         \|\p(a)\|_{\mathcal{L}(B^{m-s}_{1,1},\,B^{-s}_{1,1})}\leq C\|a\|_{m,s,2|\gamma|}, \quad 
        \end{equation}
 for all multi-index $\gamma$ with $|\gamma|>d$  
\end{teor}
\begin{coro}
Let $a\in S^{m}(s)$, then $\p(a)^{\star}$ defined in \eqref{op:adjoint}
extends to a bounded operator
\begin{equation}
\p(a)^{\star}:B^{s}_{\infty,\infty}\longrightarrow B^{s-m}_{\infty,\infty}.
\end{equation}
Moreover, 
\[ \|\p(a)^{\star}\|_{\mathcal{L}(B^{s}_{\infty,\infty}, B^{s-m}_{\infty,\infty})}\leq C\|a\|_{m,s,2|\gamma|}.\]
\end{coro}
Therefore, recalling that $B^{s}_{\infty,\infty}(\R^d)\sim C^{s}(\R^d)$, we have shown that for symbols $a\in S^{m}(s)$
\[ \|\p(a)^{\star}\|_{\mathcal{L}(C^{s}(\R^{d}), C^{s-m}(\R^{d}))}\leq C\|a\|_{m,s,2|\gamma|}. \]

\subsection{Transferring properties to periodic pseudo-differential operators}\label{subsec:23}
It is worth to notice that the presented objects and results in the previous subsections are defined on the whole Euclidean space $\R^d$. In particular, we have shown that the adjoint of the pseudo-differential operator given by a symbol $a\in S^m(s)$ is bounded from $C^{s}(\R^d)$ to $C^{s-m}(\R^d)$. However, as discussed in Section \ref{section1:intro} and as stated in the main result of this work (cf. Theorem \ref{mainteor}), this is not directly well-suited to our geometrical setting. Indeed, we need to work in periodic domains, namely, in the periodic torus $\T^{d}$. For instance, we have to define pseudo-differential operators $\p(a)^{\star}u$ for $u\in C^s(\T^d)$. Moreover, the involved operators we will deal with are given in terms of formal expressions involving Fourier series. In this section, we will prove that these formal expressions can be rigorously identified with the adjoint of a pseudo-differential operator as long as the space of functions we are working with is sufficiently smooth.

In the following, let us show that pseudo-differential operators map periodic functions into periodic function as along as the symbol $a(x,\xi)$ is also periodic in the $x$ variable.

\begin{lema}\label{periodico}Let $s\notin \mathbb{N}$, and $v\in C^s$ be periodic. Assume that $a\in S^m(s)$ is periodic in the $x$ variable with the same period as $v$. Then $\p(a)^{\star}v$ is also periodic, and has the same period as $v$.
\end{lema}
\begin{proof}[Proof of Lemma \ref{periodico}]
We need to check that
$$\int_{\R^d} \p(a)^{\star}v(x+h)\cdot u(x) dx=\int_{\R^d} \p(a)^{\star}v(x)\cdot u(x)dx,$$
for $h$ the period of $v$ and $u\in\mathcal{S}(\R^n)$. First of all, since $\p(a)^{\star}v\in C^s(\R^d)\hookrightarrow \mathcal{S}'(\R^d)$, we find that
$$S_j\left( \p(a)^{\star}v\right)\longrightarrow \p(a)^{\star}v \text{ in }\mathcal{S}'.$$
Therefore, using the dual pairing formula for $B^{s}_{1,1}$ and $B^{-s}_{\infty,\infty}$ (cf. \cite[Remark 2.2.8]{Chemin-Danchin-Bahouri-2011}) given by
 \[v(u)=\sum_{j=-1}^\infty\,\sum_{|j-j'|\leq 1}\langle \Delta_{j'}v,\Delta_{j}u\rangle,\]
and the definition of adjoint operator yields
\begin{align*}
\left\langle \p(a)^{\star}v,u\right\rangle=\sum_{j=-1}^\infty \left\langle \Delta_j\left(\p(a)^{\star}v\right),u\right\rangle&=\sum_{j=-1}^\infty \sum_{|j'-j|\leq 1}\left\langle \Delta_j\left(\p(a)^{\star}v\right),\Delta _{j'}u\right\rangle \\
&=\sum_{j=-1}^\infty \langle \Delta_j v,\p(a)u\rangle.
\end{align*}
Finally since $\sum_{j=-1}^\infty \Delta_j v$ converges uniformly to $v$, we obtain that
\[ \left\langle \p(a)^{\star}v,u\right\rangle=\sum_{j=-1}^\infty \sum_{|j'-j|\leq 1}\left\langle \Delta_j v,\Delta _{j'}\left(\p(a)u\right)\right\rangle=\sum_{j=-1}^\infty \langle \Delta_j v,\p(a)u\rangle=\langle v,\p(a)u\rangle,\]
where we have used the fact that $\p(a)u$ decays faster that any polynomial. Consequently for $u\in\mathcal{S}(\mathbb{R}^{d})$, 
\begin{align*}\int_{\R^d} \p(a)^{\star}v(x+h)\cdot u(x) dx&=\int_{\R^d} v(x)\cdot\p(a)\left(u(x-h)\right)dx\\
&=\frac{1}{(2\pi)^d}\int_{\R^d} v(x) \int_{\R^d} a(x-h,\xi)\reallywidehat{u}(\xi)e^{i(x-h)\cdot \xi}d\xi dx\\
&=\int_{\R^d} v(x)\cdot \left(\p(a)u\right)(x-h)dx\\
&=\int_{\R^d} \p(a)^{\star}v(x)\cdot u(x)dx,
\end{align*}
as desired.
\end{proof}

In the following, let us present the main result of this subsection which relates the regularity of the adjoint of a pseudo-differential operator acting on periodic functions and its precise Fourier coefficient expression. The proof relies on a lemma which relates the regularity of a function with the decay of its Fourier coefficients using Wiener spaces. For the sake of completeness, we recall the definition of Wiener spaces and the aforementioned lemma.

\begin{defi} Denote by $A(\T^d)$ the space of continuous functions such that 
\begin{equation*}
  \|u\|_{A(\T^{d})}=\sum_{\xi\in\Z^d}|\reallywidehat{u}(\xi)|<\infty.  
\end{equation*}
The space $(A(\T^d),\|\cdot\|_{A(\T^d)})$ usually known as Wiener algebra is a Banach space. 
\end{defi}
The following result can be found in \cite[Section~3.3.3]{Graf-2014}:
\begin{lema}\label{decay}
    Suppose that $u\in C^{\lfloor d/2\rfloor}(\T^d)$, and that all partial derivatives of order $\lfloor \frac{d}{2}\rfloor$ of $u$ are of class $C^\gamma$ with $\frac{d}{2}-\lfloor \frac{d}{2}\rfloor<\gamma<1$. Then, $u\in A(\T^d)$, and 
\[ \|u\|_{A(\T^d)}\leq |\reallywidehat{u}(0)|+C(n,\gamma)\sup_{|\alpha|=\lfloor \frac{d}{2}\rfloor}[\partial^\alpha u]_{C^\gamma(\T^d)}\leq C(d,\gamma)\|u\|_{C^{\lfloor n/2\rfloor,\gamma}}.\]
\end{lema}
The main result of this subsection reads
\begin{teor}\label{solucion}
Let $k\geq1$, and $a\in S^{0}(k+\alpha)$ with $0<\alpha<1$, periodic on the first variable and $v\in C^{k,\alpha}(\T^2)$. Then, $\p(a)^{\star}v$ equals a $C^{k,\alpha}(\T^2)$ function with Fourier coefficients given by 
   \begin{equation}\label{adj:per:coef:four}
   \reallywidehat{\p(a)^{\star}v}(\xi)=\int_0^{2\pi}\int_0^{2\pi}v(r)\overline{a(r,\xi)}e^{-ix\cdot \xi} \,dr.
   \end{equation}  
\end{teor}

\begin{proof}[Proof of Theorem \ref{solucion}]
 To that purpose, take $u\in \mathcal{S}(\R^2)$. Then, by the same reasoning as in the proof of Lemma \ref{periodico}, we find that 
 \begin{align*}
     \left\langle\mathcal{F}\left[\p(a)^{\star}v\right],\,u \right\rangle=\left\langle v,\, \p(a)\reallywidehat{u}\right\rangle&=\f\int_{\R^2} \overline{v(r)}\left(\int_{\R^2} a(r,\zeta)\reallywidehat{\reallywidehat{u}}(\zeta)e^{ir\cdot\zeta}d\zeta\right) dr\\
     &=\lim_{m\rightarrow \infty}\int_{[-\pi m,\pi m]^2} \overline{v(r)}\left(\int_{\R^2} a(r,\zeta)u(-\zeta)e^{ir\cdot \zeta}d\zeta \right)dr\\
     &=\lim_{m\rightarrow \infty}\int_{\R^2} u(-\zeta)\overline{\left(\int_{[-\pi m,\pi m]^2}v(r) \overline{a(r,\zeta)}e^{-ir\cdot \zeta}dr\right) }d\zeta.
 \end{align*}
Here, $\mathcal{F}$ represents the Fourier transform in the sense of distributions.
 Due to Plancherel's theorem and the fact that $v(r)\overline{a(r,\xi)}$ is $2\pi$-periodic on the \textit{r} variable,
 \begin{align*}
     \int_{[-\pi m,\pi m]^2}v(r) \overline{a(r,\zeta)}e^{-ir\cdot \zeta}dr=&\sum_{\xi\in\Z^2}\f\left(\int_{\T^2}v(r)\overline{a(r,\zeta)}e^{-i \xi\cdot r}dr\right)\cdot \left(\int_{[-m\pi,m\pi]^2}e^{ir\cdot (\xi-\zeta)}dr\right).
 \end{align*}
Plugging it into the previous expression, 
 \begin{align*}
     \langle \p(a)^{\star}v,\, u\rangle &=\lim_{m\rightarrow \infty}\frac{1}{(2\pi)^2}\int_{\R^2}\sum_{\xi\in\Z^2}\left(\int_{[-\pi,\pi]^2}\overline{v(r)}a(r,\zeta)e^{i\xi\cdot r}dr\right)\left(\int_{[-\pi m,\pi m]^2}e^{i(\zeta-\xi)\cdot r}dr\right)u(-\zeta)d\zeta.
     \end{align*}
     We are now interested in taking the integral over $[-\pi m, \pi m]^2$ outside to be able to take the limit. We see that, for every $\zeta\in \R^2$,

     $$\overline{v(r)}a(r,\zeta)\in C^{k,\alpha}(\T^2)\,\text{ and } \,\|\overline{v}(\cdot)a(\cdot ,\zeta)\|_{C^{k,\alpha}(\T^2)}\leq C,$$
     with $C$ independent of $\zeta$, so by means of Lemma \ref{decay}, the Fourier coefficients
     \[c(\xi,\zeta)=\int \overline{v(r)} a(r,\zeta)e^{i\xi\cdot x} \ dx\in A(\mathbb{T}). \]
     Moreover, 
     \[ \norm{c(\xi,\zeta)}_{\ell^{1}_\xi}\leq C \norm{\overline{v(\cdot)}a(\cdot,\zeta)}_{k,\alpha}\leq C,\]
     with $C$ independent of $\zeta$. Therefore,
     \begin{equation*}
     \begin{split} \int_{\R^2}\sum_{\xi\in\Z^2}\int_{[-\pi m,\pi m]^2}&\left|\left(\int_{[-\pi,\pi]^2}\overline{v(r)}a(r,\zeta)e^{i\xi\cdot r}dr\right)e^{i(\zeta-\xi)\cdot x}u(\zeta)\right|\,dx  \ d\zeta\\
     &\leq C (2\pi m)^2 \int_{\R^2}|u(\zeta)|d\zeta <\infty.
     \end{split}
     \end{equation*}
Invoking Fubini's theorem, we can exchange the order of the $x$, the $\zeta$ and the $\xi$ integrals, where the latter is taken with respect to the counting measure. Hence,
\begin{equation}\label{limit:integral:exchange}
     \langle \p(a)^{\star}v,\, u\rangle=\lim_{m\longrightarrow \infty}\int_{[-\pi m,\pi m ]^2}\sum_{\xi\in\Z^2}\left(\int_{\R^2}\left(\int_{[-\pi,\pi]^2}v(r)\overline{a(r,\zeta)}e^{-i\xi\cdot r }dr\right)e^{i(\zeta-\xi)\cdot x }u(\zeta)d\zeta\right)dx.
 \end{equation}
 To conclude the proof we would like to check that the limit tends to the integral in the whole space, but it is not even clear whether such limit even exists. To show the existence of the limit, we have to provide a better decay for the integrand in the $x$ variable via an oscillatory integral type argument. More precisely, we have that
 \begin{align*}
     (1+|x|^2)&\sum_{\xi\in\Z^2}\left(\int_{\R^2}\left(\int_{[-\pi,\pi]^2}\overline{v(r)}a(r,\zeta)e^{i\xi\cdot r}dr\right)e^{i(-\zeta+\xi)\cdot x}u(\zeta)d\zeta\right)dx\\
     &=\sum_{\xi\in\Z^2}\left(\int_{\R^2}\left(\int_{[-\pi,\pi]^2}\overline{v(r)}a(r,\zeta)e^{i\xi\cdot r}dr\right)(1-\Delta_\zeta)e^{i(\zeta-\xi)\cdot x}u(\zeta)d\zeta\right)dx.
 \end{align*}
Using integration by parts combined with the fact that $u\in \mathcal{S}(\mathbb{R}^{2})$ (and therefore has very fast decay) we  get a combination of terms of the following type
 $$\sum_{\xi\in\Z^2}\left(\int_{\R^2}\left(\int_{[-\pi,\pi]^2}v(r)\partial^\alpha_\zeta\overline{a(r,\zeta)}e^{-i\xi\cdot r}dr\right)e^{i(\zeta-\xi)\cdot x}\partial^\beta_\zeta u(\zeta)d\zeta\right)dx.$$
Consequently, since $a\in S^0(1+\alpha)$,we can invoke Proposition \ref{decay} once again, to show that the resulting integral is finite. This fact, allows to take the limit in $m$ and also to exchange the $x$ integral and the sum in \eqref{limit:integral:exchange} to find that
\[ \langle \p(a)^{\star}v,\, u\rangle=\frac{1}{(2\pi)^2}\sum_{\xi\in\Z^2}\int_{\R^2}\left(\int_{\R^2}\left(\int_{[-\pi,\pi]^2}\overline{v(r)}a(r,\zeta)e^{i\xi\cdot r}dr\right)e^{i(\zeta-\xi)\cdot x}u(\zeta)d\zeta\right)dx. \]
Finally, notice that the function 
$$\zeta\mapsto u(-\zeta)\int_{[-\pi,\pi]^2}\overline{v(r)}a(r,\zeta)e^{i\xi\cdot r}dr$$
belongs to the Schwarz class and hence by the Fourier inversion formula, it follows that
$$\langle \p(a)^{\star}v,\, u\rangle = \sum_{\xi\in\Z^2}\left(\int_{[-\pi,\pi]^2}\overline{v(r)}a(r,\xi)e^{i\xi\cdot r}dr\right)u(-\xi).$$

Now, since that for any $v\in C(\T^2)$ and $u\in \mathcal{S}(\R^2)$, it holds that 

$$\langle v,\widehat{u}\rangle =\sum_{\xi\in\Z^2}\overline{\widehat{v}(\xi)}u(-\xi),$$
the result follows.
\end{proof}
\begin{obs}\label{observacionperiodica}
Theorem \ref{solucion} shows that if $a\in S^m(1+\alpha)$, we can define $\p(a)^{\star}v$ for $v\in C^{1,\alpha}(\T^2)$, and obtain formula \eqref{adj:per:coef:four} for its Fourier coefficients. One can notice that such expression only depends on the values of $a(x,\xi)$ for $\xi\in\Z^2$. This implies that $\p(a)^{\star}v=\p(b)^{\star}v$ for every other symbol $b\in S^{m}(1+\alpha)$ that coincides with $a(x,\xi)$ on $\xi\in\Z^2$.

In this article the situation in somehow reversed, i.e. we are given a certain operator $\textsf{T}$ acting on $C^{1,\alpha}(\T^2)$ that formally resembles the adjoint of a pseudodifferential operator with a symbol $\tilde{a}(x,\xi)$ that is \textit{a priori} only defined for $\xi\in\Z^2$. Thus, we can invoke Theorem \ref{solucion} if we extend the symbol $\widetilde{a}(x,\xi)$ to all $\R^2$ smoothly so that the extension $a(x,\xi)\in S^m(1+\alpha)$. Typically, there is one obvious extension of $\widetilde{a}(x,\xi)$ in such a way that the new symbol $a(x,\xi)$ is smooth outside the origin and satisfies 
\begin{equation}\label{constantes}
\|\partial_\xi^\gamma a(\cdot,\xi)\|_{C^{1,\alpha}}\leq C_{\gamma}|\xi|^{m-|\gamma|}.
\end{equation}
However, extending the symbol is such a way lacks the correct behavior at the origin. Nevertheless, taking a cut-off function $\varphi\in C^\infty_c(\R^2)$ such that $\phi\equiv 1$ on small neighborhood around the origin and $\phi\equiv 0$ outside a ball of radius $\frac{1}{2}$, the symbol
\[ S^{m}(1+\alpha)\ni b(x,\xi)=\varphi(0)\cdot \tilde{a}(x,0)+(1-\varphi(\xi))a(x,\xi)),\]
and extends $\tilde{a}(x,\xi)$. As a consequence, due to Theorem \ref{solucion}, the operator $\textsf{T}$ defines a bounded operator in $C^{1,\alpha}$ and its norm can be estimated by the constants $C_\alpha$ in \eqref{constantes}.
\end{obs}

\subsubsection{Classical Mikhlin-H\"ormander multipliers in $\R^{d}$ and $\T^{d}$}
A particular case of pseudodifferential operators are Mikhlin-H\"ormander multipliers. Mikhlin-H\"ormander multipliers of degree $m$ are ``pseudodifferential operators'' with symbols $a$ that only depend on $\xi$ and that satisfy a modified version of the decay estimates: 
$$|\partial^\gamma_\xi a(\xi)|\leq C_{\gamma}|\xi|^{m-|\gamma|}, \ \forall \gamma\in\N^{d}.$$

These can be defined for a large class of functions, and some multiplier theory give us boundedness in $L^2(\R^d)$, for instance. Note that, due to the Remark \ref{observacionperiodica}, Mihklin-H\"ormander multipliers are bounded from $C^{m+\alpha}(\T^d)$ to $C^{\alpha}(\T^d)$ and more in general  
$C^{k+m,\alpha}(\T^d)$ to $C^{k,\alpha}(\T^d)$ for $k,m\in\N$, cf. \cite{Taylor-1981}.

There are two prominent examples operators whose symbols are Mihklin-H\"ormander multipliers: the Hilbert transform $\mathcal{H}$ and the Riesz operator $\mathcal{R}_{j}$. The former is given by the principal value integral
\[\mathcal{H}f(x)=\mbox{p.v.} \frac{1}{\pi} \int_{\mathbb{R}} f(x-y) \  \frac{dy}{y},\]
and has the associated multiplier
\[ m(\xi)=-i\text{sgn}(\xi), \quad \xi\in\mathbb{R}.\]
The latter is also expressed by the principal value integral
\[ \mathcal{R}_{j}f(x)=\int_{\mathbb{R}^{d}} \frac{y_{j
}}{|y|^{d+1}}f(x-y) \  dy,\]
where $j=1,\ldots, d.$ The multiplier is given by
\[ m_{j}(\xi)=-i\frac{\xi_{j}}{|\xi|}. \]
We take as a convention that $m_{j}(0)=0$. 
The Riesz transform operator will appear throughout this work repeatedly. Moreover, as a consequence of the previous exposed arguments we infer that the Riesz transform is a bounded linear operator such that
\begin{equation}
\mathcal{R}_{j}:C^{k,\alpha} (\T^d)\to C^{k,\alpha} (\T^d).
\end{equation}
\subsection{H\"older estimates for the div-curl system and transport type problem}\label{subsec:24}
In this subsection, we present two auxiliary results regarding H\"older estimates for the two building blocks, the div-curl and transport problem, employed in order to apply the so called vorticity-transport method mentioned in Section \ref{section1:intro}. Although the two auxiliary results rely on well-known and classical ideas (cf. \cite{Giaquinta-2012, Gilbarg-Trudinger-2001}), we prefer to include them in this subsection to make the article as self-contained as possible.

\subsubsection{H\"older estimates for the div-curl system.} The first auxiliary result we present deals with the well-posedness and regularity of solutions to the div-curl system. More precisely, we show the following result:

\begin{prop}\label{wellposedness}
    Let $\Omega=\T^2\times [0,L]$, $j\in C^{1,\alpha}(\Omega,\R^{3})$,  $f\in C^{2,\alpha}(\partial\Omega)$, and $J_1$, $J_2\in \R$. Assume that $j$ is divergence free, and that that $f$ and $j_3$ satisfy the following compatibility condition
\begin{equation}\label{compati:divcurl:prop}
\int_{\partial\Omega_{-}}j_3=0\quad \text{and}\quad \int_{\partial\Omega_{-}}f=\int_{\partial\Omega_{+}}f.
\end{equation}

    Then, there exists a unique solution $W\in C^{2,\alpha}(\Omega,\mathbb{R}^{3})$ for the div-curl problem 

    \begin{equation}\label{divcurlotravez}
        \left\{\begin{array}{rl}
           \nabla\times W=j  & \text{in }\Omega, \\
            \nabla\cdot W=0 & \text{in }\Omega,\\
            W\cdot n=f& \text{on }\partial\Omega,\\
        \end{array}\right.
    \end{equation}
    with $\int_{\{x=0\}}Wd\vec{S}=J_1, \quad \int_{\{y=0\}}Wd\vec{S}=J_2$. Furthermore, there exists a constant $C>0$ such that 
\begin{equation}\label{estimate:div:curl}
\|W\|_{C^{2,\alpha}(\Omega)}\leq C\left( \|j\|_{C^{1,\alpha}(\Omega)}+\|f\|_{C^{2,\alpha}(\partial\Omega)}+|J_1|+|J_2|\right).
\end{equation}
\end{prop}

\begin{proof}[Proof of Proposition \ref{wellposedness}]
The proof follows the lines in \cite[Proposition 3.11]{Alo-Velaz-2021}. However, we need to modify the argument accordingly to the three-dimensional setting. To solve the system \eqref{divcurlotravez} we first examine two complementary auxiliary problems, namely
\begin{equation}\label{potencial:1}
    \left\{\begin{array}{rl}
        \Delta Z=-j & \text{ in } \Omega, \\
        Z_1=0 & \text{ on }\partial \Omega, \\
        Z_2=0 & \text{ on }\partial \Omega,  \\
        \partial_3 Z_3=0 & \text{ on }\partial \Omega,
    \end{array}\right.
   \end{equation} 
  and
  \begin{equation}\label{potencial:2}
    \left\{\begin{array}{rl}
       \Delta \phi=0  & \text{ in } \Omega, \\
       \partial_\nu \phi=f &   \text{ on }\partial \Omega.
    \end{array}\right.
\end{equation}
Here, $Z=(Z_{1},Z_{2},Z_{3})$ is a vector, $\phi$ is a scalar function and $\nu$ is the outward unit normal vector to $\partial\Omega$ and $\partial_{\nu} \phi=\nabla \phi\cdot \nu$. Therefore, since $j\in C^{1,\alpha}(\Omega,\mathbb{R}^{3})$, the components 
$Z_{1},Z_{2}$ are the unique solutions to the classical Dirichlet problem. Moreover, by classical H\"older estimates it is well-known that
\begin{equation}\label{est:z1z2:dirichlet}
\|Z_1\|_{C^{3,\alpha}(\Omega)}\leq C\|j\|_{C^{1,\alpha}(\Omega)}, \quad \|Z_2\|_{C^{3,\alpha}(\Omega)}\leq C\|j\|_{C^{1,\alpha}(\Omega)}.
\end{equation}
On the other hand, due to Schauder theory for the Neumann problem and the compatibility condition \eqref{compati:divcurl:prop} there exists a unique zero mean solution $\phi$ to \eqref{potencial:2}. Moreover, we have that
\begin{equation}\label{est:phi:neumann}
\|\phi\|_{C^{3,\alpha}(\Omega)}\leq C\|f\|_{C^{2,\alpha}(\partial\Omega)}.
\end{equation}
We are left with the existence and uniqueness of the component $Z_{3}$ which solves a Neumann type problem. Therefore, to ensure the existence and uniqueness we just need to show that
\[ \int_{\Omega} j_{3} =0.\]
Indeed, we find that
\[\frac{d}{dz}\int_0^{2\pi}\int_0^{2\pi}j_3\,dxdy=-\int_0^{2\pi}\int_0^{2\pi} \partial_1 j_1 dxdy-\int_0^{2\pi}\int_0^{2\pi} \partial_2 j_2\, dxdy=0,\]
where we have used the divergence free condition on $j$ and the periodicity of $j$ in the $x,y$ variables. Thus, using the compatibility condition  \eqref{compati:divcurl:prop} on $\partial\Omega_{-}$ we conclude the assertion. Consequently, a unique mean zero $Z_3$ solution exists. Moreover, similarly as before, it satisfies 
\begin{equation}\label{est:z3:neumann}
\|Z_3\|_{C^{3,\alpha}(\Omega)}\leq C \|j\|_{C^{1,\alpha}(\Omega)}.
\end{equation}
Combining \eqref{est:z1z2:dirichlet} and \eqref{est:z3:neumann} we infer that
\begin{equation}\label{est:z3}
\|Z\|_{C^{3,\alpha}(\Omega)}\leq C \|j\|_{C^{1,\alpha}(\Omega)}.
\end{equation}
To conclude we define $W$ as 
\[ W=\nabla\times Z+\nabla\phi+A_1 \vec{e}_1+A_2\vec{e}_2,\]
where $A_1$ and $A_2$ are given by
\begin{align}
A_{1}=\frac{J_{1}}{2\pi L}-\int_{0}^L\int_0^{2\pi}\nabla\times Z(0,y,z)\cdot e_1 dydz-\int_0^L\int_0^{2\pi}\nabla\phi (0,y,z)\cdot e_1 dydz ,\\
A_{2}=\frac{J_{2}}{2\pi L}-\int_{0}^L\int_0^{2\pi}\nabla\times Z(x,0,z)\cdot e_2 dxdz -\int_0^L\int_0^{2\pi}\nabla\phi (x,0,z)\cdot e_2 dxdz  . 
\end{align}
Therefore, it can be readily check that 
\[ \nabla\times W= \nabla\times\nabla\times Z=\nabla (\nabla\cdot Z)-\Delta Z=j \quad \mbox{in } \Omega. \]
Notice that $\nabla \cdot Z=0$, since $Z\in C^{3,\alpha}(\Omega,\R^{3})$ solves \eqref{potencial:1} and by hypothesis $\nabla\cdot j=0.$ Similarly, we have that $\nabla\cdot W=\Delta \phi=0$  in  $\Omega$ and $W\cdot n =(\nabla\times Z)\cdot n+\partial_{n} \phi=f$ on  $\partial\Omega.$ Moreover, by construction of the constants $A_{1},A_{2}$ the flux condition on $W$ are also satisfied. Furthermore, by means of \eqref{est:phi:neumann} and \eqref{est:z3} we find that
\begin{equation}\label{est:W}
\|W\|_{C^{2,\alpha}(\Omega)}\leq C\left( \|j\|_{C^{1,\alpha}(\Omega)}+\|f\|_{C^{2,\alpha}(\partial\Omega)}+|J_1|+|J_2|\right).
\end{equation}
To show the uniqueness of solutions, assume that $W\in C^{2,\alpha}(\Omega,\mathbb{R}^{3})$ solution to \eqref{divcurlotravez} with $f=j=J_{1}=J_{2}=0$. Then, we have that
\[\nabla\times \left(\nabla \times W\right)=-\Delta W=0 \mbox{ in } \Omega.\]
Since $W_{3}=0$ on $\partial\Omega$, we conclude that $W_{3}=0$ in $\Omega$. Hence,
\[\nabla\times W=(-\partial_3 W_2,\partial_3 W_{1}, \partial_1 W_2-\partial_2 W_1)=0,\]
i.e. $W_i=W_{i}(x,y)$ for $i=1,2$ are harmonic functions in $\T^2$. Thus, 
\[W=(A_{1},A_{2},0) \mbox { in } \Omega, \]
for $A_{1},A_{2}\in\mathbb{R}$. However, since $J_{1}=J_{2}=0$ we obtain that $W\equiv 0$ in  $\Omega$. Then uniqueness follows by linearity.
\end{proof}

\subsubsection{H\"older estimates for the transport problem}
 The second auxiliary result focuses on H\"older estimates for a transport problem. We present first an elementary result regarding the uniqueness and regularity for a system of ODEs, using the method of characteristics. We just sketch the proof, since it is a straightforward adaption to the three dimensional setting of the result in \cite[Proposition 3.7]{Alo-Velaz-2021}.

 \begin{prop}\label{prop:regularidad:odes} Let $v\in C^{2,\alpha}(\Omega;\R^2)$. Then, there exists a unique solution $\Psi_z\in C^{2,\alpha}(\Omega)$ solving the system of ODEs 
\begin{equation}\label{ode}
    \left\{\begin{array}{ll}
        \frac{\partial\Psi_z}{\partial z}= v(\Psi_z,z), \quad z\in[0,L]  \\
         \Psi_{0}=(x,y), \quad (x,y)\in \T^2
    \end{array}
    \right. .
\end{equation}
Moreover, for every $z$, $\Psi_z$ describes a $C^{2,\alpha}$ embedding, so that there exists a constant $C=C\left(\|v\|_{2,\alpha},L\right)$ such that 
\begin{equation}\label{reg:bound:ode}
\|\Psi_z\|_{2,\alpha}\leq C\qquad \|\Psi^{-1}_z\|_{2,\alpha}\leq C.
\end{equation}
\end{prop}
In order to show Proposition \ref{prop:regularidad:odes} we will repeatedly use the following calculus result
\begin{lema}\label{calculus:composition}
    Let $u:\Omega\rightarrow\Omega$ be a Lipschitz function, and $f:\Omega\rightarrow \R^d$ be a $C^\alpha(\Omega,\R^d)$ function. Then, $f\circ u\in C^{\alpha}(\Omega;\R^d)$. 
\end{lema}








\begin{proof}[Proof of Proposition \ref{prop:regularidad:odes}]
The existence and uniqueness of solutions to \eqref{ode} follows from the classical Picard-Lindel\"of theorem. Moreover, since $v$ is a bounded, periodic and globally Lipschitz, the solution exists for all $z\in [0,L]$. The the function $\Psi_z:\T^2\rightarrow \Omega$ defines a topological embedding for each $z$ (cf. \cite[Theorem 7.1]{Coddington-Levison-1955}). In only remains to check the regularity bounds \eqref{reg:bound:ode}. Direct differentiation shows that
    \begin{align*}
    \frac{\partial\Psi_z}{\partial z}&= v(\Psi_z,z),\\
    \frac{\partial^2\Psi_z}{\partial z^2}&=\frac{\partial v(\Psi_z,z)}{\partial X}v_1(\Psi_z,z)+\frac{\partial v(\Psi_z,z)}{\partial Y}v_2(\Psi_z,z)+\frac{\partial v(\Psi_z,z)}{\partial z}.
    \end{align*}
Therefore, invoking Lemma \ref{calculus:composition} we find that
    \begin{align*}
        \|\partial_z \Psi_z\|_{\infty} & \leq \|v\|_\infty, \quad 
        \|\partial^2_z\Psi_z\|_{\infty}\leq C\|D v\|_{\infty}(\|v\|_\infty+1), \\
        [\partial^2_z\Psi_z]_\alpha & \leq C([D v]_{\alpha}\|v\|_\infty +\|D v\|_\infty [v]_\alpha+[D v]_\alpha) (\|D\Psi_z\|_{\infty}^\alpha+1).
    \end{align*}
    Since $[D v]_\alpha\leq C\|v\|_{C^2}$, we get the desired estimates for the $z$ derivatives of $\Psi_z$. To bound the $x$ derivatives, standard results of differentiability with respect to parameters of ordinary differential equations yield
   \[\left|\frac{\partial\Psi_z}{\partial x}\right|\leq   \left(1+\|v\|_{C^1}Le^{L\|v\|_{C^1}}\right), \]
and hence we obtain bounds for the $C^{1}$ norm of $\Psi_{z}.$  To estimate the $C^{\alpha
}$  H\"older semi-norm in the $x$ variable and higher order  $C^{2,\alpha}$ estimates we proceed in similar fashion. 
In order to show $C^{2,\alpha}$ bounds for the inverse $\Psi_{z}^{-1}$, we find the following relation between 
\[ \Psi_z=(X(x,y,z),Y(x,y,z)), \quad \Psi^{-1}_z=(X^{-1}(x,y,z),Y^{-1}(x,y,z)),\]
namely,
\begin{align*}\left(\begin{array}{cc}
    \frac{\partial X^{-1}}{\partial x} & \frac{\partial X^{-1}}{\partial y} \\\\
    \frac{\partial Y^{-1}}{\partial x} & \frac{\partial Y^{-1}}{\partial y}
\end{array}\right)(x,y,z) 
&=\exp\left(-\int_0^z D\, v(x,y,s)ds\right)\left(\begin{array}{cc}
    \frac{\partial Y}{\partial y} & -\frac{\partial X}{\partial y} \\\\
    -\frac{\partial Y}{\partial x} & \frac{\partial X}{\partial x}
\end{array}\right)(\Psi_z^{-1}(x,y),z).
\end{align*}
Using the previous relation we can readily check that $\Psi_z^{-1}$ is Lipschitz. Moreover, using the representation one can compute H\"older semi-norms and higher order derivatives in terms of $\Psi_z$, proving the desired bounds \eqref{reg:bound:ode} and concluding the proof.
\end{proof}

In the following we will derive H\"older estimates for the hyperbolic transport problem
    \begin{equation}\label{transporte:1}
    \left\{\begin{array}{ll}
        \partial_3 j +(b\cdot \nabla)j=(j\cdot\nabla) b & \text{in } \Omega, \\
         j=j_0& \text{on } \partial\Omega_{-}.
    \end{array}\right.
    \end{equation}
   
More precisely, we show the following result:
\begin{prop}\label{existenciatransporte}
    There exists $M_0$ small enough so that for every $b\in C^{2,\alpha}(\Omega;\R^3)$ with $\|b\|_{C^{2,\alpha}}\leq M<M_0$ and $j_0\in C^{1,\alpha}(\T^2;\R^3)$, the system \eqref{transporte:1} admits a unique solution $j\in C^{1,\alpha}(\Omega;\R^3)$. Moreover, there exists a constant $C=C(\alpha,L)$ such that 
    \begin{equation}\label{regularidad1}
        \|j\|_{C^{1,\alpha}}\leq C\|j_0\|_{C^{1,\alpha}}.
    \end{equation}
Furthermore, let $j^1, j^2\in C^{1,\alpha}(\Omega;\R^3)$ be two solutions to \eqref{transporte:1} with $b$ given by $b^{1}, b^{2}$ respectively. Then, we have the lower order estimate
    \begin{equation}\label{regularidad2}
        \|j^1-j^2\|_{C^{\alpha}}\leq C\left[\|j_0^1-j_0^2\|_{C^{\alpha}}+\|j_0^1\|_{C^{1,\alpha}}\|b^1-b^2\|_{C^{\alpha}}\right].
    \end{equation}
\end{prop}
\begin{obs}
If we compare Proposition \ref{existenciatransporte} with its two-dimensional analogue shown in \cite[\textit{Proposition 3.8}]{Alo-Velaz-2021}, \cite[\textit{Proposition 12}]{Alo-Velaz-2022}, we notice that in this case we have an extra stretching term in the hyperbolic equation \eqref{transporte:1}, namely, $(j\cdot \nabla )b$. Nevertheless, from the regularity point of view, this term has better regularity than the transport term, and thus we can show identical bounds \eqref{regularidad1}-\eqref{regularidad2} as for its two-dimensional analogue. The idea of the proof also follow closely the one in \cite[\textit{Proposition 3.8}]{Alo-Velaz-2021}, \cite[\textit{Proposition 12}]{Alo-Velaz-2022}.
\end{obs}
\begin{proof}[Proof of Proposition \ref{existenciatransporte}]
It is well-known that the first order hyperbolic transport problem \eqref{transporte:1} can be solved via method of characteristics. More precisely, let us introduce the flow map defined by (characteristic lines of \eqref{transporte:1})
    \begin{equation}\label{car}
      \left\{\begin{array}{ll}
        \frac{\partial X}{\partial z}&=\frac{b_1}{1+b_3}, \\  \frac{\partial Y}{\partial z}&=\frac{b_2}{1+b_3},
      \end{array}\right.  
    \end{equation}
  with initial data  $X(x,y,0)=x,\, Y(x,y,0)=y$. Here, we denoted by $b_{\ell}$ for $\ell=1,2,3$, the components of the vector field $b$.
System \eqref{car} yields formally two maps 
\begin{equation}\label{flows} 
\Psi_z=(X,Y), \quad \Phi=(\Psi_z,z).
\end{equation}
Moreover, note that if $M_{0}<1$, the right hand side in \eqref{car} 
\[ \frac{b_1}{1+b_3} \in C^{2,\alpha}(\Omega), \quad \frac{b_2}{1+b_3}\in C^{2,\alpha}(\Omega), \] 
and hence, Proposition \ref{prop:regularidad:odes} shows that $\Psi_z,\Phi\in C^{2,\alpha}(\Omega)$ with $C^{2,\alpha}(\Omega)$ inverse. Therefore, the solution $j$ to system \eqref{transporte:1} is explicitly given by
\begin{equation}\label{representation:formula}
j(x,y,z)=w\circ \Phi^{-1}(x,y,z),
\end{equation}
where $w$ satisfies the differential equation
\begin{equation}\label{regularidad4}\left\{\begin{array}{ll}
        \frac{\partial w}{\partial z}=A(b(\Phi))\cdot w & 0\leq z\leq L \\
        w=j_0 & z=0
    \end{array}\right.,
    \end{equation}
    where the matrix $A(b)$ has entries $A_{ij}=\frac{\partial_j b_i}{1+b_3}$ with $i,j=1,2,3.$ We know that $\Phi$ is Lipschitz, and thus (again, since
    $M_0<1$) we can deduce that the function $A(b(\Phi))\in C^{1,\alpha}(\R^{3},M_{3}(\R))$. Therefore, the system \eqref{regularidad4} has
    a unique solution $w$ for every $(x,y)\in \T^2$. As a consequence, the representation formula \eqref{representation:formula} gives us the unique solution for the hyperbolic transport equation  \eqref{transporte:1}.
    
    In order to prove estimates \eqref{regularidad1} and \eqref{regularidad2} we make use of the $C^{2,\alpha}$ regularity of $\Phi$ and the fact that $w$ satisfies a linear equation. Integrating \eqref{regularidad4} and using Gr\"onwall's inequality we find that 
    \[ \|w\|_{\infty}\leq \|j_0\|_\infty\left(1+\frac{\|b\|_{C^1}}{1-M_0}L\exp\left(\frac{L\|b\|_{C^1}}{1-M_0}\right)\right),\]
and hence by means of the representation formula \eqref{representation:formula} we deduce that
\begin{equation*} 
\|j\|_\infty\leq C\|j_0\|_\infty. 
\end{equation*}
To derive higher order estimates, we make use of the classical chain rule formula
\begin{equation}\label{chain:rule:high}Dj=\left(Dw\circ \Phi^{-1}\right)\cdot D\Phi^{-1}.
\end{equation}
One the one hand, recalling Proposition \ref{prop:regularidad:odes} we infer that $D\Phi^{-1}\in C^{1,\alpha}(\Omega)$, and thus it is Lipschitz. On the other hand, direct use of the representation formulas \eqref{representation:formula} and \eqref{regularidad4} shows that
\begin{equation*}
\left\|\frac{\partial w}{\partial z}\right\|_\infty \leq C\|b\|_{C^1}\frac{1}{1-M_0}\|w\|_\infty\leq C\|j_0\|_\infty.  
\end{equation*}
Similarly, differentiating w.r.t the $x$ variable in \eqref{regularidad4} we obtain
\[ \frac{\partial}{\partial z}\frac{\partial w_i}{\partial x}=\frac{\partial A_{ij}}{\partial X}\frac{\partial X}{\partial x}w_j+\frac{\partial A_{ij}}{\partial Y}\frac{\partial Y}{\partial x}w_j+A_{ij}\frac{\partial w_j}{\partial x}. \]
Integrating for $0\leq z\leq L$
and invoking once again Gr\"onwalls inequality we find that
\begin{equation*}
    \left\|\frac{\partial w}{\partial x}\right\|_{\infty}\leq \left(\|j\|_{C^1}+C\|b\|_{C^2}\|\Phi\|_{C^1}\|w\|_\infty\right)\left(1+\frac{1}{1-M_0}\|b\|_{C^1}Le^{\frac{1}{1-M_0}\|b\|_{C^1}}\right)\leq C\|j_0\|_{C^1}.
\end{equation*}
Mimicking the previous estimates w.r.t the $y$ variable and combining them together shows that $\|w\|_{C^1}\leq C \|j_0\|_{C^1}$. Hence, using formula \eqref{chain:rule:high} we conclude that
\begin{equation*}
\|j\|_{C^1}\leq C\|j_0\|_{C^1}.
\end{equation*}
Hence, since $\|D\Phi^{-1}\|_{\alpha}\leq C$ , estimate \eqref{regularidad1} follows if we prove that $[Dw]_\alpha\leq C\|j_0\|_{C^{1,\alpha}}$. To that purpose, denoting by $r_1=(x_1,y_1),\, r_2=(x_2,y_2)$ and differentiating \eqref{regularidad4} we have that

    \begin{align*}
        \frac{\partial}{\partial z}\left(\frac{\partial w_i}{\partial x}(r_1,z)-\frac{\partial w_i}{\partial x}(r_2,z)\right)&\leq \left.\sum_j \nabla A_{ij}\frac{\partial\Psi_s}{\partial x}w_j\right|_{(r_1,z)}-\left.\sum_j \nabla A_{ij}\frac{\partial\Psi_s}{\partial x}w_j\right|_{(r_2,z)}\\
        &+\left.\sum_jA_{ij}\frac{\partial w_i}{\partial x}\right|_{(r_1,z)}-\left.\sum_jA_{ij}\frac{\partial w_i}{\partial x}\right|_{(r_2,z)}.\\
    \end{align*}
Integrating between $0$ and $z$ and suitably pairing the terms, we obtain
\begin{align*}
    \left|\frac{\partial w_i}{\partial x}(r_{1},z)-\frac{\partial w_i}{\partial x}(r_{2},z)\right|&\leq \left|\frac{\partial j_0}{\partial x}(r_{1})-\frac{\partial j_0}{\partial x}(r_{2})\right|\\ 
    &+\sum_j\int_0^z\left|\nabla A_{ij}(r_{1},s)-\nabla A_{ij}(r_{2},s)\right|\left|\frac{\partial \Psi_s}{\partial x}(r_{1})\right|\left|w_j(r_{1},s)\right|ds\\
    &+\sum_j\int_0^z\left|\nabla A_{ij}(r_{2},s)\right|\left|\frac{\partial \Psi_s}{\partial x}(r_{1})-\frac{\partial \Psi_s}{\partial x}(r_{2})\right|\left|w_j(r_{1},s)\right|ds\\
    &+\sum_j\int_0^z\left|\nabla A_{ij}(r_{2},s)\right|\left|\frac{\partial \Psi_s}{\partial x}(r_{2})\right|\left|w_j(r_{1},s)-w_j(r_{2},s)\right|ds\\
    &+\sum_j\int_0^z\left|A_{ij}(r_{1},s)-A_{ij}(r_{2},s)\right|\left|\frac{\partial w_j}{\partial x}(r_{1},s)\right|ds\\
    &+\sum_j\int_0^z\left|A_{ij}(r_{2},s)\right|\left|\frac{\partial w_j}{\partial x}(r_{1},s)-\frac{\partial w_j}{\partial x}(r_{},s)\right|ds\\
    &\leq C|r_{1}-r_{2}|^\alpha \|j_0\|_{C^{1,\alpha}}+C\|b\|_{C^{1,\alpha}}\int_0^z\left|\frac{\partial w}{\partial x}(r_{1},s)-\frac{\partial w}{\partial x}(r_{2},s)\right|ds.
\end{align*}
Here, we have used the fact that $b\in C^{2,\alpha}(\Omega;\R^3)$ with $\|b\|_{C^{2,\alpha}}\leq M_0<1$ which implies  $A(b(\Phi))\in C^{1,\alpha}(\R^{3},M_{3}(\R))$, the fact that $\Psi_s\in C^{2,\alpha}(\Omega)$ and that $w\in C^{1}(\Omega;\R^3)$. Therefore, invoking Gr\"onwalls inequality yields
\[ \left|\frac{\partial w}{\partial x}(r_{1},z)-\frac{\partial w}{\partial x}(r_{2},z)\right|\leq  C\|j_0\|_{C^{1,\alpha}}|r_{1}-r_{2}|^\alpha.\]
The same computations can be carried out to show that 
\[ \left|\frac{\partial w}{\partial y}(r_{1},z)-\frac{\partial w}{\partial y}(r_{2},z)\right|\leq  C\|j_0\|_{C^{1,\alpha}}|r_{1}-r_{2}|^\alpha.\]
and conclude bound \eqref{regularidad1}. It should be stressed that     the constants we obtain depend on $\alpha$, and on $L$ and $\|b\|_{C^{2,\alpha}}$ via the estimates we get from Gr\"onwall's inequality. However, this $C$ is an increasing function of $\|b\|_{C^{2,\alpha}}$, which is bounded from above by $M_0$. Thus, we can get a $C=C(\alpha,L)$ as claimed  in the statement. 
   
To show the difference estimate \eqref{regularidad2}, we have to repeat the previous ideas. For the rest of the proof, we will use the notation $\widetilde{(\cdot)}=(\cdot)^{1}-(\cdot)^{2}$ to denote the difference of two quantities. For $i=1,2$, denote by $w_{i}$ the solution of \eqref{regularidad4} associated to the magnetic field $b^i$ with initial condition $j_0^i$. Analogously, denote by $\Phi^i$ to the change of variables arising from the characteristic equation with magnetic field $b^i$. Then, by using the explicit formula $j^i=w^i\circ(\Phi^i)^{-1}$, we have that
\begin{equation}\label{regularidad5} \widetilde{j}=\widetilde{w}\circ (\Phi^1)^{-1}+w^2\circ(\Phi^1)^{-1}-w^2\circ(\Phi^2)^{-1}.
\end{equation}
The last two terms on the right hand side in \eqref{regularidad5} are bounded by
\[ \left|w^2\circ(\Phi^1)^{-1}(x,y,z)-w^2\circ(\Phi^2)^{-1}(x,y,z)\right|\leq \|w\|_{C^1}\left|\widetilde{\Psi_{z}}\right|\leq \|j_{0}\|_{C^{1,\alpha}}\left|\widetilde{\Psi_{z}}\right|. \]
Moreover, $\widetilde{\Psi_{z}}$ satisfies the ODE
\begin{equation*}\label{eqs}
      \left\{\begin{array}{ll}
        \frac{\partial \widetilde{X}}{\partial z}=\frac{\widetilde{b_{x}}b_3^2-b_1^2\widetilde{b_{3}}}{(1+b^1_z)(1+b_3^2)}, \\ \\
        \frac{\partial \widetilde{Y}}{\partial z}=\frac{\widetilde{b_{y}}b_3^2-b_2^2\widetilde{b_{3}}}{(1+b^1_z)(1+b_3^2)}, 
      \end{array}\right.  
    \end{equation*}
with initial data $\widetilde{X}(x,y,0)=0, \widetilde{Y}(x,y,0)=0$. Integrating between $0$ and $z$, we readily see that 
\[\left|\widetilde{\Psi_z}(x,y)\right|\leq \frac{2M_0}{1-M_0}\|\widetilde{b}\|_\infty.\]
To bound the first on the right hand side in \eqref{regularidad5} we first notice that 
\begin{equation*} 
        \left\{\begin{array}{ll}
        \frac{\partial \widetilde{w}}{\partial z}=A_{b^2}(\Phi^2)\cdot  \widetilde{w}+w^1\left( A_{b^1}(\Phi^1)-A_{b^2}(\Phi^2)\right) & z>0, \\
         \widetilde{w}= \widetilde{j_{0}} & z=0. 
    \end{array}\right.
    \end{equation*}
Next, since 
\[ \|A(b^1(\Phi^1))-A(b^2(\Phi^2))\|_\infty\leq \|A(b^1(\Phi^1))\|_{\infty}\|\widetilde{\Phi}\|_\infty+\|A\widetilde{b}\|_\infty\leq C\|\widetilde{b}\|_\infty,\]
we conclude after integrating in $z$ that
\[\left|\widetilde{w}\right|\leq C\left(\|\widetilde{j_{0}}\|_\infty+\|j_0\|_\infty\|\widetilde{b}\|_\infty\right)\int_0^z|\widetilde{w}|ds. \]
An application of Gr\"onwalls inequality yields the desired estimate. To show the complete $C^{1,\alpha}$ norm estimate one can argue analogously.  
\end{proof}

For further reference, we include here the following result, which is a corollary of Proposition \ref{existenciatransporte}. This will be used later in the proof of Theorem \ref{mainteor}.

\begin{coro}\label{direct:coro:proptransport}
 Consider $M_0<1$, $b\in C^{2,\alpha}(\Omega;\R^3)$ and $j_0\in C^{1,\alpha}(\T^2,\R^3)$ as in Proposition \ref{existenciatransporte}. Let $j\in C^{1,\alpha}(\Omega;\R^3)$ be the unique solution of \eqref{transporte:1}, and $\bar{\jmath}\in C^{1,\alpha}(\Omega;\R^3) $ be the unique solution of 
    \begin{equation}\label{perturbacion}
    \left\{\begin{array}{ll}
        \partial_3 \bar{\jmath}+\left(b\cdot \nabla\right)\bar{\jmath}=0 & \text{in } \Omega, \\
        \bar{\jmath}=j_0 & \text{on } \partial\Omega_{-}.
    \end{array}\right. 
    \end{equation}

    Then, there exists a constant $C=C(\alpha,L)$ such that $\delta j=j-\bar{\jmath}$ satifies

    \begin{equation}\label{perturbacion1}
        \|\delta j\,\|_{C^{1,\alpha}}\leq 
    C\|b\|_{C^{2,\alpha}}\|j_0\|_{C^{1,\alpha}}.
    \end{equation}

    Moreover, if we have two magnetic fields $b^1$ and $b^2$, and two initial conditions $j_0^1$, and $j_0^2$, we can estimate their difference: 

    \begin{equation}\label{perturbation2}
    \|\delta j^1-\delta j^2\|_{C^{\alpha}}\leq C\|b^1-b^2\|_{C^{1,\alpha}}\left(\|j_0^1-j_0^2\|_{C^{\alpha}}+\|\delta j^1\|_{C^{1,\alpha}}\|b^1-b^2\|_{C^{\alpha}}\right).
    \end{equation}
\end{coro}

\begin{proof}
    We know that the solution for \eqref{perturbacion} is given by $\bar{\jmath}=j_0\circ \Phi^{-1}$, where $\Phi$ is given by \eqref{flows}. We need to estimate the difference $\delta j=(w-j_0)\circ\Phi^{-1}$ where $w-j_0$ satisfies the following problem
\[\left\{\begin{array}{ll}
        \frac{\partial (w-j_0)}{\partial z}=A_b\cdot w=A_b\cdot (w-j_0)+A_b\cdot j_0 & z>0, \\
        w-j_0=0 & z=0.
        \end{array}\right. 
        \]

The rest of the proof follows after applying the same ideas as in the Proposition above, with the main difference that now the initial condition is zero. This allows us to obtain estimates that depend directly on $\|b\|_{C^{2,\alpha}}$. For example, for the $L^\infty$ norm of $w-j_0$ we can integrate the ODE above, and after using Gr\"onwall's inequality,
\begin{equation*}
\|w-j_0\|_\infty \leq C\|b\|_{C^1}\|j_0\|_\infty.
\end{equation*}
The $C^{1,\alpha}$ bounds to show \eqref{perturbacion1} follow similarly.
In order to show estimate \eqref{perturbacion1} for the difference, one can mimic the same ideas done in Proposition \ref{existenciatransporte} to show \eqref{regularidad2}.
\end{proof}
\subsubsection{Transport divergence condition along characteristics}
To conclude this subsection, we provide an important result regarding the transport nature of the divergence free condition along characteristics. A similar result was already provided in Alber \cite{Alber-1992}. However, in \cite{Alber-1992} the magnetic field $B$ and hence the current $j$ (being more precise, the fluid velocity field $v$ and the vorticity $\omega$) is one derivative more regular which allows to understand the operators and the equation in the pointwise sense. We need to modify and adapt the argument to handle our regularity hypothesis.
\begin{prop}\label{transport:div:prop}
    Let $b\in C^{2,\alpha}(\Omega)$ be a divergence free vector field that satisfies the hypothesis in Proposition \ref{existenciatransporte} and $B=(0,0,1)+b$. Assume that $j_0\in C^{1,\alpha}(\partial\Omega)$ and denote the unique solution to \eqref{transporte:1} by $j$. Assume further that $\text{div}\, j|_{\partial\Omega_-}=0$. Then, $\text{div}\,j=0$ in $\Omega$.
\end{prop}
\begin{proof}
It is sufficient to show that 
 \begin{equation}\label{diver}
     B\cdot \nabla \text{div} j=0, \ \mbox{ in } \Omega,
     \end{equation}
is satisfied in the sense of distributions. More precisely, we will show that 
\begin{equation}\label{weaksol}
    \int_{\T^2\times [0,L]}\left(\text{div}j\right)B\cdot \nabla \varphi= \left. \int_{\mathbb{T}^{2}}  \left(\text{div}j\right)\, B_3\,  \varphi  \right|^{z=L}_{z=0}, \ \forall \varphi\in C^2 (\T^2\times [0,L]).
\end{equation}
Testing the left hand side of \eqref{transporte:1} against $\nabla\varphi$ and integrating by parts we find that 
     \begin{equation*}
         \begin{split}
             \sum_{\ell,j=1}^3 \int_{\T^2\times [0,L]}B_\ell \,\partial_\ell j_j\, \partial_j \varphi&= \sum_{\ell,j=1}^3 \int_{\T^2\times [0,L]} \partial_\ell \left(B_\ell \,j_j \,\partial_j \varphi\right) -\sum_{\ell,j=1}^3 \int_{\T^2\times [0,L]} B_\ell\,j_j\partial_\ell \,\partial_j \varphi\\
             &=\left.\left\langle B_3\, j\cdot\nabla \varphi\right\rangle \right|^{z=L}_{z=0}-\sum_{\ell,j=1}^3\int_{\T^2\times [0,L]} B_\ell\, j_j\, \partial_\ell\partial_j \varphi,
         \end{split}
     \end{equation*}
     using that $B$ is divergence-free. The same argument applied to the right hand side of \eqref{transporte:1} yields
     \begin{equation*}
         \begin{split}
             \sum_{\ell,j=1}^3 \int_{\T^2\times [0,L]} j_\ell\, \partial_ \ell B_j\, \partial_j\varphi &=\left.\left\langle j_3\, B\cdot \nabla \varphi\right\rangle\right|^{z=L}_{z=0} -\sum_{\ell,j=1}^3 \int_{\T^2\times [0,L]}\left(\partial_\ell j_\ell\right)B_j\, \partial_j\varphi \\
             &-\sum_{\ell,j=1}^3 \int_{\T^2\times [0,L]}j_\ell \,B_j\,\partial_\ell\partial_j \varphi.
         \end{split}
     \end{equation*}
Combining both identities we obtain that
$$\int_{\T^2\times [0,L]} \left(\text{div}\,j\right)\, B\cdot\nabla\varphi = \left.\left\langle j_3\, B\cdot \nabla \varphi\right\rangle\right|^{z=L}_{z=0}-\left.\left\langle B_3 \,j\cdot\nabla \varphi\right\rangle \right|^{z=L}_{z=0}.$$
The terms on the right hand side containing $\partial_{3}\varphi$ cancel out. Therefore, integrating by parts in $x,y$ we deduce that
\begin{align}
        \left\langle j_3\, B\cdot \nabla \varphi\right\rangle-\left\langle B_3 \,j\cdot\nabla \varphi\right\rangle&=\left\langle j_3\, B_1\,\partial_1 \varphi\right\rangle+\left\langle j_3\, B_2\,\partial_2 \varphi\right\rangle-\left\langle B_3 \,j_1\,\partial_1 \varphi\right\rangle-\left\langle B_3 \,j_2\,\partial_2\varphi\right\rangle \nonumber \\
&\hspace{-2.5cm}= -\left\langle \partial_1 j_3\, B_1\,\varphi\right\rangle -\left\langle \partial_2j_3\,B_2\,\varphi\right\rangle+\left\langle j_3\,\partial_3B_3\,\varphi\right\rangle +\left\langle \partial_1 B_3\, j_1\,\varphi\right\rangle +\left\langle \partial_2B_3\,j_2\,\varphi\right\rangle-\left\langle B_3\,\partial_3j_3\,\varphi\right\rangle+\left\langle B_3\,\text{div}j\,\varphi\right\rangle \nonumber  \\
&=\langle (j\cdot \nabla)B \rangle- \langle (B\cdot \nabla)j\rangle+\left\langle B_3\,\text{div}j\,\varphi\right\rangle= \left\langle B_3\,\text{div}j\,\varphi\right\rangle, \nonumber
\end{align}
yielding \eqref{weaksol}. The result follows after we show the uniqueness of weak solutions. To that purpose, we will use the following duality argument. Take $\psi\in C^{2}(\T^2\times[0,L])$ an arbitrary function and consider the following \textit{backwards} transport problem
\begin{equation}\label{eq:backwards}
\left\{\begin{array}{ll}
   B\cdot\nabla\varphi=\psi, &  \text{in }\Omega,\\
   \varphi=0,  & \text{on }\partial\Omega_+.
\end{array}\right.
\end{equation} 
Recalling that  $B=(0,0,1)+ (b_1,b_2,b_3)$ and $\|b\|_{C^2}< 1$ by hypothesis, we can use the method of characteristics to solve equation \eqref{eq:backwards}. The resulting characteristic curves yield a $C^2$ change of coordinates, and given that $\psi\in C^2$ we have that $\varphi\in C^2(\T^2\times[0,L])$. To conclude, using the fact that $\varphi=0$ on $\partial\Omega_{+}$, and by hypothesis $\text{div}j=0$ on $\partial\Omega_{-}$, we find that both terms in the right hand side of \eqref{weaksol} lead to no contribution. Thus,
$$\int_{\T^2\times[0,L]} \text{div}j\cdot\psi =0, \forall \psi\in C^{2}(\T^2\times[0,L]),$$
showing that $\text{div}j=0$ in $\Omega$.
\end{proof}





\section{The linearized problem}\label{sec:linearized}
The main purpose of this section is to describe the formal idea behind the method illustrated in the introduction to construct a solution $(B,p)$ of the magnetohydrostatic equations satisfying the boundary value conditions \eqref{mhs2}. Although we avoid any technicalities and remark that the exposed ideas are purely formal, we believe that the following digression is instructive to understand the core ideas towards the proof of Theorem \ref{mainteor}. 

To that purpose, we have to define an operator $\Gamma$ which has a fixed point $b$ such that $B=(0,0,1)+b$ is a solution to \eqref{mhs2}. Following the ideas developed by Alber in \cite{Alber-1992} (as well as in subsequent works \cite{Alo-Velaz-2021,Alo-Velaz-2022}) we divide the problem into two building blocks: a first order transport-type equation for the current $j$, and a div-curl problem for the magnetic field $B$. Given, $b\in C^{2,\alpha}$, we define $j\in C^{1,\alpha}(\Omega)$ as the solution to the transport equation
\begin{equation}\label{transport:2}
\begin{cases}
((0,0,1)+b)\cdot\nabla)j=(j\cdot\nabla)((0,0,1)+b) \quad \mbox{in } \Omega,\\
j=j_{0} \quad \mbox{on } \partial\Omega_{-}
\end{cases}
\end{equation}
Expanding the equation, we readily check that 
\begin{equation*}
    \partial_3 j+\frac{b_1}{1+b_3}\partial_1 j+\frac{b_2}{1+b_3}\partial_2 j=A(b)j, 
\end{equation*}
where the matrix $A(b)$ reads 
\begin{equation}\label{expression:matrix:A}
    A(b)=(a_{\ell j})_{\ell,j},\text{ with } a_{\ell j}=\frac{\partial_jb_\ell}{1+b_3}, \quad \mbox{ for } \ell,j=1,2,3.
\end{equation}
If $b$ is a small perturbation,  $\norm{b}_{C^{2,\alpha}(\Omega)}\ll 1$, and dropping the small nonlinear terms, the resulting equation simply reads  $\partial_{3}j=0$ in $\Omega$ and hence, $j$ remains constant along the vertical variable, this is
\begin{equation}\label{j_0:independent:z}
    j(x,y,z)=j(x,y)=j_{0}(x,y)  \mbox{ in } \Omega.
\end{equation}
Notice that $j_0(x,y)$ is an unknown of the problem. 
Taking into account this approximation, we solve the div-curl problem to recover a new magnetic field $B$. Since $B=(0,0,1)+b$, then the resulting problem is given by
\begin{equation}\label{divcurl:1}
\begin{cases}
\nabla\times B=j_{0}(x,y) \quad \mbox{in } \Omega,\\
\nabla \cdot B=0  \quad \mbox{in } \Omega, \\
B\cdot n =f \quad \mbox{on } \partial\Omega,
\end{cases}
\end{equation}
and complemented with the integral conditions 
$$\int_{\{x=0\}}B\cdot d\vec{S}=J_1\quad \text{ and }\quad \int_{\{y=0\}}B\cdot d\vec{S}=J_2.$$ 
The fluxes along the two directions of the period of the torus are a degree of freedom of the div-curl problem \eqref{divcurl:1}. Those conditions will be used in the sequel to obtain a uni-valued pressure function. Later, we will also impose that the constructed solution $B=(0,0,1)+b$ also satisfies the boundary condition 
\begin{equation}\label{imposing:tangential}
 B_{\tau}=g=(g_{1},g_{2})  \quad \mbox{on } \partial\Omega_{-}.
 \end{equation}

In order to solve \eqref{divcurl:1}, we use a vector potential $Z$. More precisely, we transform \eqref{divcurl:1} into an elliptic problem for $Z$. To that purpose, we define
\begin{equation}\label{potencialvector}
B=\nabla\times Z+A(0,0,1), \quad \mbox{ where  } A=\int_{\partial\Omega_-}f=\int_{\partial\Omega_+}f. 
\end{equation}
The newly defined potential vector $Z$ solves the following auxiliary problem 
\begin{equation}\label{aux:1}
    \left\{
    \begin{array}{ll}
       -\Delta Z=j_{0}  &  \text{in }\Omega,\\
        Z_2=h_2^{-},\,Z_1=h_1^{-} & \text{on }\partial\Omega_{-},\\
        Z_2=h_2^{+}-\frac{J_2}{(2\pi)^2},\, Z_1=h_1^{+} +\frac{J_1}{(2\pi)^2} & \text{on }\partial\Omega_{+},\\  
        \partial_3 Z=-\partial_1 Z_1-\partial_2 Z_2 & \text{on } \partial \Omega,
    \end{array}
    \right.
\end{equation}
where the different $h^{\pm}_\ell$ ($\ell=1,2$) are defined as 
\begin{equation} \label{boundary:f:h}
\begin{split}
h_2^-(x,y)&=\int_0^x f(s,y,0)\,ds-\frac{x}{2\pi}\int_0^{2\pi}f(s,y,0)\,ds,\\
h_1^{-}(x,y)&=-\frac{1}{2\pi}\int_0^{2\pi}\int_0^yf(s,t,0)\,dtds+\frac{y}{(2\pi)^2}\int_0^{2\pi}\int_0^{2\pi}f(s,t,0)\,dtds,\\
h_2^+(x,y)&=\int_0^xf(s,y,L)\,ds-\frac{x}{2\pi}\int_0^{2\pi}f(s,y,L)\,ds,\\
h_1^{+}(x,y)&=-\frac{1}{2\pi}\int_0^{2\pi}\int_0^yf(s,t,L)\,dtds+\frac{y}{(2\pi)^2}\int_0^{2\pi}\int_0^{2\pi}f(s,t,L)\,dtds.
\end{split}
\end{equation}
System \eqref{aux:1} consists on two uncoupled Dirichlet problems for $Z_1$ and $Z_2$, and a Neumann problem for $Z_3$. The first two problems for $Z_{1},Z_{2}$ can be solved without any further complication, but the Neumann problem for $Z_{3}$ requires that $j$ satisfies an integrability condition, which is indeed a consequence of $j_0^3$ having zero mean and imposing that $j$ is divergence free. Furthermore, one obtains as an easy consequence that $Z$ has zero divergence. Indeed, the divergence free condition on $j$ and the choice of $Z_3$ implies that $\text{div}\,Z$ is harmonic with zero boundary conditions, i.e. $\text{div}\,Z$ vanishes. Therefore, formula \eqref{potencialvector} gives us a solution for the div-curl system.
Similarly as in the two dimensional case \cite[Section 2]{Alo-Velaz-2022}, in order to solve the Dirichlet elliptic problem for $Z_{1}$ and $Z_{2}$, one can apply Fourier transform in the two periodic variables $x,y$. This transforms the two elliptic problems into two uncoupled second-order non-homogeneous ODE with constant coefficients. After using variation of parameters method we find that
\begin{align} 
    Z_1(r,z)&=\sum_{\xi\in\Z^2}\frac{\sinh(|\xi|L)-\sinh(|\xi|z)-\sinh(|\xi|(L-z))}{|\xi|^2\sinh(|\xi|L)}\reallywidehat{j}_0^1(\xi)e^{ir\cdot\xi}\nonumber\\
    &+\sum_{\xi\in \Z^2}\reallywidehat{h}_1^{-}(\xi)\frac{\sinh(|\xi|(L-z))}{\sinh(|\xi|L)}\foul+\sum_{\xi\in\Z^2}\reallywidehat{h}_1^+(\xi)\frac{\sinh(|\xi|z)}{\sinh(|\xi|L)}\foul+\frac{z}{(2\pi)L}J_1 \label{Z:1:computation} \\
    \nonumber\\
    Z_2(r,z)&=\sum_{\xi\in\Z^2}\frac{\sinh(|\xi|L)-\sinh(|\xi|z)-\sinh(|\xi|(L-z))}{|\xi|^2\sinh(|\xi|L)}\reallywidehat{j}_0^2(\xi)e^{ir\cdot\xi},\nonumber\\
    &+\sum_{\xi\in \Z^2}\reallywidehat{h}_2^{-}(\xi)\frac{\sinh(|\xi|(L-z))}{\sinh(|\xi|L)}\foul+\sum_{\xi\in\Z^2}\reallywidehat{h}_2^+(\xi)\frac{\sinh(|\xi|z)}{\sinh(|\xi|L)}\foul-\frac{z}{(2\pi)L},J_2.\label{Z:2:computation}
\end{align}
where we used the notation $r=(x,y)\in\mathbb{T}^{2}$. In the previous sums and throughout the rest of the article, the zero Fourier mode $\xi=0$ is understood as the limit as $\xi\rightarrow 0$. More precisely, if $\xi=0$ we will use that
\[\frac{\sinh(|\xi|(L-z))}{\sinh(|\xi|L)}\rightarrow \frac{L-z}{L},\qquad \frac{\sinh(|\xi|z)}{\sinh(|\xi|L)}\rightarrow \frac{z}{L},\]
and 
\[\frac{\sinh(|\xi|L)-\sinh(|\xi|z)-\sinh(|\xi|(L-z))}{|\xi|^2\sinh(|\xi|L)}\rightarrow \frac{Lz}{2}\left(1-\frac{z}{L}\right).\]
Furthermore, for further reference, if $\xi=0$ we follow the convention
\begin{equation}
\frac{\cosh(|\xi|L)-1}{|\xi|\sinh(|\xi|L)}\rightarrow \frac{L}{2}, \quad \frac{|\xi|}{\sinh(|\xi|L)}\rightarrow \frac{1}{L} , \quad \frac{|\xi|\cosh(|\xi|L)}{\sinh(|\xi|L)}\rightarrow \frac{1}{L}.
\end{equation}

In order to compute $Z_{3}$, one can solve the Neumann problem, but it can be easily deduced \textit{a posteriori} by means of the divergence free condition. Indeed, assuming that the constructed solution is divergence free we obtain 
\begin{equation}\label{div:nula}
\partial_1 Z_1+\partial_2 Z_2+\partial_3 Z_3=0,
\end{equation}
we have that $\partial_3^2Z_3=-\partial_3\partial_1Z_1-\partial_3\partial_2Z_2.$
Taking the Fourier transform in the first two variables, the equation reads
\begin{equation}\label{div:nula2}
|\xi|^2\reallywidehat{Z_3}(\xi,z)=\reallywidehat{j}_0^3(\xi)-im\partial_3\reallywidehat{Z}_1(\xi,z)-in\partial_3\reallywidehat{Z}_2(\xi,z). 
\end{equation}
Notice that the previous equation is trivially satisfied for $\xi=0$ for an arbitrary choice of $\reallywidehat{Z}_{3}(0,z)$  since 
 $\reallywidehat{j}_0^3(0)=0$. Therefore, combining \eqref{Z:1:computation}, \eqref{Z:2:computation} with \eqref{div:nula2} we obtain that
 \begin{align}\label{Z:3:computation}
   Z_3(r,z)&=\sum_{\xi\in \Z^2_*}\frac{1}{|\xi|^2}\reallywidehat{j}_0^3(\xi)\foul\,+\sum_{\xi\in\Z^2_*}\frac{im}{|\xi|^2}\frac{\cosh(|\xi|z)-\cosh(|\xi|(L-z))}{|\xi|\sinh(|\xi|L)}\reallywidehat{j}_0^2(\xi)\foul \nonumber\\
    &+\sum_{\xi\in\Z^2_*}\frac{im}{|\xi|^2}\frac{\cosh(|\xi|z)-\cosh(|\xi|(L-z))}{|\xi|\sinh(|\xi|L)}\reallywidehat{j}_0^1(\xi)\foul +\sum_{\xi\in \Z_{*}^2}\reallywidehat{h}_1^{-}(\xi)\frac{im}{|\xi|}\frac{\cosh(|\xi|(L-z))}{\sinh(|\xi|L)}\foul \nonumber \\
    &-\sum_{\xi\in\Z_{*}^2}\reallywidehat{h}_1^+(\xi)\frac{im}{|\xi|}\frac{\cosh(|\xi|z)}{\sinh(|\xi|L)}\foul
    +\sum_{\xi\in \Z_{*}^2}\reallywidehat{h}_2^{-}(\xi)\frac{in}{|\xi|}\frac{\cosh(|\xi|(L-z))}{\sinh(|\xi|L)}\foul \nonumber \\
    &\hspace{2cm} -\sum_{\xi\in\Z_{*}^2}\reallywidehat{h}_2^+(\xi)\frac{in}{|\xi|}\frac{\cosh(|\xi|z)}{\sinh(|\xi|L)}\foul + f(z).
\end{align} 
Here $f(z)$ denotes an arbitrary function and each node of the lattice $\xi\in \Z^2$ is written as $\xi=(m,n)$. Moreover,  we denote $\Z^2_*=\Z^2\setminus\{0\}$. Using the zero Fourier mode of the divergence free condition \eqref{div:nula} we find that $f'(z)=0$. Thus, $f(z)$ is constant and since $B$ is defined as \eqref{potencialvector} we can assume without loss of generality that $f(z)=0$.

 Hence, we have computed the potential vector field $Z=(Z_{1},Z_{2},Z_{3})$ given by the formulas \eqref{Z:1:computation}-\eqref{Z:3:computation} solving \eqref{aux:1}. Hence, recalling \eqref{potencialvector}, we can recover $B$ that satisfies \eqref{divcurl:1}. The main issue of this construction is that $j_0(x,y)$ is still an unknown of the problem. On the other hand, the constructed solution $B$ does not satisfies the tangential boundary condition \eqref{imposing:tangential}.  The key idea to circumvent this problem is to choose $j_0(x,y)$ in \eqref{Z:1:computation}, \eqref{Z:2:computation} and \eqref{Z:3:computation} in such a way
that \eqref{imposing:tangential} is satisfied. This will yield an integral equation for the unknown current $j_0(x,y)$.  More precisely, since  $B_{\tau}=g$ on $\partial\Omega_{-}$, we have that 
\begin{align*}
\partial_2 Z_3(r,0)-\partial_3 Z_2(r,0)=g_1(r) & \text{ on }\partial\Omega_{-},\\
 \partial_3 Z_1(r,0)-\partial_1 Z_3(r,0)=g_2(r) & \text{ on }\partial\Omega_{-}.
\end{align*}

Therefore, imposing the first condition on the vector potential $Z$ and using the explicit formulas for $Z_{2}$ and $Z_{3}$ given in \eqref{Z:2:computation}, \eqref{Z:3:computation} respectively, we obtain
\begin{align}\label{lin1}
    g_1(r)&=\sum_{\xi\in \Z^2_*}\frac{in}{|\xi|^2}\reallywidehat{j}_0^3(\xi)\foul \nonumber+\sum_{\xi\in\Z^2}\left(\frac{mn}{|\xi|^2}\chi_{\{\xi\neq0\}}\right)\frac{\cosh(|\xi|L)-1}{|\xi|\sinh(|\xi|L)}\,\reallywidehat{j}^{1}_0(\xi)\,\foul  \nonumber   \\ 
    &\quad +\sum_{\xi\in\Z^2}\left(\frac{n^2}{|\xi|^2}\chi_{\{\xi\neq0\}}-1\right)\frac{\cosh(|\xi|L)-1}{|\xi|\sinh(|\xi|L)}\,\reallywidehat{j}^{1}_0(\xi)\,\foul +\z_{1}(r)+\frac{J_2}{(2\pi)L}.
\end{align}
The function $\z_{1}$ that depends only on the boundary value $f$ (cf. \eqref{boundary:f:h}) is given by
\begin{align}
\z_{1}(r)&= -\sum_{\xi\in \Z^2}\reallywidehat{h}_1^{-}(\xi)\left(\frac{mn}{|\xi|^2}\chi_{\{\xi\neq0\}}\right)\frac{|\xi|\cosh(|\xi|L)}{\sinh(|\xi|L)}\foul +\sum_{\xi\in \Z^2}\reallywidehat{h}_1^{+}(\xi)\left(\frac{mn}{|\xi|^2}\chi_{\{\xi\neq0\}}\right)\frac{|\xi|}{\sinh(|\xi|L)}\foul\nonumber \\
&-\sum_{\xi\in \Z^2}\reallywidehat{h}_2^{-}(\xi)\left(\frac{n^2}{|\xi|^2}\chi_{\{\xi\neq0\}}-1\right)\frac{|\xi|\cosh(|\xi|L)}{\sinh(|\xi|L)}\foul +\sum_{\xi\in \Z^2}\reallywidehat{h}_2^{+}(\xi)\left(\frac{n^2}{|\xi|^2}\chi_{\{\xi\neq0\}}-1\right)\frac{|\xi|}{\sinh(|\xi|L)}\foul. \label{freakz1:computation}
\end{align}
Analogously, using the explicit formulas for $Z_{1}$ and $Z_{3}$ given in \eqref{Z:1:computation}, \eqref{Z:3:computation} respectively, we find that
\begin{align}\label{lin2}
g_2(r)&=-\sum_{\xi\in \Z^2_*}\frac{im}{|\xi|^2}\reallywidehat{j}_0^3(\xi)\foul -\sum_{\xi\in\Z^2}\left(\frac{m^2}{|\xi|^2}\chi_{\{\xi\neq0\}}-1\right)\frac{\cosh(|\xi|L)-1}{|\xi|\sinh(|\xi|L)}\,\reallywidehat{j}^{1}_0(\xi)\,\foul \nonumber   \\ 
    &-\sum_{\xi\in\Z^2}\left(\frac{mn}{|\xi|^2}\chi_{\{\xi\neq0\}}\right)\frac{\cosh(|\xi|L)-1}{|\xi|\sinh(|\xi|L)}\,\reallywidehat{j}^{2}_0(\xi)\,\foul +\z_{2}(r)+\frac{J_1}{(2\pi)L}.
\end{align}
Here, the function $\z_{2}$ is given by
\begin{align}
\z_{2}(r)&=\sum_{\xi\in \Z^2}\reallywidehat{h}_1^{-}(\xi)\left(\frac{m^2}{|\xi|^2}\chi_{\{\xi\neq0\}}-1\right)\frac{|\xi|\cosh(|\xi|L)}{\sinh(|\xi|L)}\foul -\sum_{\xi\in \Z^2}\reallywidehat{h}_1^{+}(\xi)\left(\frac{m^2}{|\xi|^2}\chi_{\{\xi\neq0\}}-1\right)\frac{|\xi|}{\sinh(|\xi|L)}\foul \nonumber \\
&+\sum_{\xi\in \Z^2}\reallywidehat{h}_2^{-}(\xi)\left(\frac{mn}{|\xi|^2}\chi_{\{\xi\neq0\}}\right)\frac{|\xi|\cosh(|\xi|L)}{\sinh(|\xi|L)}\foul-\sum_{\xi\in \Z^2}\reallywidehat{h}_2^{+}(\xi)\left(\frac{mn}{|\xi|^2}\chi_{\{\xi\neq0\}}\right)\frac{|\xi|}{\sinh(|\xi|L)}\foul. \label{freakz2:computation}
\end{align}

Moreover, we denote $\widetilde{g}_{\ell}(r)=g_{\ell}(r)-\z_{\ell}(r)$ for $\ell=1,2$. Notice that these functions depend only on the boundary values imposed on the function $B$ (cf.  \eqref{imposing:tangential} and \eqref{boundary:f:h}). We introduce the multiplier function 
\begin{equation}\label{multiplier:m:lineal}
\mathfrak{m}(\xi)=\frac{\cosh(|\xi|L)-1}{|\xi|\sinh(|\xi|L)},
\end{equation} that will be associated to the operator $\mathsf{T}$ defined as 
\[ \reallywidehat{\mathsf{T}f}(\xi)=\mathfrak{m}(\xi) \reallywidehat{f}(\xi).\]

Combining the previous notations, equations \eqref{lin1}-\eqref{lin2} read
\begin{align}
\widetilde{g}_{1}(r)&=\sum_{\xi\in \Z^2_*}\frac{in}{|\xi|^2}\reallywidehat{j}_0^3(\xi)\foul+\sum_{\xi\in\Z^2}\left(\frac{mn}{|\xi|^2}\chi_{\{\xi\neq0\}}\right)\mathfrak{m}(\xi)\,\reallywidehat{j}^{1}_0(\xi)\,\foul \nonumber \\
&\quad  \quad +\sum_{\xi\in\Z^2}\left(\frac{n^2}{|\xi|^2}\chi_{\{\xi\neq0\}}-1\right)\mathfrak{m}(\xi)\,\reallywidehat{j}^{1}_0(\xi)\,\foul+\frac{J_{2}}{(2\pi)L}, \label{lin3} \\
\widetilde{g}_{2}(r)&=-\sum_{\xi\in \Z^2_*}\frac{im}{|\xi|^2}\reallywidehat{j}_0^3(\xi)\foul+\sum_{\xi\in\Z^2}\left(\frac{m^2}{|\xi|^2}\chi_{\{\xi\neq0\}}\right)\mathfrak{m}(\xi)\,\reallywidehat{j}^{1}_0(\xi)\,\foul  \nonumber \\
&\quad  -\sum_{\xi\in\Z^2}\left(\frac{mn}{|\xi|^2}\chi_{\{\xi\neq0\}}\right)\mathfrak{m}(\xi)\,\reallywidehat{j}^{2}_0(\xi)\,\foul+\frac{J_{1}}{(2\pi)L}. \label{lin4}
\end{align}
Using the fact that the multiplier $\mathfrak{m}(\xi)\neq 0, \forall \xi\in\mathbb{Z}^{2}$, it turns out that the operator $\mathsf{T}^{-1}$ exists. Hence, applying $\mathsf{T}^{-1}$ to both sides of equations \eqref{lin3}-\eqref{lin4} we find that
 \begin{align}
     \label{preintegrallineal1}\reallywidehat{\textsf{G}_1}(\xi)&=\frac{in}{|\xi|^2}(\mathfrak{m}(\xi))^{-1}\chi_{\{\xi\neq0\}}\reallywidehat{j}^3_0(\xi)+\left(\frac{n^{2}}{|\xi|^{2}}\chi_{\{\xi\neq0\}}-1\right)\reallywidehat{j}_0^2(\xi)+\frac{mn}{|\xi|^{2}}\chi_{\{\xi\neq0\}}\reallywidehat{j}_0^1(\xi)+\frac{J_2}{\pi L^2}\delta_{0,n}\delta_{0,m}, \\
     \label{preintegrallineal2}\reallywidehat{\textsf{G}_2}(\xi)&=-\frac{im}{|\xi|^2}(\mathfrak{m}(\xi))^{-1}\chi_{\{\xi\neq0\}}\reallywidehat{j}^3_0(\xi)-\left(\frac{m^{2}}{|\xi|^{2}}\chi_{\{\xi\neq0\}}-1\right)\reallywidehat{j}_0^1(\xi)-\frac{mn}{|\xi|^2}\chi_{\{\xi\neq0\}}\reallywidehat{j}_0^2(\xi)+\frac{J_1}{\pi L^2}\delta_{0,n}\delta_{0,m},
 \end{align}
where $\reallywidehat{\textsf{G}_{\ell}}(\xi)=(\mathfrak{m}(\xi))^{-1}\reallywidehat{\widetilde{g}}_{\ell}(\xi)$ for $\ell=1,2$. For $\xi=0$ the solution of the system \eqref{preintegrallineal1}-\eqref{preintegrallineal2} is given by 
\begin{equation}\label{final:lin:xi0}
 \reallywidehat{j}_{0}^{1}(0)=\reallywidehat{\textsf{G}_2}(0)-\frac{J_1}{\pi L^2}, \quad \reallywidehat{j}_{0}^{2}(0)=-\reallywidehat{\textsf{G}_1}(0)+\frac{J_2}{\pi L^2}. 
 \end{equation}
On the other hand if $\xi\neq 0$ the system \eqref{preintegrallineal1}-\eqref{preintegrallineal2} has three unknowns $j_{0}=(j_{0}^{1},j_{0}^{2},j_{0}^{3})$. Furthermore, recall that we have a constraint on the divergence of the current $j$, i.e. , $\nabla\cdot j=0$. This allows to reduce the number of independent functions from three to two. More precisely, in the particular linearized setting under consideration, \eqref{j_0:independent:z} implies that $\partial_{3}j_{0}^{3}=0$. Therefore the divergence free condition for $j$ yields
\begin{equation}\label{divergencialineal}
    im\reallywidehat{j}_0^1+in\reallywidehat{j}_0^2=0.
\end{equation}
Then, for $\xi\neq 0$,  the system of equations \eqref{preintegrallineal1}-\eqref{divergencialineal} can be inverted since 
\begin{equation}\label{cuenta:det:mxi}
 \mbox{det} A= -(\mathfrak{m}(\xi))^{-1}\neq 0, \mbox{ where } A= \left(\begin{array}{ccc}
        \frac{mn}{|\xi|^2} & \frac{n^2}{|\xi|^2}-1   &\frac{in}{|\xi|^2}(\mathfrak{m}(\xi))^{-1}\\
        -\frac{m^2}{|\xi|^2}+1 & -\frac{mn}{|\xi|^2} &-\frac{im}{|\xi|^2}(\mathfrak{m}(\xi))^{-1}\\
        im & in & 0
    \end{array}\right).
\end{equation}
In addition, we have that
\begin{equation}
  A^{-1}=\left(\begin{array}{ccc}
        \frac{mn}{|\xi|^2} & \frac{n^2}{|\xi|^2}  &\frac{-im}{|\xi|^2} \\
        -\frac{m^2}{|\xi|^2} & -\frac{mn}{|\xi|^2} &-\frac{in}{|\xi|^2}\\
        -in \mathfrak{m}(\xi) & im  \mathfrak{m}(\xi) & 0
    \end{array}\right),
\end{equation}
 and therefore, we can invert the linear system and obtain that for $\xi\neq 0$
 \begin{equation}\label{final:lin:xi1}
 \begin{split}
\reallywidehat{j}_{0}^{1}(\xi)&=\frac{mn}{|\xi|^2} \reallywidehat{\textsf{G}_1}(\xi)+\frac{n^2}{|\xi|^2}\reallywidehat{\textsf{G}_2}(\xi), \\
\reallywidehat{j}_{0}^{2}(\xi)&=-\frac{m^{2}}{|\xi|^2} \reallywidehat{\textsf{G}_1}(\xi)-\frac{mn}{|\xi|^2}\reallywidehat{\textsf{G}_2}(\xi), \\
\reallywidehat{j}_{0}^{3}(\xi)&=in\reallywidehat{g_{1}}(\xi)-im\reallywidehat{g_{2}}(\xi),
 \end{split}
 \end{equation}
where in the last equality we used the fact that $\reallywidehat{\textsf{G}_{\ell}}(\xi)=(\mathfrak{m}(\xi))^{-1}\reallywidehat{\widetilde{g}}_{\ell}(\xi)$ for $\ell=1,2$. Therefore, combining \eqref{final:lin:xi0} and \eqref{final:lin:xi1} we conclude that for $\xi\in\mathbb{Z}^{2}$
 \begin{equation}\label{final:lin}
 \begin{split}
j_{0}^{1}(r)&=\displaystyle\sum_{\xi\in\mathbb{Z}^{2}_{*}}\frac{mn}{|\xi|^2} \reallywidehat{\textsf{G}_1}(\xi)\foul-\displaystyle\sum_{\xi\in\mathbb{Z}^{2}_{*}}\frac{m^2}{|\xi|^2} \reallywidehat{\textsf{G}_2}(\xi)\foul+ \textsf{G}_2(r)-\frac{J_1}{\pi L^2}, \\
j_{0}^{2}(r)&= - \displaystyle\sum_{\xi\in\mathbb{Z}^{2}_{*}}\frac{mn}{|\xi|^2}\reallywidehat{\textsf{G}_2}(\xi)\foul +\displaystyle\sum_{\xi\in\mathbb{Z}^{2}_{*}}\frac{n^2}{|\xi|^2} \reallywidehat{\textsf{G}_1}(\xi)\foul-\textsf{G}_1(r)+\frac{J_2}{\pi L^2}, \\
j_{0}^{3}(r)&=\partial_{1}g_{2}(r)-\partial_{2}g_{1}(r).
\end{split}
 \end{equation}
Hence, expression for $j_{0}$ in \eqref{final:lin} yields the value of $Z$ by means of formulas \eqref{Z:1:computation}, \eqref{Z:2:computation} and \eqref{Z:3:computation} and thus the value of a solution $B$ via \eqref{potencialvector} solving the desired boundary value problem \eqref{divcurl:1}-\eqref{imposing:tangential}. However, the values $J_{1}, J_{2}$ appearing in \eqref{final:lin} are until now undetermined. We will fix both values in order to guarantee that the pressure $p$ (which is also an unknown of the problem) is a uni-valued function in $\Omega$. Indeed using the first equation in \eqref{mhs2} and linearizing around $B=(0,0,1)+b$ we find that
\begin{equation}\label{pressure:cond:j:lin}
j^{2}(x,y,z)=\partial_{x}p(x,y,z), \quad -j^{1}(x,y,z)=\partial_{y}p(x,y,z), \quad 0=\partial_{z}p(x,y,z).
\end{equation}
Equation \eqref{pressure:cond:j:lin} shows that $p(x,y,z)=p(x,y)$ and therefore $j^{1}(x,y,z)=j^{1}_{0}(x,y)$ and $j^{2}(x,y,z)=j^{2}_{0}(x,y)$ as well. Moreover, \eqref{pressure:cond:j:lin} implies that $\nabla p=\nabla p_{|\partial\Omega_{-}}=(j^{1}_{0},j^{2}_{0})$. Hence, we can obtain a solution to \eqref{pressure:cond:j:lin} 
\[ p(x,y,z)=\int_{0}^{\textbf{\textit{x}}}j_{0}(x,y) \ d\textbf{\textit{x}}, \]
where the integral on the right hand side in the line integration computed along any contour that connects $(0,0,0)$ and $\textbf{\textit{x}}\in \Omega$. Therefore, to ensure that the pressure $p(x,y,z)$ is a uni-valued function in $\Omega=\T^2\times[0,L]$ we need that 
\[
  \int_{0}^{2\pi}\int_{0}^{2\pi} j^{1}(x,y,0) \ dxdy=0,  \int_{0}^{2\pi}\int_{0}^{2\pi} j^{2}(x,y,0) dx dy=0.
\]
Using \eqref{final:lin} we obtain that
\[ \frac{J_{1}}{\pi L^{2}}=\langle \textsf{G}_{2}\rangle, \quad  \frac{J_{2}}{\pi L^{2}}=\langle \textsf{G}_{1}\rangle.\]

\section{The non-linear setting: the integral equation for the current}\label{sec:4}
In this section, we will follow the same arguments used in Section \ref{sec:linearized} but taking into account the whole non-linear problem. In particular, we will derive an integral equation for the current $j$ on $\partial\Omega_{-}$, namely, $j_0(x,y)$ for $(x,y)\in\mathbb{T}^{2}$. As expected, the integral equation will be a perturbation of equation \eqref{final:lin} in terms of the perturbed magnetic field $b$ and the boundary values $f$ and $g$. During this section, the computations will be again formal and we will not consider convergence issues that might arise from the Fourier series. The following section (cf. Section \ref{sec:5}) is devoted to justify the convergence of the Fourier series and the precise definitions of the operators derived hereafter.

Given, $b\in C^{2,\alpha}$ and $B=(0,0,1)+b$ we define $j\in C^{1,\alpha}(\Omega)$ as the solution to 
\begin{equation}\label{transport}
\begin{cases}
((0,0,1)+b)\cdot\nabla)j=(j\cdot\nabla)((0,0,1)+b) \quad \mbox{in } \Omega,\\
j=j_{0} \quad \mbox{on } \partial\Omega_{-}.
\end{cases}
\end{equation}
More precisely, we can write the previous system as 
\begin{equation}\label{transport:modified}
    \partial_3 j+\frac{b_1}{1+b_3}\partial_1 j+\frac{b_2}{1+b_3}\partial_2 j=A(b)j, 
\end{equation}
where the matrix $A(b)$ reads 
\begin{equation}\label{transport:matrix:modified}
    A(b)=(a_{\ell j})_{\ell,j},\text{ with } a_{\ell j}=\frac{\partial_jb_\ell}{1+b_3}.
\end{equation}
The system can be solved by means of characteristics. Indeed, using Proposition \ref{existenciatransporte},
the solution is explicitly given by 
\begin{equation}\label{j:current:nonlinear1}
j(x,y,z)=w\circ\Phi^{-1}(x,y,z), 
\end{equation}
where $w$ satisfies the differential equation \eqref{regularidad4} and $\Phi$ is the flow mapping defined by the characteristic lines of \eqref{transport} given in \eqref{car}-\eqref{flows}. Notice that a priori, the value of the $j_0 \mbox{ on } \partial\Omega_{-}$ is unknown. 

As in the linearizing setting (cf. equations \eqref{divcurl:1}-\eqref{imposing:tangential} in Section \ref{sec:linearized}), we next solve the div-curl problem to obtain a new field $W$, namely,
\begin{equation}\label{divcurl}
\left\{
\begin{array}{ll}
    \nabla\times W=j& \text{in }\Omega, \\
    \text{div}\,W=0 & \text{in }\Omega, \\
    W\cdot n=f & \text{on }\partial\Omega, \\
    W_{\tau}=g & \text{on }\partial\Omega_{-},
\end{array}
\right.    
\end{equation}
complemented with the integral conditions 
$$\int_{\{x=0\}}W\cdot d\vec{S}=\frac{J_1}{(2\pi)^2}\quad \text{ and }\quad \int_{\{y=0\}}W\cdot d\vec{S}=\frac{J_2}{(2\pi)^2}.$$ 
In order to solve \eqref{divcurl}, we use a divergence free vector potential $Z$, namely 
\[W=\nabla\times Z+A(0,0,1), \quad \mbox{ where  } A=\int_{\partial\Omega_-}f=\int_{\partial\Omega_+}f\] The resulting problem (identical as in \eqref{aux:1} in Section \ref{sec:linearized}) reads
\begin{equation}\label{aux}
    \left\{
    \begin{array}{ll}
       -\Delta Z=j  &  \text{in }\Omega, \\
         Z_2=h_2^{-},\,Z_1=h_1^{-} & \text{on }\partial\Omega_{-}, \\
        Z_2=h_2^{+}-\frac{J_2}{(2\pi)^2},\, Z_1=h_1^{+} +\frac{J_1}{(2\pi)^2}& \text{on }\partial\Omega_{+},\\
         \partial_3 Z=-\partial_1 Z_1-\partial_2 Z_2, & \text{on } \partial \Omega.
    \end{array}
    \right. 
\end{equation}
where the different $h^{\pm}_i$ ($i=1,2$) are defined as in \eqref{boundary:f:h} and $j=j(x,y,z)$ is given by \eqref{j:current:nonlinear1}. This leads again to an uncoupled problem for the first two components $Z_1$ and $Z_2$. Following the approach in \cite[Section 3.1]{Alo-Velaz-2022}, by means of the Green's function for the Laplace operator, the unique solutions $Z_1$ and $Z_2$ are  
\begin{align*}
    Z_1(r,z)&=\int_0^L\int_{\T^2}\sum_{\xi\in\Z^2}\mathfrak{G}(\xi,z,z_0)\,j^{1}(\zeta,z_0)\,\fou d\zeta dz_0 \nonumber\\
    &+\sum_{\xi\in \Z^2}\reallywidehat{h}_1^{-}(\xi)\frac{\sinh(|\xi|(L-z))}{\sinh(|\xi|L)}\foul+\sum_{\xi\in\Z^2}\reallywidehat{h}_1^+(\xi)\frac{\sinh(|\xi|z)}{\sinh(|\xi|L)}\foul+\frac{z}{(2\pi)L}J_1,  \\
    Z_2(r,z)&=\int_0^L\int_{\T^2}\sum_{\xi\in\Z^2}\mathfrak{G}(\xi,z,z_0)\,j^{2}(\zeta,z_0)\,\fou d\zeta dz_0 \nonumber\\
    &+\sum_{\xi\in \Z^2}\reallywidehat{h}_2^{-}(\xi)\frac{\sinh(|\xi|(L-z))}{\sinh(|\xi|L)}\foul+\sum_{\xi\in\Z^2}\reallywidehat{h}_2^+(\xi)\frac{\sinh(|\xi|z)}{\sinh(|\xi|L)}\foul-\frac{z}{(2\pi)L}J_2. 
\end{align*}
where 
\begin{equation}\label{green:function:laplace}
\mathfrak{G}(\xi,z;z_0)= \left\{\begin{array}{ll}
  \frac{1}{(2\pi)^2}  \frac{\sinh(|\xi|z)\sinh(|\xi|(L-z_0))}{|\xi|\sinh(|\xi|L)} & z<z_0, \\\\\
   \frac{1}{(2\pi)^2} \frac{\sinh(|\xi|z_0)\sinh(|\xi|(L-z)) }{|\xi|\sinh(|\xi|L)}& z\geq z_0,
\end{array}\right. 
\end{equation}
is the Green function solving the problem
\begin{equation}
\begin{cases}
-\frac{\partial^{2}}{\partial z^{2}}\mathfrak{G}(\xi,z;z_0) + |\xi|^{2}\mathfrak{G}(\xi,z;z_0)=\frac{1}{(2\pi)^2}\delta(z-z_{0}),\ \mbox{ in } \Omega, \\ 
\mathfrak{G}=0, \ \mbox{ on } \partial\Omega.
\end{cases}
\end{equation}
Following the computations \eqref{div:nula}-\eqref{Z:3:computation} in Section \ref{sec:linearized} using the divergence free condition, we can determine $Z_{3}$, namely,
\begin{align}
   Z_3(r,z)&=\frac{1}{(2\pi)^2}\int_{\T^2}\sum_{\xi\in \Z^2_*}\frac{1}{|\xi|^2}j^3(\zeta,z)\fou\, d\zeta-\int_0^L\int_{\T^2}\sum_{\xi\in\Z_{*}^2}\frac{im}{|\xi|^2}\partial_3\mathfrak{G}(\xi,z,z_0)\,j^{1}(\zeta,z_0)\,\fou d\zeta dz_0 \nonumber\\
    &\hspace*{-1cm}-\int_0^L\int_{\T^2}\sum_{\xi\in\Z_{*}^2}\frac{in}{|\xi|^2}\partial_3\mathfrak{G}(\xi,z,z_0)\,j^{2}(\zeta,z_0)\,\fou d\zeta dz_0 + \sum_{\xi\in \Z_{*}^2}\reallywidehat{h}_1^{-}(\xi)\frac{im}{|\xi|}\frac{\cosh(|\xi|(L-z))}{\sinh(|\xi|L)}\foul \nonumber \\
    &\hspace*{-0.5cm}-\sum_{\xi\in\Z_{*}^2}\reallywidehat{h}_1^+(\xi)\frac{im}{|\xi|}\frac{\cosh(|\xi|z)}{\sinh(|\xi|L)}\foul +\sum_{\xi\in \Z_{*}^2}\reallywidehat{h}_2^{-}(\xi)\frac{in}{|\xi|}\frac{\cosh(|\xi|(L-z))}{\sinh(|\xi|L)}\foul-\sum_{\xi\in\Z_{*}^2}\reallywidehat{h}_2^+(\xi)\frac{in}{|\xi|}\frac{\cosh(|\xi|z)}{\sinh(|\xi|L)}\foul. \label{Z3:computaton:non}
\end{align}
Once we have found the only possible solutions for the vector potential $Z$, we can use the remaining tangential boundary condition $W_{\tau}=g$ on $\partial\Omega_{-}$ which in terms of potential $Z$ reads
\begin{align*}
\partial_2 Z_3(r,0)-\partial_3 Z_2(r,0)=g_1(r) & \text{ on }\partial\Omega_{-},\\
 \partial_3 Z_1(r,0)-\partial_1 Z_3(r,0)=g_2(r) & \text{ on }\partial\Omega_{-},
\end{align*}
to obtain an integral equation for the current $j$. Notice that the resulting formulas are very similar to those obtained in the linearized case with the only difference that the current expression for $j$ is more involved.  After some lengthy but straightforward computations we find that
\begin{align}\label{prim1}
    \widetilde{g}_1(r)&=\frac{1}{(2\pi)^2}\int_{\T^2}\sum_{\xi\in \Z^2_*}\frac{in}{|\xi|^2}j^3(\zeta,0)\fou\, d\zeta +\int_0^L\int_{\T^2}\sum_{\xi\in\Z^2}\left(\frac{mn}{|\xi|^2}\chi_{\{\xi\neq0\}}\right)\partial_3 \mathfrak{G}(\xi,0,z_0)\,j^{1}(\zeta,z_0)\,\fou d\zeta dz_0 \nonumber   \\ 
    &+\int_0^L\int_{\T^2}\sum_{\xi\in\Z^2}\left(\frac{n^2}{|\xi|^2}\chi_{\{\xi\neq0\}}-1\right)\partial_3 \mathfrak{G}(\xi,0,z_0)\,j^{2}(\zeta,z_0)\,\fou d\zeta dz_0+\frac{J_2}{(2\pi)L},
\end{align}
where $\widetilde{g}_{1}(r)= g_{1}(r)-\z_{1}(r)$ and $\z_{1}(r)$ depends only on the boundary value $f$ and is given in \eqref{freakz1:computation}. Analogously,
\begin{align}\label{prim2}
\widetilde{g}_{2}(r)&=-\frac{1}{(2\pi)^2}\int_{\T^2}\sum_{\xi\in \Z^2_*}\frac{im}{|\xi|^2}j^3(\zeta,0)\fou\, d\zeta -\int_0^L\int_{\T^2}\sum_{\xi\in\Z^2}\left(\frac{m^2}{|\xi|^2}\chi_{\{\xi\neq0\}}-1\right)\partial_3 \mathfrak{G}(\xi,0,z_0)\,j^{1}(\zeta,z_0)\,\fou d\zeta dz_0 \nonumber   \\ 
    &-\int_0^L\int_{\T^2}\sum_{\xi\in\Z^2}\left(\frac{mn}{|\xi|^2}\chi_{\{\xi\neq0\}}\right)\partial_3 \mathfrak{G}(\xi,0,z_0)\,j^{2}(\zeta,z_0)\,\fou d\zeta dz_0 +\frac{J_1}{(2\pi)L}.
\end{align}
Here, $\widetilde{g}_{2}(r)= g_{2}(r)-\z_{2}(r)$ and $\z_{2}(r)$ is defined in \eqref{freakz2:computation}.  In order to write a more convenient expression for the current, let us introduce the operator
\begin{equation}\label{operator:A:general}
\mathcal{A}[\theta](r,z)=\int_0^{L}\int_{\T^2}\sum_{\xi\in\Z^2} \partial_3\mathfrak{G}(\xi,0,z_0)\theta(\zeta,z_{0})\fou\,d\zeta dz_0.
\end{equation}
In particular, for $\theta(r,z)=j^{\ell}_{0}\circ \Phi^{-1}$, $\ell\in\{1,2\}$ we have that
\begin{align}
\mathcal{A}[j^{\ell}_{0}\circ \Phi^{-1}](r,z)&=\int_0^{L}\int_{\T^2}\sum_{\xi\in\Z^2} \partial_3\mathfrak{G}(\xi,0,z_0)\left(j^{\ell}_{0}\circ \Phi^{-1}\right)(\zeta,z_{0})\fou\,d\zeta dz_0 \nonumber \\
&=\frac{1}{(2\pi)^2}\int_0^{L}\int_{\T^2}\sum_{\xi\in\Z^2}\frac{\sinh(|\xi|(L-z_0))}{\sinh(|\xi|L)}\left(j^{\ell}_{0}\circ \Phi^{-1}\right)(\zeta,z_{0})\fou\,d\zeta dz_0, \label{expression:A:operator1}
\end{align}
where the normal derivative of $\mathfrak{G}(\xi,z;z_{0})$ at $z=0$, i.e. $\partial_3\mathfrak{G}(\xi,0,z_0)$ has been computed using \eqref{green:function:laplace}. Furthermore, we make the change of variable given by diffeomorphism $g(\eta)=(X(\eta,s),Y(\eta,s))$ with $X,Y$ satisfying \eqref{car} and with inverse $(X^{-1},Y^{-1})$. The Jacobian is defined by
$$J_g=\left(\begin{array}{cc}
    \partial_1 X & \partial_2 X \\
    \partial_1 Y & \partial_2 Y
\end{array}\right)(\eta,s)=\left(\begin{array}{cc}
    \partial_1 X^{-1} & \partial_2 X^{-1} \\
    \partial_1 Y^{-1} & \partial_2 Y^{-1}
\end{array}\right)^{-1}\left(X(\eta,s),Y(\eta,s)\right).$$
Denoting by  $\Theta(\eta,s)=|J_g|-1$,  \eqref{expression:A:operator1} becomes 
\begin{align}\mathcal{A}[j^{\ell}_{0}\circ \Phi^{-1}](r,z)&=\frac{1}{(2\pi)^2}\int_0^{L}\int_{\T^2}\sum_{\xi\in\Z^2}\frac{\sinh(|\xi|(L-s))}{\sinh(|\xi|L)}\frac{j^{\,\ell}_0\left(\eta\right)}{1+\Theta(\eta,s)}e^{i\left(r-\Phi(\eta,s)\right)\cdot \xi}\,d\eta ds \nonumber \\
&=\frac{1}{(2\pi)^2}\int_{\T^2}\mathcal{G}(r-\eta,\eta)j_0^{\,\ell}(\eta)d\eta, \label{operatorA:newvariable}
\end{align}
with 
\begin{equation}\label{kernel:original:G}
\mathcal{G}(r,\eta)=\sum_{\xi\in\Z^2}a_{\xi}(\eta)e^{ir\cdot \xi}, \ \mbox{where } a_{\xi}(\eta)=\int_0^L\frac{\sinh(|\xi|(L-s))}{\sinh(|\xi|L)}\frac{e^{i\Lambda(\eta,s)\cdot \xi}}{1+\Theta(\eta,s)}ds.
\end{equation}
Furthermore,  using the same decomposition as in \cite[Section 3.2]{Alo-Velaz-2022}, we can write the kernel \eqref{kernel:original:G} as  \[ \mathcal{G}(r,\eta)=\mathcal{G}_{0}(r,\eta)+\displaystyle\sum_{\kappa=1}^4\mathcal{G}_{\kappa}(r,\eta), \] where
\begin{align}
       \label{op0} \mathcal{G}_0(r)&=\sum_{\xi\in \Z^2}\mathfrak{m}(\xi) \foul,\\
    \label{op1}
    \mathcal{G}_1(r,\eta)&=\sum_{\xi\in\Z^2} \foul\int_0^Le^{-|\xi|s}\left(\frac{e^{-i\xi\cdot\Lambda(\eta,s)}-1}{1+\Theta(\eta,s)}\right)ds,\\
    \label{op2}\mathcal{G}_2(r,\eta)&=\sum_{\xi\in\Z^2}\foul\int_0^Le^{-|\xi|s}\left(\frac{\Theta(\eta,s)}{1+\Theta(\eta,s)}\right)ds,\\
    \label{op3}\mathcal{G}_3(r,\eta)&=\sum_{\xi\in\Z^2} \foul\int_0^LM(\xi,s)\left(\frac{e^{-i\xi\cdot\Lambda(\eta,s)}-1}{1+\Theta(\eta,s)}\right)ds,\\
    \label{op4}\mathcal{G}_4(r,\eta)&=\sum_{\xi\in\Z^2}\foul\int_0^LM(\xi,s)\left(\frac{\Theta(\eta,s)}{1+\Theta(\eta,s)}\right)ds.
\end{align}
We recall that the multiplier $\mathfrak{m}(\xi)$ is defined in \eqref{multiplier:m:lineal} and the function $M(\xi,s)$ is defined as
\begin{equation}\label{definition:M:smoothing}
 M(\xi,s)=\frac{e^{-2|\xi|L}\left(e^{|\xi|s}-e^{-|\xi|s}\right)}{1-e^{-2|\xi|L}}.
 \end{equation}
Using decomposition \eqref{op0}-\eqref{op4}, we rewrite the operator $\mathcal{A}[j^{\ell}_{0}\circ \Phi^{-1}]$ in \eqref{operatorA:newvariable} as
\begin{equation} \label{operatorA:decomp}
\mathcal{A}[j^{\ell}_{0}\circ \Phi^{-1}](r,z)=\mathcal{T}j^{\ell}_{0}= \mathcal{T}_{0}j^{\ell}_{0}+\displaystyle\sum_{\kappa=1}^{4} \mathcal{T}_{\kappa}j^{\ell}_{0}
\end{equation}
where
\begin{align}
\mathcal{T}_{0}j^{\ell}_{0}&=\frac{1}{(2\pi)^2}\int_{\T^2}\mathcal{G}_{0}(r-\eta)j_0^{\ell}(\eta)d\eta, \mbox{ for } \ell=1,2, \label{opT0} \\
\mathcal{T}_{\kappa}j^{\ell}_{0}&=\frac{1}{(2\pi)^2}\int_{\T^2}\mathcal{G}_{\kappa}(r-\eta,\eta)j_0^{\ell}(\eta)d\eta, \mbox{ for } \kappa=1,\ldots, 4, \ \ell=1,2. \label{opTk}
\end{align}
It is important to notice the operator $\mathcal{T}_{0}$ coincides with the operator $\mathsf{T}$ in Section \ref{sec:linearized} given by means of the multiplier function $\mathfrak{m}(\xi)=\frac{\cosh(|\xi|L)-1}{|\xi|\sinh(|\xi|L)}$. Therefore, $\mathcal{T}_{0}$ is an operator that can be inverted as long as  $(\mathfrak{m}(\xi))^{-1}\neq 0$, fact that we have already shown in \eqref{cuenta:det:mxi} in Section \ref{sec:linearized}.

We may use now the previous digression to help us deriving the integral equation for $j_0$. To that purpose, we denote by $\delta j=j-\bar{\jmath}$ (notation introduced in Corollary \ref{direct:coro:proptransport}) to write
\begin{equation}\label{current:decom:1}
j(x,y,z)=\delta j + \bar{\jmath}= [(w-j_{0})\circ\Phi^{-1}+ j_{0}\circ \Phi^{-1}](x,y,z) , \ \mbox{ in } \Omega,
\end{equation}
where $w-j_{0}$ satisfies
\[\left\{\begin{array}{ll}
        \frac{\partial (w-j_0)}{\partial z}=A_b\cdot (w-j_0)+A_b\cdot j_0 & z>0, \\
        w-j_0=0, & z=0.
        \end{array}\right. 
        \]
Hence, using \eqref{current:decom:1} and the definition of the operator \eqref{operator:A:general}, the second and third term on the right hand side in \eqref{prim1} are given by
\begin{align*}
\int_0^L\int_{\T^2}\sum_{\xi\in\Z^2}\left(\frac{mn}{|\xi|^2}\chi_{\{\xi\neq0\}}\right)\partial_3 \Phi(\xi,0,z_0)\,j^{1}(\zeta,z_0)\,\fou d\zeta dz_0&=  \mathcal{R}_{x}\mathcal{R}_{y}\left(\mathcal{A}(\delta j^1)+\mathcal{A}( j^1_{0}) \right),\\ 
\int_0^L\int_{\T^2}\sum_{\xi\in\Z^2}\left(\frac{n^2}{|\xi|^2}\chi_{\{\xi\neq0\}}-1\right)\partial_3 \mathfrak{G}(\xi,0,z_0)\,j^{2}(\zeta,z_0)\,\fou d\zeta dz_0&= \left(\mathcal{R}^{2}_{y}-\text{id}\right) \left(\mathcal{A}(\delta j^2)+\mathcal{A}( j^2_{0}) \right).
 \end{align*}
We recall that $\mathcal{R}_x$ and $\mathcal{R}_y$ are the Riesz transforms. Thus, we infer using \eqref{operatorA:decomp} that
\begin{align}\label{prim1:new}
    \tilde{g}_1-\frac{J_2}{(2\pi)L}-\mathcal{S}_{1}&=\mathcal{R}_x\mathcal{R}_y\mathcal{T}j_0^1+\left(\mathcal{R}^{2}_y-\text{id}\right)\mathcal{T}j_0^2+\sum_{\xi\in\Z^2_{*}}\frac{in}{|\xi|^2}\reallywidehat{j}^3_0(\xi)\foul,
    \end{align}
  with $\mathcal{S}_{1}=\mathcal{R}_x\mathcal{R}_y\mathcal{A}(\delta j^1)+\left(\mathcal{R}^{2}_y-\text{id}\right)\mathcal{A}(\delta j^2)$.  Similarly, for  \eqref{prim2} we find that
    \begin{align}\label{prim2:new}
    \tilde{g}_2-\frac{J_1}{(2\pi)L}+\mathcal{S}_{2}&=-\left(\mathcal{R}^{2}_x-\text{id}\right)\mathcal{T}j_0^1-\mathcal{R}_x\mathcal{R}_y\mathcal{T}j_0^2-\sum_{\xi\in\Z^2_{*}}\frac{im}{|\xi|^2}\reallywidehat{j}^3_0(\xi)\foul,
\end{align}
with $\mathcal{S}_{2}=\left(\mathcal{R}^{2}_x-\text{id}\right)\mathcal{A}(\delta j^1)+\mathcal{R}_x\mathcal{R}_y\mathcal{A}(\delta j^2)$. Furthermore, invoking decomposition \eqref{operatorA:decomp}, inverting the operator $\mathcal{T}_{0}$ and denoting by
\begin{equation}\label{varios:def:operators}
 \textsf{T}_{\kappa}=\mathcal{T}_0^{-1}\mathcal{T}_{\kappa}, \quad \textsf{G}_{\ell}=\mathcal{T}_0^{-1}\tilde{g}_{\ell}, \quad \mathsf{S}_{\ell}= \mathcal{T}_0^{-1}\mathcal{S}_{\ell},
 \end{equation}
for $\kappa=1,\ldots, 4$, $\ell=1,2,$ we obtain that \eqref{prim1:new} and \eqref{prim2:new} read
\begin{align}
 \label{alm1} \textsf{G}_1-\frac{J_2}{\pi L^2}+\mathsf{S}_{1}&=\mathcal{R}_x\mathcal{R}_y\sum_{\kappa=1}^4\textsf{T}_{\kappa} j_0^1+\left(\mathcal{R}^{2}_y-\text{id}\right)\sum_{\kappa=1}^4\textsf{T}_{\kappa} j_0^2+\mathcal{R}_x\mathcal{R}_y j_0^1+\left(\mathcal{R}^{2}_y-\text{id}\right)j_0^2  \nonumber \\
 &\quad \quad +\sum_{\xi\in\Z^2_{*}}\frac{in}{|\xi|^2}(\mathfrak{m}(\xi))^{-1}\reallywidehat{j}^3_0(\xi)\foul,  \\
  \label{alm2}  \textsf{G}_2-\frac{J_1}{\pi L^2}+\mathsf{S}_{2}=&-\left(\mathcal{R}^{2}_x-\text{id}\right)\sum_{\kappa=1}^4\textsf{T}_{\kappa} j_0^1-\mathcal{R}_x\mathcal{R}_y\sum_{\kappa=1}^4\textsf{T}_{\kappa} j_0^2-(\mathcal{R}^{2}_x-\text{id})j_0^1-\mathcal{R}_x\mathcal{R}_yj_0^2\nonumber \\
&\quad \quad -\sum_{\xi\in\Z^2_{*}}\frac{im}{|\xi|^2}(\mathfrak{m}(\xi))^{-1}\reallywidehat{j}^3_0(\xi)\foul. 
\end{align}
The resulting integral equations \eqref{alm1}-\eqref{alm2} resemble equations \eqref{preintegrallineal1}-\eqref{preintegrallineal2} in Section \ref{sec:linearized}. The main difference is that in this nonlinear setting we have an extra perturbation we will have to estimate and control later on. Furthermore, similarly as in Section \ref{sec:linearized}, we need to include the div-free condition to obtain a well-defined system of equations (three unknowns and three equations). Using Proposition \ref{transport:div:prop}, and recalling that $B$ is divergence free, it is enough to impose that $\nabla\cdot j=0$ on $\partial\Omega_{-}$, i.e. $\nabla\cdot j_0=0$. Furthermore, $\nabla\cdot j_{0}= \partial_1 j^1_{0}+\partial_2 j^2_{0}+\partial_3 j^3_{0}$, so we can use the transport equation \eqref{transport:modified} to write $\partial_{3}j^{3}_{0}$ in terms of $B$ and the tangential derivatives of $j_{0}$, resulting in 
\[
0=\dive j_{0}=\partial_1 j^1_{0}+\partial_2 j^2_{0}+A(b)_{3\ell}j^\ell_{0}-\frac{b_1}{1+b_3}\partial_1 j^3_{0}-\frac{b_2}{1+b_3}\partial_2 j^3 _{0},
\]
and hence
\begin{equation}\label{label:expansion1}
(1+b_3)\partial_1j^1_{0}+(1+b_3)\partial_2 j^2_{0}+(1+b_3)A(b)_{3\ell}j^{\ell}_{0}-b_1\partial_1j^3_{0}-b_2\partial_2j^3_{0}=0,  \quad \ell=1,2,3,
\end{equation}
where the matrix $A(b)$ is given in \eqref{transport:matrix:modified}. Here, notation $A(b)_{3\ell}j^{\ell}_{0}$ means summation over repeated indices, namely
\[ A(b)_{3\ell}j^{\ell}_{0}= A(b)_{31}j^{1}_{0}+ A(b)_{32}j^{2}_{0}+A(b)_{33}j^{3}_{0}.\]
Recalling \eqref{expression:matrix:A} we can simplify expression \eqref{label:expansion1} to find that
\begin{equation}\label{label:expansion2}
(1+b_3)\partial_1j^1_{0}+(1+b_3)\partial_2 j^2_{0}+\partial_{\ell}b_{3}j^{\ell}_{0}-b_1\partial_1j^3_{0}-b_2\partial_2j^3_{0}=0,  \quad \ell=1,2,3,
\end{equation}

By means of the previous computations, we infer that
\begin{align}
 \label{alm:new1} \textsf{G}_1-\frac{J_2}{\pi L^2}+\mathsf{S}_{1}-\mathsf{H}_{1}&=\mathcal{R}_x\mathcal{R}_y\sum_{\kappa=1}^4\textsf{T}_{\kappa} j_0^1+\left(\mathcal{R}^{2}_y-\text{id}\right)\sum_{\kappa=1}^4\textsf{T}_{\kappa} j_0^2-j_0^2 \nonumber \\
& \quad \quad +\sum_{\xi\in\Z^2_{*}}\frac{in}{|\xi|^2}(\mathfrak{m}(\xi))^{-1}\reallywidehat{j}^3_0(\xi)\foul,  \\
  \label{alm:new2}  \textsf{G}_2-\frac{J_1}{\pi L^2}+\mathsf{S}_{2}+\mathsf{H}_{2}=&-\left(\mathcal{R}_x^{2}-\text{id}\right)\sum_{\kappa=1}^4\textsf{T}_{\kappa} j_0^1-\mathcal{R}_x\mathcal{R}_y\sum_{\kappa=1}^4\textsf{T}_{\kappa} j_0^2+j_0^{1} \nonumber \\
 & \quad \quad -\sum_{\xi\in\Z^2_{*}}\frac{im}{|\xi|^2}(\mathfrak{m}(\xi))^{-1} \reallywidehat{j}^3_0(\xi)\foul. 
\end{align}
where
\begin{align}
\mathsf{H}_{1}&=\mathcal{B}_y \left(b_1\partial_1j^3_0+b_2\partial_2j^3_0-\partial_{\ell}b_{3}j^\ell_0-b_3\partial_1j_0^1-b_3\partial_2j_0^2\right), \\
\mathsf{H}_{2}&=\mathcal{B}_x \left(b_1\partial_1j^3_0+b_2\partial_2j^3_0-\partial_{\ell}b_{3}j^\ell_0-b_3\partial_1j_0^1-b_3\partial_2j_0^2\right),
\end{align}
and $\mathcal{B}_x, \mathcal{B}_y$ are defined via the multipliers 
\[ \widehat{\mathcal{B}_{x}f}(\xi)=-i\frac{m}{|\xi|^{2}}\widehat{f}(\xi), \ \widehat{\mathcal{B}_{y}f}(\xi)=-i\frac{n}{|\xi|^{2}}\widehat{f}(\xi).\]


Furthermore, the expression for $j_0^{3}$ can be directly computed by means of equations \eqref{alm:new1} and \eqref{alm:new2}. Indeed, we have that
\begin{equation}\label{alm:new3}
\sum_{\xi\in\Z^2_{*}}(\mathfrak{m}(\xi))^{-1}\reallywidehat{j}_0^3e^{ir\cdot\xi}=\partial_1 \textsf{G}_2-\partial_2 \textsf{G}_1, 
\end{equation}
and hence since by the definition of the functions $\textsf{G}_1$ and $\textsf{G}_2$ in \eqref{varios:def:operators} we find that
\begin{equation}\label{alm:new3:final}
 j_0^{3}(r)=\partial_1 g_2(r)-\partial_2 g_1(r).
\end{equation} 

%
Combining \eqref{alm:new1}-\eqref{alm:new3} we conclude that the first two components of the current are explicitly given by
\begin{align}
    j_0^1&=(\mathcal{R}^{2}_x-\text{id})\sum_{\kappa=1}^4\textsf{T}_{\kappa} j_0^1+\mathcal{R}_x\mathcal{R}_y\sum_{\kappa=1}^4\textsf{T}_{\kappa} j_0^2 + \mathsf{S}_{2}-\mathsf{H}_{2}+\left(\text{id}-\mathcal{R}_x^{2}\right)\textsf{G}_2+\mathcal{R}_x\mathcal{R}_y\textsf{G}_1-\frac{J_1}{\pi L^2}, \label{alm:new4} \\
    j_0^2&=\left(\mathcal{R}^{2}_y-\text{id}\right)\sum_{\kappa=1}^4\textsf{T}_\kappa j_0^2+\mathcal{R}_x\mathcal{R}_y\sum_{\kappa=1}^4\textsf{T}_\kappa j^1_0- \mathsf{S}_{1}+\mathsf{H}_{1}  -\left(\text{id}-\mathcal{R}_y^{2}\right)\textsf{G}_1-\mathcal{R}_x\mathcal{R}_y\textsf{G}_2+\frac{J_2}{\pi L^2}. \label{alm:new5}
\end{align}
To conclude the formal derivation of the integral equation $j_0$ we have to determine the values of the constants $J_1, J_2 \in \mathbb{R}$. As in the linearized case, the values $J_1,J_2$ are crucial to ensure that  the constructed pressure $p$ is a uni-valued function. We take the constants to be 
\begin{align}
\frac{J_2}{\pi L^2}&=\sum_{\kappa=1}^{4}\left\langle \textsf{T}_\kappa j_0^2\right\rangle + \left\langle \textsf{A}(\delta j^2)\right\rangle +\langle\textsf{G}_1\rangle-\langle j_0^2 f\rangle +\langle g_2(\partial g_2-\partial_2 g_1)\rangle. \label{J1}\\
\frac{J_1}{\pi L^2}&-\sum_{\kappa=1}^{4}\left\langle \textsf{T}_\kappa j_0^2\right\rangle - \left\langle \textsf{A}(\delta j^2)\right\rangle +\langle\textsf{G}_1\rangle-\langle j_0^2 f\rangle +\langle g_2(\partial g_2-\partial_2 g_1)\rangle. \label{J2}
\end{align}
The reason of this choice will be clarified in Section \ref{sec:6}. All in all, we conclude that the system of integral equations we want to solve is formally given by the expression
\begin{align}
    j_0^1&=(\mathcal{R}^{2}_x-\text{id})\sum_{\kappa=1}^4 \widetilde{\textsf{T}_{\kappa} j_0^1}+\mathcal{R}_x\mathcal{R}_y\sum_{\kappa=1}^4\textsf{T}_{\kappa} j_0^2 + \mathsf{S}_{2}[\widetilde{\delta j^{1}},\delta j^{2}]-\mathsf{H}_{2} +\left(\text{id}-\mathcal{R}_x^{2}\right)\widetilde{\textsf{G}_2}+\mathcal{R}_x\mathcal{R}_y\textsf{G}_1 \nonumber \\
    &\quad \quad -\langle(\partial_1 g_2-\partial_2 g_1)g_1\rangle-\langle j^1_0 f\rangle, \label{alm:new4:final} \\
    j_0^2&=\left(\mathcal{R}^{2}_y-\text{id}\right)\sum_{\kappa=1}^4\widetilde{\textsf{T}_\kappa j_0^2}+\mathcal{R}_x\mathcal{R}_y\sum_{\kappa=1}^4\textsf{T}_\kappa j^1_0- \mathsf{S}_{1}[\delta j^{1},\widetilde{\delta j^{2}}]+\mathsf{H}_{1}  -\left(\text{id}-\mathcal{R}_y^{2}\right)\widetilde{\textsf{G}_1}-\mathcal{R}_x\mathcal{R}_y\textsf{G}_2  \nonumber \\
    &\quad \quad  -\langle(\partial_1 g_2-\partial_2 g_1)g_2\rangle-\langle j^2_0 f\rangle.  \label{alm:new5:final}
\end{align}
We recall that the overline notation $\widetilde{v}$ is defined as $\widetilde{v}=v-\langle v\rangle$.

\section{Rigorous derivation of integral current equation and a priori estimates}\label{sec:5}
In this section, we provide a rigorous approach to the formal computations introduced in Section \ref{sec:4}. More precisely, we use the pseudo-differential approach in H\"older spaces (mainly contained in Subsection \ref{subsec:21}-Subsection \ref{subsec:23}) to show that the operators  $\textsf{T}_\kappa$ in \eqref{alm:new4}-\eqref{alm:new5} are well defined bounded operators in H\"older spaces, cf. Theorem \ref{teoremon1}. We will also show estimates for the differences of those operators in lower order norms, cf. Proposition \ref{diferencia}. Furthermore, we will also provide several estimates for the remaining operators $\mathsf{S}_{\ell}, \mathsf{H}_{\ell}$ and $\mathsf{G}_{\ell}$ for $\ell=1,2,$ which will be needed in Section \ref{sec:6} to solve the integral equation \eqref{alm:new4:final}-\eqref{alm:new5:final} via Neumann series. These bounds are contained in Proposition \ref{prop:estimate:S} , Proposition \ref{Prop:H:estimadas:ori} and Corollary \ref{coro:estimate:G}, respectively.

\subsection{H\"older estimates for the operators $\textsf{T}_\kappa$}\label{subsec:51}
Let us start this subsection by showing that the operators $\textsf{T}_\kappa$ are well-defined operators and that $\mathsf{T}_{\kappa}\in \mathcal{L}(C^{1,\alpha}(\mathbb{T}^{2}))$, $\kappa=1,\ldots, 4$.  To that purpose, we recall their precise expression: 
\begin{align}
    \textsf{T}_1 j_{0}^{\ell}(r)&=\int \sum_{\xi\in\Z^2}e^{i(r-\eta)\cdot\xi}(\mathfrak{m}(\xi))^{-1}\left(\int_0^Le^{-|\xi|s}\left(\frac{e^{-i\xi\cdot\Lambda(\eta,s)}-1}{1+\Theta(\eta,s)}\right)ds\right)j_{0}^{\ell}(\eta)d\eta,\label{T1j0} \\
    \textsf{T}_2 j_{0}^{\ell}(r)&=\int \sum_{\xi\in\Z^2}e^{i(r-\eta)\cdot\xi}(\mathfrak{m}(\xi))^{-1}\left(\int_0^Le^{-|\xi|s}\left(\frac{\Theta(\eta,s)}{1+\Theta(\eta,s)}\right)ds\right)j_{0}^{\ell}(\eta)d\eta,\\
    \textsf{T}_3 j_{0}^{\ell}(r)&=\int \sum_{\xi\in\Z^2}e^{i(r-\eta)\cdot\xi}(\mathfrak{m}(\xi))^{-1}\left(\int_0^LM(\xi,s)\left(\frac{e^{-i\xi\cdot\Lambda(\eta,s)}-1}{1+\Theta(\eta,s)}\right)ds\right)j_{0}^{\ell}(\eta)d\eta,\\
    \textsf{T}_4 j_{0}^{\ell}(r)&=\int \sum_{\xi\in\Z^2}e^{i(r-\eta)\cdot\xi}(\mathfrak{m}(\xi))^{-1}\left(\int_0^LM(\xi,s)\left(\frac{\Theta(\eta,s)}{1+\Theta(\eta,s)}\right)ds\right)j_{0}^{\ell}(\eta)d\eta, \label{T4j0} 
\end{align}
for $\ell=1,2,$ and where we recall that $(\mathfrak{m}(\xi))^{-1}=\frac{|\xi|\sinh(|\xi|L)}{\cosh(|\xi|L)-1}$. In order show that the operators $\mathsf{T}_{\kappa}$ for $\kappa=1,\ldots,4 $ are well-defined, the main idea is to define them as the adjoint of a pseudo-differential operator with associated symbols
\begin{align}
    a_1(x,\xi)&=(\mathfrak{m}(\xi))^{-1}\int_0^Le^{-|\xi|s}\frac{\left(e^{-i\xi\cdot\Lambda(x,s)}-1\right)}{1+\Theta(x,s)}ds,\label{symbol:a1} \\
    a_2(x,\xi)&=(\mathfrak{m}(\xi))^{-1}\int_0^Le^{-|\xi|s}\frac{\Theta(x,s)}{1+\Theta(x,s)}ds,   \label{symbol:a2} \\
        a_3(x,\xi)&=(\mathfrak{m}(\xi))^{-1}\int_0^{L} M(\xi,s) e^{-|\xi|s}\frac{\left(e^{-i\xi\cdot\Lambda(x,s)}-1\right)}{1+\Theta(x,s)}ds, \label{symbol:a3}\\ 
    a_{4}(x,\xi)&= (\mathfrak{m}(\xi))^{-1}\int_0^L M(\xi,s) e^{-|\xi|s}\frac{\Theta(x,s)}{1+\Theta(x,s)}ds.\label{symbol:a4}
\end{align}
Indeed, notice that formally
\[ \langle \phi, \mathsf{T}_{\kappa}j_0^{\ell}\rangle= \langle  \p(a)_{\kappa}\phi, j_0^{\ell}\rangle, \]
and hence due to Theorem \ref{solucion}, if we show that the symbols \eqref{symbol:a1}-\eqref{symbol:a4} belong to the symbol class $S^{0}(1+\alpha)$ the operators $\mathsf{T}_{\kappa}$ are well-defined and bounded in the corresponding H\"older spaces, cf. Theorem \ref{solucion} for precise statement. Notice that we are in the situation described in Remark \ref{observacionperiodica}, so we just need to prove that the obvious extensions of the symbols $a_\kappa(x,\xi)$,  $\kappa=1,\ldots, 4$,  satisfy estimates \eqref{constantes}.

In order to simplify the exposition, we define the symbols
\begin{align}
 b_{1}(x,\xi)&= |\xi|\int_0^Le^{-|\xi|s}\frac{\left(e^{i\xi\cdot\Lambda(x,s)}-1\right)}{1+\Theta(x,s)}ds,  \label{s1} \\
 b_{2}(x,\xi)&=|\xi|\int_0^Le^{-|\xi|s}\frac{\Theta(x,s)}{1+\Theta(x,s)}ds. \label{s2}
 \end{align}
Therefore, it is easy to check that
\[ a_{1}(x,\xi)=b_{1}(x,\xi)+ \mathsf{r}_{1}(x,\xi), \quad  a_{2}(x,\xi)= b_{2}(x,\xi)+ \mathsf{r}_{2}(x,\xi), \]
where $\mathsf{r}_{1}(x,\xi), \mathsf{r}_{2}(x,\xi)$ are smoothing remainders. Therefore, we will directly work with the new symbols $b_{1}(x,\xi), b_{2}(x,\xi)$ and later also show that the same estimates and conclusions trivially hold for the remainder terms $\mathsf{r}_{1}(x,\xi), \mathsf{r}_{2}(x,\xi)$, and hence for $a_{1}(x,\xi), a_{2}(x,\xi)$. The fast decay of the function $M(\xi,s)$ defined in \eqref{definition:M:smoothing}, provides that the kernels defined out of $a_{3}(x,\xi), a_{4}(x,\xi)$ are smooth and as a consequence the operators $\textsf{T}_3$ and $\textsf{T}_4$ are smoothing operators. This fact will be shown in Proposition \ref{prop:a3:a4}.

Furthermore, we consider that the functions $\Lambda$ and $\Theta$ are generic functions that satisfy the following two requirements: 
\begin{asu}\label{asu1} The function $\Lambda:\Omega\to \mathbb{T}^{2}$ belongs to $C^{2,\alpha}(\Omega)$ and satisfies that $\Lambda(x,0)=0$. Moreover, we have that there exists $\delta_{0}\in (0,\frac{1}{2})$ such that
\[ \norm{\Lambda}_{C^{2,\alpha}(\Omega)} \leq \delta_{0}. \] 
\end{asu}
\begin{asu}\label{asu2}
 The function $\Theta: \Omega \to \mathbb{T}^{2}$ belongs to $C^{1,\alpha}(\Omega)$ and satisfies that
  \[ \norm{\Theta}_{C^{1,\alpha}(\Omega)} \leq \delta_{1}, \mbox{ for } \delta_{1}\in (0,\frac{1}{2}). \]
  \end{asu}
We will see later that indeed the functions $\Lambda$ and $\Theta$ do indeed satisfy these requirements, cf. Proposition \ref{estimations:Lambda:Theta}. In the following result, we show that under Assumptions \ref{asu1} and \ref{asu2}, the symbols
\eqref{s1}-\eqref{s2} extend to symbols in  $S^{0}(1+\alpha)$, cf. Definition \ref{symbol:class:def}. If so, a direct application of Theorem \ref{solucion}, yields that the associated operators $\mathsf{T}_{1}, \mathsf{T}_{2}$ are well-defined bounded operators from $C^{2,\alpha}$ to $C^{2,\alpha}$.

\begin{prop}\label{propb1:b2}
 Let Assumptions \ref{asu1} and \ref{asu2} hold. Then, the symbols $b_{1},b_{2}$ given in \eqref{s1}-\eqref{s2} respectively, extend to symbols in $S^0(1+\alpha)$.
\end{prop}
\begin{proof}[Proof of Proposition \ref{propb1:b2}]
Let us first demonstrate the result for the symbol $b_{1}(x,\xi)$ defined in \eqref{s1}. To that purpose, recalling that the symbol $\xi$ is a Mihklin-H\"ormander multiplier of degree 1, we have that 
\[b_{1}\in S^{0}(1+\alpha)\Longleftrightarrow \frac{b_{1}}{|\xi|} \in S^{-1}(1+\alpha).\]
In order to show that claim, we begin by pointing out that for any $\gamma\in \mathbb{N}^2$ multi-index,
\begin{equation}\label{eq:exp:1}
\partial^\gamma_\xi e^{-|\xi|s}=e^{-|\xi|s}\sum_{\kappa}s^{k^\gamma_\kappa}p^\gamma_\kappa (\xi),
\end{equation}
with $k_\kappa^\gamma \leq |\gamma|$ and $p^\gamma_\kappa(\xi)\in C^\infty(\R^2\setminus \{0\})$ satisfying
\begin{equation}\label{eq:exp:2}
 \abs{\partial^\gamma_{\xi} p^\gamma_\kappa(\xi)}\lesssim |\xi|^{k^\gamma_\kappa-|\gamma|}.
 \end{equation}
 We remark that the summation in \eqref{eq:exp:1} is finite. For simplicity we do not precise where the summation index $\kappa$ is taken. Moreover, for $\gamma, \beta\in \mathbb{N}^{2}$ multi-index, we write $\beta <\gamma$ if the $\beta$ is less than $\gamma$ componentwise.
    
\subsubsection*{\underline{$L^{\infty}$ bounds for $\partial_{\xi}^{\gamma}\left(b_{1}(x,\xi)/|\xi|\right)$}}
Invoking Leibniz rule and recalling definition \eqref{s1} we infer that
 \begin{align}
    \partial^\gamma_\xi \left(\frac{b_1(x,\xi)}{|\xi|}\right)&=\int_0^L\partial^\gamma_\xi\left(e^{-|\xi|s}\right)\frac{\left(e^{i\xi\cdot \Lambda(x,s)}-1\right)}{1+\Theta(x,s)}ds +\displaystyle\sum_{\substack{\beta\leq\gamma\\ \beta\neq \gamma}}{\gamma\choose\beta} \int_0^L\partial^\beta_\xi \left(e^{-|\xi|s}\right)\partial_\xi^{\gamma-\beta}\left(e^{i\Lambda(x,s)\cdot\xi}-1\right)\frac{ds}{1+\Theta(x,s)} \nonumber  \\
    &=b_{11}(x,\xi)+b_{12}(x,\xi). \label{structure:same}
    \end{align}

Taking into account Assumption \ref{asu1} we find that 
\begin{equation}\label{eq:cota:exp1}
 \left|e^{i\Lambda(x,s)\cdot \xi}-1\right|=\left|\int_0^si\xi\cdot\partial_w\Lambda(x,w)e^{-i\xi\cdot \Lambda(x,w)}dw\right|\leq s |\xi| \|\Lambda\|_{C^1}.
\end{equation}
Combining \eqref{eq:exp:1}-\eqref{eq:cota:exp1}, we obtain that
  \begin{align}\left|b_{11}(x,\xi)\right|\leq \frac{\|\Lambda\|_{C^1}}{1-\|\Theta\|_{\infty}}\sum_{\kappa}\int_0^L|\xi|^{1-|\gamma|+k^\gamma_\kappa}s^{k^\gamma_\kappa+1}e^{-|\xi|s}ds&\leq \frac{\|\Lambda\|_{C^1}}{1-\|\Theta\|_{\infty}}|\xi|^{-|\gamma|-1}\displaystyle\sum_{\kappa_\gamma}\int_0^\infty u^{k^\gamma_\kappa+1}e^{-u}du \nonumber \\
        &\lesssim \frac{\|\Lambda\|_{C^1}}{1-\|\Theta\|_{\infty}}|\xi|^{-|\gamma|-1}. \label{cota:b11:linfty}
        \end{align}
In order to bound $b_{12}(x,\xi)$, we first notice that
\begin{equation}\label{eq:exp:3}
        \partial^\beta_{\xi} \left(e^{-|\xi|s}\right)\partial_\xi^{\gamma-\beta}\left(e^{i\Lambda(x,s)\cdot\xi}-1\right)\\
        =e^{-|\xi|s}\sum_{\kappa}s^{k^\beta_\kappa}p^\beta_\kappa(\xi)\left(i\Lambda(x,s)\right)^{\gamma-\beta}e^{i\Lambda(x,s)\cdot\xi}.
        \end{equation}
Using \eqref{eq:exp:3}, the fact that $\beta\neq\gamma$ and Assumption \ref{asu1} we find that
\begin{align}
        \abs{b_{12}(x,\xi)} \leq \frac{\|\Lambda\|^{|\gamma-\beta|}_{C^1}}{1-\|\Theta\|_{\infty}}\sum_{\kappa}\int_0^L|\xi|^{k_\kappa^\beta-|\beta|}s^{k^\beta_\kappa+|\gamma-\beta|}e^{-|\xi|s}ds&\leq \frac{\|\Lambda\|^{|\gamma-\beta|}_{C^1}}{1-\|\Theta\|_{\infty}}|\xi|^{-|\gamma|-1}\sum_{\kappa}\int_0^{\infty}u^{k^\beta_\kappa+|\gamma-\beta|}e^{-u}du \nonumber \\
        &\lesssim \frac{\|\Lambda\|^{|\gamma-\beta|}_{C^1}}{1-\|\Theta\|_{\infty}}|\xi|^{-|\gamma|-1}.\label{cota:b12:linfty}
    \end{align}
Hence, bounds \eqref{cota:b11:linfty}-\eqref{cota:b12:linfty} yield the decay
\[ \left|\partial^\gamma_\xi \left(\frac{b_1(x,\xi)}{|\xi|}\right)\right| \lesssim \bigg( \frac{\|\Lambda\|_{C^1}}{1-\|\Theta\|_{\infty}}+ \sum_{\substack{\beta\leq\gamma\\ \beta\neq \gamma}}\frac{\|\Lambda\|^{|\gamma-\beta|}_{C^1}}{1-\|\Theta\|_{\infty}} \bigg) |\xi|^{-|\gamma|-1}\leq C\|\Lambda\|_{C^1}|\xi|^{-|\gamma|-1}, \ \forall \gamma\in\mathbb{N}^{2} \ \mbox{multi-index}.  \]

\subsubsection*{\underline{$L^{\infty}$ bounds for $\partial_{\xi}^{\gamma}\nabla_{x}\left(b_{1}(x,\xi)/|\xi|\right)$}} In order to obtain $L^{\infty}$ estimates we compute  
 \begin{align}\label{derivada}
    \nabla_{x}\left(\frac{b_1(x,\xi)}{|\xi|}\right)&=\int_0^Le^{-|\xi|s}\,i\nabla_x\Lambda\cdot\xi\frac{e^{i\Lambda(x,s)\cdot\xi}}{1+\Theta(x,s)}ds -\int_0^Le^{-|\xi|s}\frac{\left(e^{i\Lambda(x,s)\cdot\xi}-1\right)}{(1+\Theta(x,s))^2}\nabla_x\Theta(x,s)ds.
    \end{align}
Combining \eqref{derivada} with Leibniz rule we infer that
\[\partial^\gamma_\xi \nabla_{x}\left(\frac{b_1(x,\xi)}{|\xi|}\right)=b_{13}(x,\xi)-b_{14}(x,\xi),\]   
    where 
   \begin{align} 
  b_{13}(x,\xi)&=\int_0^L\partial^\gamma_\xi \left(e^{-|\xi|s}\right)i\nabla_{x}\Lambda(x,s)\cdot \xi e^{i\Lambda(x,s)\cdot\xi} \frac{ds}{1+\Theta(x,s)} \nonumber \\ 
   & \quad \quad + \displaystyle\sum_{\substack{\beta\leq\gamma\\ \beta\neq \gamma}}{\gamma\choose\beta} \int_0^L\partial^\beta_\xi \left(e^{-|\xi|s}\right)i\nabla_{x}\Lambda(x,s)\cdot \partial_\xi^{\gamma-\beta}\left(\xi e^{i\Lambda(x,s)\cdot\xi}\right) \frac{ds}{1+\Theta(x,s)},\label{summand:b13}  \\ 
 b_{14}(x,\xi)&=\int_0^L\partial^\gamma_\xi\left(e^{-|\xi|s}\right) \left(e^{i\xi\cdot \Lambda(x,s)}-1\right) \frac{\nabla_{x}\Theta(x,s)}{(1+\Theta(x,s))^{2}}  \ ds  \nonumber \\
 & \quad \quad +\displaystyle\sum_{\substack{\beta\leq\gamma\\ \beta\neq \gamma}}{\gamma\choose\beta} \int_0^L\partial^\beta_\xi \left(e^{-|\xi|s}\right)\partial_\xi^{\gamma-\beta}\left(e^{i\Lambda(x,s)\cdot\xi}-1\right)\frac{\nabla_{x}\Theta(x,s)}{(1+\Theta(x,s))^{2}} \ ds. \label{summand:b14}
\end{align}
The term $b_{14}(x,\xi)$ has the same structure as \eqref{structure:same} and hence performing the same estimates we find that
\begin{equation}\label{b14:estimate:1}
\abs{b_{14}(x,\xi)} \lesssim \left( \frac{\norm{\Lambda}_{C^1}\norm{\Theta}_{C^1}}{\left(1-\norm{\Theta}_{L^\infty}\right)^{2}}+\sum_{\substack{\beta\leq\gamma\\ \beta\neq \gamma}}\frac{\norm{\Lambda}_{C^1}^{|\gamma-\beta|}\norm{\Theta}_{C^1}}{\left(1-\norm{\Theta}_{L^\infty}\right)^{2}}\right)\abs{\xi}^{-|\gamma|-1}\leq \|\Lambda\|_{C^1}|\xi|^{-1-|\gamma|}.
\end{equation}
 Using \eqref{eq:exp:3} and the cancellation property in Assumption \ref{asu1}, namely, $\nabla_{x}\Lambda(x,0)=0$ we can compute similarly that 
 \begin{equation}\label{b13:estimate:1}
 \abs{b_{13}(x,\xi)} \lesssim \left( \frac{\|\Lambda\|_{C^2}\|\Lambda\|_{C^1}}{1-\|\Theta\|_\infty}+ \sum_{\substack{\beta\leq\gamma\\ \beta\neq \gamma}} \frac{\|\Lambda\|_{C^2}\|\Lambda\|_{C^1}^{|\gamma-\beta|}}{1-\|\Theta\|_\infty} \right)|\xi|^{-|\gamma|-1}\leq C\|\Lambda\|_{C^2}|\xi|^{-1-|\gamma|}.
 \end{equation}
 Therefore, combining \eqref{b14:estimate:1}-\eqref{b13:estimate:1}, we have that
 \begin{align}
\left|\partial^\gamma_\xi \nabla_{x}\left(\frac{b_1(x,\xi)}{|\xi|}\right)\right|
&\lesssim \left( \frac{\|\Lambda\|_{C^2}\|\Lambda\|_{C^1}}{1-\|\Theta\|_\infty}+  \sum_{\substack{\beta\leq\gamma\\ \beta\neq \gamma}}\frac{\|\Lambda\|_{C^2}\|\Lambda\|_{C^1}^{|\gamma-\beta|}}{1-\|\Theta\|_\infty} \right)|\xi|^{-|\gamma|-1} \leq C\|\Lambda\|_{C^2}|\xi|^{-1-|\gamma|} \nonumber  \\
&\quad + \left( \frac{\norm{\Lambda}_{C^1}\norm{\Theta}_{C^1}}{\left(1-\norm{\Theta}_{L^\infty}\right)^{2}}+\sum_{\substack{\beta\leq\gamma\\ \beta\neq \gamma}}\frac{\norm{\Lambda}_{C^1}^{|\gamma-\beta|}\norm{\Theta}_{C^1}}{\left(1-\norm{\Theta}_{L^\infty}\right)^{2}}\right)\abs{\xi}^{-|\gamma|-1}\leq C\|\Lambda\|_{C^2}|\xi|^{-1-|\gamma|}.
 \end{align}
\subsubsection*{\underline{$C^{\alpha}$ bounds for $\partial_{\xi}^{\gamma}\nabla_{x}\left(b_{1}(x,\xi)/|\xi|\right)$}}
To finally show that the symbol $b_{1}(x,\xi)\in S^{0}(1+\alpha)$, we have to prove that the $C^{\alpha}$ norms for the terms \eqref{summand:b13} and \eqref{summand:b14} have the desired decay. More precisely, we will obtain the following bounds
\begin{align}
\|b_{13}(\cdot,\xi)\|_{C^{1,\alpha}}&\leq C\|\Lambda\|_{C^{2,\alpha}}|\xi|^{-1-|\alpha|}, \label{bound:calpha:b13} \\
\|b_{14}(\cdot,\xi)\|_{C^{1,\alpha}}&\leq C\|\Lambda\|_{C^{2,\alpha}}|\xi|^{-1-|\alpha|}. \label{bound:calpha:b14} 
\end{align}
Through the derivation of the bounds \eqref{bound:calpha:b13}-\eqref{bound:calpha:b14} we will repeatedly employ the elementary inequality
\begin{equation}\label{inequality:calpha}
[fgh]_{\alpha}\leq [f]_\alpha\|g\|_\infty\|h\|_\infty+\|f\|_\infty[g]_\alpha\|h\|_\infty+\|f\|_\infty\|g\|_\infty[h]_\alpha.
\end{equation}
Moreover, we also have that
\[ \left[\int_0^L f(\cdot,\xi,s) \ ds \right]_\alpha\leq \int_0^L\left[f(\cdot,\xi,s\right)]_\alpha \ ds.\]
Therefore, in order to bound the $C^\alpha$ semi-norm of \eqref{summand:b13} and \eqref{summand:b14} we have to estimate
\begin{align} 
  [b_{13}(\cdot,\xi)]_{\alpha}
&\leq C \int_0^L\partial^\gamma_\xi \left(e^{-|\xi|s}\right)\left[\nabla_{x}\Lambda(\cdot,s)\cdot \xi e^{i\Lambda(\cdot,s)\cdot\xi} \frac{1}{1+\Theta(\cdot,s)}\right]_\alpha\, ds \nonumber \\ 
   & \quad + \displaystyle\sum_{\substack{\beta\leq\gamma\\ \beta\neq \gamma}}{\gamma\choose\beta} \int_0^L\partial^\beta_\xi \left(e^{-|\xi|s}\right)\left[\nabla_{x}\Lambda(\cdot,s)\cdot \partial_\xi^{\gamma-\beta}\left(\xi e^{i\Lambda(\cdot,s)\cdot\xi}\right) \frac{1}{1+\Theta(\cdot,s)}\right]_\alpha\,ds \nonumber  \\
  &\quad: =b_{131}+b_{132},\label{2summand:b13}  \\ 
 [b_{14}(\cdot,\xi)]_{\alpha}&\leq C\int_0^L\partial^\gamma_\xi\left(e^{-|\xi|s}\right) \left[\frac{\left(e^{i\xi\cdot \Lambda(\cdot,s)}-1\right)}{(1+\Theta(\cdot,s))^{2}} \nabla_{x}\Theta(\cdot,s)\right]_{\alpha}\, ds  \nonumber \\
 & \quad \quad +\displaystyle\sum_{\substack{\beta\leq\gamma\\ \beta\neq \gamma}}{\gamma\choose\beta} \int_0^L\partial^\beta_\xi \left(e^{-|\xi|s}\right)\left[\partial_\xi^{\alpha-\beta}\left(e^{i\Lambda(\cdot,s)\cdot\xi}-1\right)\frac{\nabla_{x}\Theta(\cdot,s)}{(1+\Theta(\cdot,s))^{2}}\right]_{\alpha}\, ds \nonumber\\
&\quad: =b_{141}+b_{142}. \label{2summand:b14}
\end{align}
Let us analyze the different $C^{\alpha}$ semi-norms of the quantities involved in both terms. Let us start with the ones appearing in $b_{13}$,  namely, $[\nabla_x \Lambda(x,s)]_{\alpha},[e^{i\Lambda(x,s)\cdot \xi}]_{\alpha}$ and $\left[\frac{1}{1+\Theta(x,s)}\right]_{\alpha}$.  \\

\paragraph{\underline{The $[\nabla_x \Lambda(x,s)]_{\alpha}$ bound}} Using the fundamental theorem of calculus and Assumption \ref{asu1}, we readily check that 
\[ \left|\nabla_x\Lambda(x,s)\right|=\left|\int_0^s\partial_\tau\nabla_x\Lambda(x,\tau)\,d\tau \right|\leq s\|\Lambda\|_{C^2}.\]
which yields the an estimate for the $L^\infty$ norm. Similarly, we find that
\begin{align}
\left[\nabla_x\Lambda(\cdot,s)\right]_\alpha&=\sup_{x_1\neq x_2}\frac{|\nabla_x\Lambda(x_1,s)-\nabla_x\Lambda(x_2,s)|}{|x_1-x_2|^\alpha} \nonumber \\
&\leq \sup_{x_1\neq x_2}\int_0^s\frac{|\partial_\tau\nabla_x\Lambda(x_1,\tau)-\partial_\tau\nabla_x\Lambda(x_2,\tau)|}{|x_1-x_2|^\alpha}\,d\tau\leq s\|\Lambda\|_{C^{2,\alpha}}.\label{estimate:lambda:calpha}
\end{align}
\paragraph{\underline{The $[e^{i\Lambda(x,s)\cdot \xi}]_{\alpha}$ bound}} Trivially, the $L^{\infty}$ estimate equals one. For the $C^\alpha$ semi-norm, we use again Assumption \ref{asu1}, leading to
\begin{align}
\left[e^{i\Lambda(\cdot,s)\cdot\xi}\right]_\alpha =\sup_{x_1\neq x_2}\frac{\left|e^{i\Lambda(x_1,s)\cdot\xi}-e^{i\Lambda(x_2,s)\cdot\xi}\right|}{|x_1-x_2|^\alpha}&\leq \sup_{x_1\neq x_2}\frac{1}{|x_1-x_2|^\alpha}\int_{\xi\cdot \Lambda(x_1,s)}^{\xi\cdot \Lambda(x_2,s)}\left|ie^{iw}\right|\,dw\nonumber  \\
&\leq |\xi|\sup_{x_1\neq x_2}\frac{|\Lambda(x_1,s)-\Lambda(x_2,s)|}{|x_1-x_2|^\alpha} \nonumber \\ 
&\leq |\xi|\sup_{x_1\neq x_2}\frac{1}{|x_1-x_2|^\alpha}\int_{0}^s\left|\partial_\tau \left(\Lambda(x_1,\tau)-\Lambda(x_2,\tau)\right)\right|\,ds \nonumber\\
&\leq s|\xi|\|\Lambda\|_{C^{1,\alpha}}.\label{estimacionalphaexponencial}
\end{align}
\paragraph{\underline{The $[\frac{1}{1+\Theta(x,s)}]_{\alpha}$ bound}} Invoking Assumption \ref{asu2}, the $L^{\infty}$ bound follows immediately, namely,  \[ \left|\dfrac{1}{1+\Theta(x,s)}\right|\leq \dfrac{1}{1-\|\Theta\|_{\infty}}.\] The $C^\alpha$ semi-norm can be estimated after finding an extra cancellation, this is
\begin{align}
\left[\frac{1}{1+\Theta(\cdot,s)}\right]_\alpha&=\sup_{x_1\neq x_2}\frac{1}{|x_1-x_2|^\alpha}\left|\frac{1}{1+\Theta(x_1,s)}-\frac{1}{1+\Theta(x_2,s)}\right| \nonumber \\
&\leq \sup_{x_1\neq x_2}\frac{1}{\left(1+\Theta(x_1)\right)\left(1+\Theta(x_2)\right)}\frac{|\Theta(x_1,s)-\Theta(x_2,s)|}{|x_1-x_2|^\alpha} \nonumber \\
&\leq \frac{\|\Theta\|_{C^{1,\alpha}}}{(1-\|\Theta\|_{\infty})^2}. \label{estimacionalphatheta}
\end{align}
Therefore, combining estimates \eqref{estimate:lambda:calpha}-\eqref{estimacionalphatheta}
and using \eqref{eq:exp:1} we find that
\begin{equation*}
\begin{split}
b_{131}&\leq \sum_{\kappa}\int_0^L (s|\xi|)^{k_{\kappa}^\gamma}|\xi|^{-\gamma}e^{-s|\xi|}\left[\nabla_{x}\Lambda(\cdot,s)\cdot \xi e^{i\Lambda(\cdot,s)\cdot\xi} \frac{1}{1+\Theta(\cdot,s)}\right]_\alpha\, ds\\
&\leq \sum_{\kappa}\int_0^L (s|\xi|)^{k_{\kappa}^\gamma}|\xi|^{-|\gamma|}e^{-s|\xi|}\frac{|\xi|s}{1-\|\Theta\|_\infty}\left(\|\Lambda\|_{C^{2,\alpha}}+s|\xi|\|\Lambda\|_{C^2}^{2}+\|\Lambda\|_{C^2}\frac{\|\Theta\|_{C^{1,\alpha}}}{1-\|\Theta\|_\infty}\right)\,ds\\
&\leq C|\xi|^{-|\gamma|-1}\|\Lambda\|_{C^{2,\alpha}}\int_0^{|\xi|L}u^{1+k_{\kappa}^\gamma}(u+1)e^{-u}\,du \leq C|\xi|^{-|\gamma|-1}\|\Lambda\|_{C^{2,\alpha}} \label{estimate:final:b131}
\end{split}
\end{equation*}
where $C$ can be taken independent of both $\xi$, $\Lambda$ and $\Theta$. Note that we have used the smallness assumption on the norms of $\Lambda$ and $\Theta$ to absorb them by means of a universal constant.

In order to bound $b_{132}$ in \eqref{summand:b13} we can repeat the previous computations just taking into account  the extra summands that arise when expanding the term $\partial_{\xi}^{\gamma-\beta}\left( \xi e^{i\Lambda(\cdot,s)\cdot \xi}\right)$.  Invoking Leibniz's rule for  for $k=1,2$ we find that
\begin{align*}
\partial_\xi^{\gamma-\beta}\left(\xi_k e^{i\Lambda(x,s)\cdot\xi}\right)&=\sum_{\sigma\leq\gamma-\beta}{\gamma-\beta\choose \sigma}\left(\partial_\xi^\sigma \xi_k\right)\left(\partial_\xi^{\gamma-\beta-\sigma} e^{i\Lambda(x,s)\cdot\xi}\right) \nonumber \\
&\hspace{-0.2cm}=\xi_k\left(i\Lambda(x,s)\right)^{\gamma-\beta}e^{i\Lambda(x,s)\cdot\xi}+(\gamma_k-\beta_k)\textbf{1}_{1\leq \gamma_k-\beta_k}\left(i\Lambda(x,s)\right)^{\gamma-\beta}(i\Lambda_k(x,s))^{-1} e^{i\Lambda(x,s)\cdot\xi}.
\end{align*}
Therefore, we can estimate its $L^\infty$ norm and $C^\alpha$ semi-norm rather easily by repeatedly using Assumption \ref{asu1} to obtain the estimates
\begin{align}
\left\|\partial_\xi^{\gamma-\beta}\left(\xi e^{i\Lambda(x,s)\cdot\xi}\right)\right\|_{\infty}&\leq Cs^{|\gamma-\beta|-1}\|\Lambda \|^{|\gamma-\beta|-1}_{C^1}\left(s|\xi|\|\Lambda\|_{C^1}+1\right), \\
\left[\partial_\xi^{\gamma-\beta}\left(\xi e^{i\Lambda(x,s)\cdot\xi}\right)\right]_{\alpha}&\leq C(1+s|\xi|)s^{|\gamma-\beta|-1}\left(s|\xi|\|\Lambda\|^{|\gamma-\beta|}_{C^{2,\alpha}}+\|\Lambda\|^{|\gamma-\beta|-1}_{C^{2,\alpha}}\right).
\end{align}
Hence recalling the expression of $b_{132}$ defined in \eqref{summand:b13} we conclude that
\begin{align}
b_{132}\leq C\sum_{\substack{\beta\leq\gamma\\ \beta\neq \gamma}}\sum_{\kappa}&\int_0^L s^{k^\beta_{\kappa}}|\xi|^{k_{\kappa}^\beta-|\beta|}s\|\Lambda\|_{C^{2,\alpha}}\left(\frac{s^{|\gamma-\beta|-1}\|\Lambda\|_{C^{2,\alpha}}^{|\gamma-\beta|-1}}{1+\|\Theta\|_{\infty}}\left(s|\xi|\|\Lambda\|_{C^{2,\alpha}}+1\right)\right.\nonumber \\
&\left.+\frac{1+s|\xi|}{1-\|\Theta\|_{\infty}}\left(|\xi|s^{|\gamma-\beta|}\|\Lambda\|^{|\gamma-\beta|}+s^{|\gamma-\beta|-1}\|\Lambda\|^{|\gamma-\beta|-1}\right)\right) \leq C\|\Lambda\|_{C^{2,\alpha}}|\xi|^{-1-|\gamma|},  \label{estimate:final:b132}
\end{align}
where we have used the fact that $|\gamma-\beta|\geq 1$, Assumption \ref{asu1} and a change of variable $u=s|\xi|$ to obtain an upper bound for the integral. Thus, bounds \eqref{estimate:final:b131} and \eqref{estimate:final:b132} yields
\begin{equation}\label{estimate:final:b13}
\|b_{13}(\cdot,\xi)\|_{C^{1,\alpha}}\leq C\|\Lambda\|_{C^{2,\alpha}}|\xi|^{-1-|\gamma|}.
\end{equation}
To conclude the proof, we just to check that the same estimate holds for the $C^\alpha$ semi-norm of $b_{14}$ given in \eqref{2summand:b14}. Noticing that since $\beta\neq\gamma$, we have that
$$\partial_\xi^{\gamma-\beta}\left(e^{i\Lambda(x,s)\cdot \xi}-1\right)=\partial_\xi^{\gamma-\beta}\left(e^{i\Lambda(x,s)\cdot \xi}\right).$$
Hence, using Assumption \ref{asu2} and \eqref{eq:exp:1} we can easily estimate $b_{142}$ in  \eqref{2summand:b14} can be estimated by
\begin{equation}\label{b141:alpha}
b_{142}\leq C \|\Lambda\|_{C^{2,\alpha}}|\xi|^{-1-|\gamma|}.
\end{equation}
To bound $b_{141}$ we have to deal with the term $e^{i\xi\cdot\Lambda(x,s)}-1$. Using Assumption \ref{asu1} we find that
$$\left|e^{i\xi\cdot\Lambda(x,s)}-1\right|=\left|\int_0^{\xi\cdot\Lambda(x,s)}e^{iw}\,dw\right|\leq |\xi||\Lambda(x,s)|\leq s|\xi|\|\Lambda\|_{C^{2,\alpha}},$$
and also the $C^\alpha$ semi-norm as
$$\left[e^{i\xi\cdot\Lambda(\cdot,s)}-1\right]_\alpha=\sup_{x_1\neq x_2}\frac{\left|e^{i\xi\cdot\Lambda(x_1,s)}-e^{i\xi\cdot\Lambda(x_2,s)}\right|}{|x_1-x_2|^\alpha}=\left[e^{i\xi\cdot\Lambda(\cdot,s)}\right]_\alpha\leq s|\xi|\|\Lambda\|_{C^{1,\alpha}}.$$
Recall that the later bound was already estimated in \eqref{estimacionalphaexponencial}. Thus, similarly as for $b_{131}$ we find that 
\begin{equation}\label{b142:alpha}
b_{141}\leq C \|\Lambda\|_{C^{2,\alpha}}|\xi|^{-1-|\gamma|}.
\end{equation}
Thus, combining \eqref{b141:alpha}-\eqref{b142:alpha} we conclude that
\begin{equation}\label{estimate:final:b14}
\|b_{14}(\cdot,\xi)\|_{C^{\alpha}}\leq C\|\Lambda\|_{C^{2,\alpha}}|\xi|^{-1-|\gamma|}.
\end{equation}
Therefore, we have shown that for any $\gamma\in \mathbb{N}^2$ multi-index
\[ \norm{\partial_{\xi}^{\gamma}\left(b_{1}(\cdot,\xi)/|\xi|\right)}_{C^{1,\alpha}} \leq C \|\Lambda\|_{C^{2,\alpha}}|\xi|^{-1-|\gamma|}, \]
and hence $\frac{b_{1}}{|\xi|} \in S^{-1}(1+\alpha)$ or similarly $b_{1}\in S^{0}(1+\alpha)$ as desired. 

In order to show that $\frac{b_{2}}{|\xi|} \in S^{1}(1+\alpha)$, where $b_{2}(x,\xi)$ is defined in \eqref{s2}, we can mimic the same estimates derived for $b_{1}(x,\xi)$. Actually, the estimates are easier since we have to deal with the term $\Theta$ instead of $e^{i\xi\cdot\Lambda(x,s)}-1$ which is simpler to handle. To avoid repetitive arguments, we directly provide the bound, namely,
\begin{equation*}
\norm{\partial_{\xi}^{\gamma}\left(b_{2}(\cdot,\xi)/|\xi|\right)}_{C^{1,\alpha}} \leq C \|\Theta\|_{C^{1,\alpha}}|\xi|^{-1-|\gamma|}, \forall \gamma\in \mathbb{N}^2.
\end{equation*}
Thus $\frac{b_{2}}{|\xi|} \in S^{-1}(1+\alpha)$ or similarly $b_{2}\in S^{0}(1+\alpha)$.
\end{proof}

\begin{obs}
Notice that the regularity constants are optimal due to the threshold of regularity given by the functions $\Theta$. Moreover, if we examine again the estimates in the case of $\textsf{T}_1$ the cancellation condition of $\Lambda$ at $s=0$, as it might be expected, since we are deriving optimal estimates for pseudodifferential operators. This has led to bounds on the seminorms of the symbols depending on the $C^{1,\alpha}$ norm of $\partial_s \Lambda$, which will be dominated in the sequel by the $C^{1,\alpha}$ norm of the magnetic field via the definition of flow map.
\end{obs}

\begin{obs}
Notice that by means of the previous computations, and due to the fact that the functions $\Lambda$ and $\Theta$ satisfy Assumptions \ref{asu1} and \ref{asu2} respectively, we can estimate, for any $l\in\mathbb{N}$, the semi-norms of both symbols $b_{1},b_{2}$ by 
\[\|b_1\|_{0,1+\alpha,l},\|b_2\|_{0,1+\alpha,l}\leq C_{\alpha,l}\left(\|\Lambda\|_{C^{2,\alpha}}+\|\Theta\|_{C^{1,\alpha}}\right).\]
\end{obs}
  A direct consequence of Proposition \ref{propb1:b2} is the following regularity result for the symbols $a_{1}, a_{2}$ defined in \eqref{symbol:a1},\eqref{symbol:a2} respectively:
  \begin{coro}\label{corob1:b2}
Under the same hypothesis of Proposition \ref{propb1:b2}, the symbols $a_{1}, a_{2}$ extend to symbols in $S^0(1+\alpha)$.
\end{coro}
\begin{proof}
Note that
\[ a_{1}(x,\xi)=b_{1}(x,\xi)+ \mathsf{r}_{1}(x,\xi), \quad  a_{2}(x,\xi)= b_{2}(x,\xi)+ \mathsf{r}_{2}(x,\xi). \]
Since $b_{1}, b_{2}\in S^{0}(1+\alpha)$, it is sufficient to show that the remainder terms $\mathsf{r}_{1}(x,\xi), \mathsf{r}_{2}(x,\xi) \in S^{-\infty}(1+\alpha)$ are smoothing symbols. Indeed, recalling that $(\mathfrak{m}(\xi))^{-1}=\frac{|\xi|\sinh(|\xi|L)}{\cosh(|\xi|L)-1},$ we have that
\begin{align}
 \mathsf{r}_{1}(x,\xi)&=\left((\mathfrak{m}(\xi))^{-1}-|\xi|\right)\int_0^Le^{-|\xi|s}\frac{\left(e^{-i\xi\cdot\Lambda(x,s)}-1\right)}{1+\Theta(x,s)}ds, \\
  \mathsf{r}_{2}(x,\xi)&=\left((\mathfrak{m}(\xi))^{-1}-|\xi|\right)\int_0^Le^{-|\xi|s}\frac{\Theta(x,s)}{1+\Theta(x,s)}ds.
\end{align}
It easy to check that
\begin{equation}\label{smoothing}
|\left((\mathfrak{m}(\xi))^{-1}-|\xi|\right)|\leq C e^{-|\xi|L},
\end{equation}
and the integral terms can be estimated following the techniques developed in Proposition \ref{propb1:b2}.
\end{proof}

Next, we provide the regularity result for the smoothing symbols $a_{3}, a_{4}$ given in \eqref{symbol:a3}, \eqref{symbol:a4} respectively.
\begin{prop}\label{prop:a3:a4}
 Suppose that Assumptions \ref{asu1} and \ref{asu2} hold. Then, the symbols $a_{3},a_{4}$ given in \eqref{symbol:a3}-\eqref{symbol:a4} respectively, extend to symbols in $S^{-\infty}(1+\alpha)$.
\end{prop}
\begin{proof}
Notice that the function $M(\xi,s)$ defined in \eqref{definition:M:smoothing} is a smooth function $s$ and decays exponentially in $\xi$. In particular, we have the simple bound
\[ |M(\xi,s)| \leq C e^{-|\xi|L}. \]
By the same reasoning as in Corollary \ref{corob1:b2}, the result follows.

\end{proof}

\begin{obs}
Notice that following the same estimates as the one derived in Proposition \ref{propb1:b2} for the symbols $b_{1},b_{2}$, we have that
\begin{equation}\label{akappanorm}
\norm{a_{\kappa}}_{0,1+\alpha,l}\leq C_{\alpha,l}\left( \norm{\Lambda}_{C^{2,\alpha}} +\norm{\Theta}_{C^{1,\alpha}} \right), \ \kappa=1,\ldots,4.
\end{equation}
\end{obs}

We are ready to present the main result of this section, which shows the boundedness of the operators $\textsf{T}_\kappa, \kappa=1,\ldots, 4$, given in \eqref{T1j0}-\eqref{T4j0}. 
\begin{teor}\label{teoremon1}
 Suppose that Assumptions \ref{asu1} and \ref{asu2} hold. Then, the operators $\textsf{T}_\kappa, \kappa=1,\ldots, 4$, given in \eqref{T1j0}-\eqref{T4j0} respectively define bounded operators in $C^{1,\alpha}(\T^2)$. More precisely, we have that
\begin{equation}\label{estimada:teoremon1}
\|\textsf{T}_\kappa \|_{\mathcal{L}(C^{1,\alpha}(\T^2))}\leq C (\|\Lambda\|_{C^{1,\alpha}}+\|\Theta\|_{C^{1,\alpha}}), \quad \kappa=1,\ldots, 4.
\end{equation}
\end{teor}

\begin{proof}[Proof of Theorem \ref{teoremon1}]
In order to show the boundedness of the operators $\textsf{T}_\kappa, \kappa=1,\ldots, 4$, we first recall the results contained in Corollary \ref{corob1:b2} and Proposition \ref{prop:a3:a4} which demonstrate that
$a_{\kappa}(x,\xi)\in S^0(1+\alpha)$ for $\kappa=1,\ldots, 4$. To conclude the proof we invoke Theorem \ref{solucion} for $k=1$. Note that the precise estimate on the operator norm follows by the bounds on the symbols \eqref{akappanorm}.

\end{proof}

\subsection{Estimates on the differences $\textsf{T}_\kappa$}\label{subsec:52}
In this subsection, we will derive estimates for the difference operators. This will be needed in order to show the lower order  $C^{\alpha}$ contraction estimate in the general fixed point argument provided in Section \ref{sec:6}. The idea towards the proof follows the same lines as the one provided in the previous subsection, however some precise control of the constants involved in the computations is needed. To that purpose, for $\Lambda_1,\Lambda_{2}$ and $\Theta_{1},\Theta_{2}$, we denote the difference of the operators $\textsf{T}_\kappa$ as
\begin{align}\label{difference:op:13}
\textsf{T}_{1}^d=\textsf{T}_1[\Lambda_1,\Theta_1]-\textsf{T}_1[\Lambda_2,\Theta_2], \quad 
\textsf{T}_{3}^d=\textsf{T}_3[\Lambda_1,\Theta_1]-\textsf{T}_3[\Lambda_2,\Theta_2], 
\end{align}
and
\begin{align}\label{difference:op:24}
\textsf{T}_{2}^d=\textsf{T}_2[\Theta_1]-\textsf{T}_2[\Theta_2], \quad 
\textsf{T}_{4}^d=\textsf{T}_4[\Theta_1]-\textsf{T}_4[\Theta_2].
\end{align}
One big advantage of the pseudo-differential approach we are following is that $\textsf{T}^d_\kappa$ is once again the adjoint of a pseudo-differential operator, $\p(c_\kappa)$. The symbol $c_\kappa$ is, obviously, given by the difference 
\begin{align}
    c_1(x,\xi)&=(\mathfrak{m}(\xi))^{-1}\int_0^Le^{-|\xi|s}\left(\frac{e^{i\xi\cdot\Lambda_1(x,s)}-1}{1+\Theta_1(x,s)}-\frac{e^{i\xi\cdot\Lambda_2(x,s)}-1}{1+\Theta_2(x,s)}\right)ds\label{symbol:c1} \\
    c_2(x,\xi)&=(\mathfrak{m}(\xi))^{-1}\int_0^Le^{-|\xi|s}\left(\frac{\Theta_{1}(x,s)}{1+\Theta_{1}(x,s)}-\frac{\Theta_{2}(x,s)}{1+\Theta_{2}(x,s)}\right) ds,   \label{symbol:c2} \\
        c_3(x,\xi)&=(\mathfrak{m}(\xi))^{-1}\int_0^{L} M(\xi,s) e^{-|\xi|s}\left(\frac{e^{i\xi\cdot\Lambda_1(x,s)}-1}{1+\Theta_1(x,s)}-\frac{e^{i\xi\cdot\Lambda_2(x,s)}-1}{1+\Theta_2(x,s)}\right)ds, \label{symbol:c3}\\ 
    c_{4}(x,\xi)&= (\mathfrak{m}(\xi))^{-1}\int_0^L M(\xi,s) e^{-|\xi|s}\left(\frac{\Theta_{1}(x,s)}{1+\Theta_{1}(x,s)}-\frac{\Theta_{2}(x,s)}{1+\Theta_{2}(x,s)}\right) ds.\label{symbol:c4}
\end{align}
As in Subsection \ref{subsec:51}, to simplify the exposition, we work with a more simplified version of $c_{1},c_{2}$, namely
\begin{align}
    d_1(x,\xi)&=|\xi| \int_0^Le^{-|\xi|s}\left(\frac{e^{i\xi\cdot\Lambda_1(x,s)}-1}{1+\Theta_1(x,s)}-\frac{e^{i\xi\cdot\Lambda_2(x,s)}-1}{1+\Theta_2(x,s)}\right)ds\label{symbol:d1}, \\
    d_2(x,\xi)&=|\xi| \int_0^Le^{-|\xi|s}\left(\frac{\Theta_{1}(x,s)}{1+\Theta_{1}(x,s)}-\frac{\Theta_{2}(x,s)}{1+\Theta_{2}(x,s)}\right) ds.   \label{symbol:d2} 
\end{align}
Similarly as in the previous subsection, we have that
\[ c_1(x,\xi)=d_1(x,\xi)+\mathsf{r}_{1}(x,\xi), \ c_2(x,\xi)=d_2(x,\xi)+\mathsf{r}_{2}(x,\xi),\]
where the the remainder terms are smoothing symbols. Therefore, it is sufficient to provide the regularity results for the symbols $d_1(x,\xi), d_2(x,\xi)$. Furthermore, the fast decay of the function $M(\xi,s)$, provides that the kernels defined out of $c_{3}(x,\xi), c_{4}(x,\xi)$ are smooth along the same lines as in Proposition \ref{prop:a3:a4}.
One has to acknowledge that we want estimates in the case of less regularity, so we cannot directly use Theorem \ref{solucion} to conclude that $\textsf{T}^d_{\kappa}\in \mathcal{L}(C^\alpha(\T^2))$. However, if we prove adequate estimates for the symbols involved, we can get boundedness on the $C^\alpha$ norm when $\textsf{T}^d_{\kappa}$ acts on $C^{1,\alpha}$ functions. 
\begin{prop}\label{prop:d1d2}
Suppose that Assumptions \ref{asu1} and \ref{asu2} hold. Then, the symbols $d_{1},d_{2}$ given in \eqref{symbol:d1}-\eqref{symbol:d1} respectively, extend to symbols in $S^0(\alpha)$. More precisely, we have that
\begin{align}
\left\|\partial_\xi^\gamma\left(d_{1}(\cdot,\xi)\right)\right\|_{C^{\alpha}} &\leq C|\xi|^{-|\gamma|} \left( \|\Lambda_1-\Lambda_2\|_{C^{1,\alpha}}+\|\Theta_1-\Theta_2\|_{C^{\alpha}} \right),  \\
\left\|\partial_\xi^\gamma\left(d_{2}(\cdot,\xi)\right)\right\|_{C^{\alpha}}&\leq C|\xi|^{-|\gamma|} \left(\|\Theta_1-\Theta_2\|_{C^{\alpha}} \right).
\end{align}
\end{prop}
\begin{proof}
 The proof follows closely the arguments of Proposition \ref{propb1:b2}. To avoid repetition, we provide the full details for the symbol $d_{1}$ and just give the precise final estimate for the symbol $d_{2}$. We begin by writing them in a more convenient way, that explicitly shows the dependence in the distance between the functions $\Lambda_1$ and $\Lambda_2$, and the functions $\Theta_1$ and $\Theta_2$, namely,
 \begin{equation*}
        \begin{split} 
        d_1(x,\xi)&=|\xi|\int_0^Le^{-|\xi|s}\left(\frac{e^{i\xi\cdot\Lambda_1(x,s)}-1}{1+\Theta_1(x,s)}-\frac{e^{i\xi\cdot\Lambda_2(x,s)}-1}{1+\Theta_2(x,s)}\right)ds\\
        &=|\xi|\int_0^Le^{-|\xi|s}\frac{1}{1+\Theta_1(x,s)}\left(e^{i\xi\cdot\Lambda_1(x,s)}-e^{i\xi\cdot\Lambda_2(x,s)}\right)ds\\
        &+|\xi|\int_0^Le^{-|\xi|s}\left(e^{i\xi\cdot\Lambda_1(x,s)}-1\right)\frac{\Theta_2(x,s)-\Theta_1(x,s)}{(\Theta_1(x,s)+1)(\Theta_2(x,s)+1)}ds=d_{11}(x,s)+d_{12}(x,s).
        \end{split}
    \end{equation*}
Noticing that $|\xi|$ is a Mikhlin-Hörmander multiplier of order $1$, we need to show that $d_1/|\xi|\in S^{-1}(\alpha)$.  
\subsubsection*{\underline{$L^{\infty}$ bounds for $\partial_{\xi}^{\gamma}\left(d_{1}(x,\xi)/|\xi|\right)$}}
 By applying Leibniz's rule and recalling \eqref{eq:exp:1}, we find that
  \begin{equation*}
        \begin{split}
          \partial^\gamma_\xi \left(d_{11}/|\xi|\right)&= \sum_{\beta\leq \gamma}{\gamma\choose\beta} \int_0^L \partial_\xi^\beta\left(e^{-s|\xi|}\right)\frac{\partial_\xi^{\gamma-\beta}\left(e^{i\xi\cdot\Lambda_1(x,s)}-e^{i\xi\cdot\Lambda_2(x,s)}\right)}{1+\Theta_1(x,s)}ds\\
          &=\sum_{\beta\leq \gamma}{\gamma\choose\beta} \int_0^L \partial_\xi^\beta\left(e^{-s|\xi|}\right)\frac{\left(\Lambda^{\gamma-\beta}_2(x,s)e^{i\xi\cdot\Lambda_2(x,s)}-\Lambda^{\gamma-\beta}_1(x,s)e^{i\xi\cdot\Lambda_1(x,s)}\right)}{1+\Theta_1(x,s)}ds\\
          &=\sum_{\beta\leq \gamma}{\gamma\choose\beta}\sum_{\ell} \int_0^L e^{-s|\xi|}s^{k^\beta_\ell}p^\beta_\ell(\xi)\frac{\left(\Lambda^{\gamma-\beta}_2(x,s)e^{i\xi\cdot\Lambda_2(x,s)}-\Lambda^{\gamma-\beta}_1(x,s)e^{i\xi\cdot\Lambda_1(x,s)}\right)}{1+\Theta_{1}(x,s)}ds.
        \end{split}
    \end{equation*}
Here $p^\beta_\kappa(\xi)\in C^\infty(\R^2\setminus \{0\})$ satisfying \eqref{eq:exp:2}. By the fundamental theorem of calculus,
    \begin{equation*}
        \begin{split}
        \Lambda^{\gamma-\beta}_2(x,s)e^{i\xi\cdot\Lambda_2(x,s)}&-\Lambda^{\gamma-\beta}_1(x,s)e^{i\xi\cdot\Lambda_1(x,s)}\\
        &=\left(\Lambda_1(x,s)-\Lambda_2(x,s)\right)\cdot\int_0^1\nabla \mathsf{h}_\xi (t\Lambda_1(x,s)+(1-t)\Lambda_2(x,s))dt,
        \end{split}
    \end{equation*}
  where $\mathsf{h}_\xi(r)=r^{\gamma-\beta}e^{i\xi\cdot r}$. Note that
\[ \nabla \mathsf{h}(r)=\mathscr{F}(r)e^{i\xi\cdot r}+r^{\gamma-\beta}\xi e^{i\xi\cdot r},\]
with $\mathscr{F}(r)$ a vector where every element consists on a monomial of degree $|\gamma-\beta|-1$. Furthermore, since by Assumption \ref{asu1} $\Lambda_1(x,0)=\Lambda_2(x,0)=0$, and $\|\Lambda_1\|_{C^{1}},\|\Lambda_2\|_{C^2}\leq 1$, we obtain that
    \begin{align}
    \left|\Lambda^{\gamma-\beta}_2(x,s)e^{i\xi\cdot\Lambda_2(x,s)}-\Lambda^{\gamma-\beta}_1(x,s)e^{i\xi\cdot\Lambda_1(x,s)}\right|&\leq C\left(s^{|\gamma-\beta|-1}+|\xi|s^{|\gamma-\beta|}\right)|\Lambda_1(x,s)-\Lambda_2(x,s)|\nonumber \\
    &\leq C\left(1+s|\xi|\right)s^{|\gamma-\beta|}\|\Lambda_1-\Lambda_2\|_{C^1}. \label{estimate:plugging1}
    \end{align}

Hence, combing \eqref{estimate:plugging1} and the fact that $\|\Theta_{1}\|_{C^{1,\alpha}}\leq 1$ , we conclude that
    \begin{align}
            \left|\partial_\xi^\gamma \left(d_{11}/|\xi|\right)\right|&\leq \|\Lambda_1-\Lambda_2\|_{C^1}\frac{C}{|\xi|^{|\gamma|}}\sum_{\beta\leq \gamma}\sum_{\ell} \int_0^L \left(|\xi|s\right)^{k^\beta_{\ell}+|\gamma-\beta|}\left(1+s|\xi|\right)e^{-s|\xi|}ds\nonumber \\
            &\leq C \|\Lambda_1-\Lambda_2\|_{C^1}\frac{1}{|\xi|^{|\gamma|+1}}\sum_{\beta\leq \gamma}\sum_\ell\int_0^\infty u^{k^\beta_\ell+|\gamma+\beta|}(1+u)e^{-u}du \nonumber \\
            &\leq C|\xi|^{-|\gamma|-1}\|\Lambda_1-\Lambda_2\|_{C^{1}}. \label{estimate:d11Linfty}
\end{align}
Let us proceed to derive the $L^{\infty}$ bounds for $d_{12}$. Similarly, using Leibniz rule we have that
    \begin{equation*}
        \begin{split}
            \partial_\xi^\gamma\left(\frac{d_{12}}{|\xi|}\right)&=\sum_{\beta\leq \gamma}{\gamma\choose\beta}\int_0^L\partial^\beta_\xi \left(e^{-s|\xi|}\right)\partial_\xi^{\gamma-\beta}\left(e^{i\xi\cdot \Lambda_2(x,s)}-1\right)\frac{\Theta_2(x,s)-\Theta_1(x,s)}{\left(1+\Theta_1(x,s)\right)\left(1+\Theta_2(x,s)\right)}ds\\
            &=\displaystyle\sum_{\substack{\beta\leq\gamma\\ \beta\neq\gamma}}{\gamma\choose\beta} \sum_{\ell} \int_0^L s^{k^\beta_\ell}p^\beta_\ell (\xi)\left(i\Lambda_2(x,s)\right)^{\gamma-\beta}e^{i\xi\cdot\Lambda(x,s)}\frac{\Theta_2(x,s)-\Theta_1(x,s)}{\left(1+\Theta_1(x,s)\right)\left(1+\Theta_2(x,s)\right)}ds\\
            &+\displaystyle\sum_{\ell} \int^L_0 s^{k^\beta_\ell}p^\beta_\ell (\xi)\left(e^{i\xi\cdot\Lambda_2(x,s)}-1\right)\frac{\Theta_2(x,s)-\Theta_1(x,s)}{(1+\Theta_1(x,s))(1+\Theta_2(x,s))}ds.            
        \end{split}
    \end{equation*}
Using that $\Lambda_2(x,0)=0$ by Assumption \ref{asu1}, we derive the bounds
\[ \left|\Lambda_2(x,s)\right|^{|\gamma-\beta|}\leq s^{|\gamma-\beta|}\|\Lambda\|_{C^1}^{|\gamma-\beta|}, \quad  \left|e^{i\xi\cdot\Lambda_2(x,s)}-1\right|\leq s |\xi| \|\Lambda_{2}\|_{C^1}. \] 
Thus,  
 \begin{align}
     \left|\partial_\xi^\gamma\left(\frac{d_{12}}{|\xi|}\right)\right|&\leq C|\xi|^{-|\gamma|}\|\Theta_1-\Theta_2\|_{\infty}\sum_{\substack{\beta\leq \gamma\\\beta\neq\gamma}}\sum_{\ell}\int_0^L (|\xi|s)^{k^\beta_{\ell}+|\gamma-\beta|}e^{-s|\xi|}ds \nonumber \\
         &\quad \quad +C|\xi|^{-|\gamma|}\|\Theta_1-\Theta_2\|_{\infty}\sum_{\ell} \int_0^L (s|\xi|)^{k^\beta_{\ell}}e^{-s|\xi|}ds \nonumber \\
         &\leq C|\xi|^{-|\gamma|-1}\|\Theta_1-\Theta_2\|_{\infty}. \label{estimate:d12Linfty}
 \end{align}
 Combing  \eqref{estimate:d11Linfty} and \eqref{estimate:d12Linfty}, we obtain the $L^{\infty}$ bound
 \begin{equation}
 \partial_{\xi}^{\gamma}\left( \frac{d_{1}(x,\xi)}{|\xi|}\right)\leq C|\xi|^{-|\gamma|-1}\left( \|\Lambda_1-\Lambda_2\|_{C^{1}}+\|\Theta_1-\Theta_2\|_{\infty} \right) . \label{final:d1:Linfty}
 \end{equation}
\subsubsection*{\underline{$C^{\alpha}$ bounds for $\partial_{\xi}^{\gamma}\left(d_{1}(x,\xi)/|\xi|\right)$}} Next, we derive the $C^{\alpha}$ bounds.  We begin with the estimates for $d_{11}$. It is more convenient to rewrite $d_{11}$ as follows
\[ d_{11}(x,\xi)=|\xi|\int_0^L e^{-|\xi|s}\frac{e^{i\xi\cdot\Lambda_2(x,s)}}{1+\Theta_1(x,s)}\left(e^{i\xi\cdot\left(\Lambda_1(x,s)-\Lambda_2(x,s)\right)}-1\right) \ ds.\]
Using Leibniz rule, we find that
\begin{align}
        \partial_\xi^\gamma\left(\frac{d_{11}}{|\xi|}\right)&=\sum_{\beta\leq \gamma}{\gamma\choose\beta}\int_0^L\partial_\xi^{\gamma-\beta}\left(e^{-s|\xi|}\right)\partial^\beta_\xi\left(\frac{e^{i\xi\cdot\Lambda_2(x,s)}}{1+\Theta_1(x,s)}\left(e^{i\xi\cdot\left(\Lambda_1(x,s)-\Lambda_2(x,s)\right)}-1\right)\right)ds \nonumber \\
        &=\sum_{\delta\leq \beta\leq \gamma}{\gamma\choose\beta}{\beta\choose\delta}\int_0^L\partial_\xi^{\gamma-\beta}\left(e^{-s|\xi|}\right)\frac{\partial^\delta_\xi \left(e^{i\xi\cdot\Lambda_2(x,s)}\right)}{1+\Theta_1(x,s)}\partial^{\beta-\delta}_\xi\left(e^{i\xi\cdot\left(\Lambda_1(x,s)-\Lambda_2(x,s)\right)}-1\right)ds. \label{chapa}
    \end{align}
 Let us of compute the $C^{\alpha}$ semi-norm of the different terms involved in \eqref{chapa}.  \\
 \paragraph{\underline{Bounds for $[\frac{1}{1+\Theta_1(x,s)}]_{\alpha}$}} The $L^\infty$ norm can be dominated by $\frac{1}{1-\|\Theta_1\|_\infty}$. On the other hand, the $C^\alpha$ semi-norm can be bounded by 
\begin{equation}\label{eq:alpha:theta1:diff}
 \left[\frac{1}{1+\Theta_1(\cdot,s)}\right]_\alpha=\sum_{x\neq y}\frac{|\Theta_1(y,s)-\Theta_1(x,s)|}{|1+\Theta_1(x,s)||1+\Theta_2(x,s)|}\leq \frac{[\Theta_1]_\alpha}{(1-\|\Theta_1\|_\infty)^2}.
 \end{equation}
\paragraph{\underline{ Bounds for $[\partial^\delta_\xi\left(e^{i\xi\cdot\Lambda_2(x,s)}\right)]_{\alpha}$}}
We readily check that
$\partial^\delta_\xi\left(e^{i\xi\cdot\Lambda_2(x,s)}\right)=\left(i\Lambda_2(x,s)\right)^\delta e^{i\xi\cdot\Lambda_2(x,s)}$. 
Moreover, due to Assumption \ref{asu1} we use that $\Lambda_2(x,0)=0$ to find that
    $$\left|\left(i\Lambda_2(x,s)\right)^\delta e^{i\xi\cdot\Lambda_2(x,s)}\right|\leq s^{|\delta|}\|\Lambda_2\|_{C^1}^\delta. $$
   For the $C^\alpha$ semi-norm, we realize that $[\cdot]_\alpha\leq C\|\cdot\|_{C^1}$. Since $\Lambda_{2}\in C^{1}$, we obtain 

    $$\nabla_x\partial^\delta_\xi\left(e^{i\xi\cdot\Lambda_2(x,s)}\right)=\nabla_x\Lambda_2(x,s)\cdot \left[\mathscr{F}(\Lambda_2(x,s))+\xi\left(i\Lambda(x,s)\right)^\delta\right]e^{i\xi\cdot \Lambda_2(x,s)},$$
  where $\mathscr{F}$ is a vector where each component is a monomials of degree $|\delta|-1$. Thus, since $\nabla_x\Lambda_2(x,0)=0$, and $\norm{\Lambda_2}_{C^{2,\alpha}}\leq \frac{1}{2}$ we conclude that
\begin{equation}\label{eq:alpha:exp:diff}
\left[\partial^\delta_\xi\left(e^{i\xi\cdot\Lambda_2(\cdot,s)}\right)\right]_\alpha\leq \left\|\partial^\delta_\xi\left(e^{i\xi\cdot\Lambda_2(\cdot,s)}\right)\right\|_{C^1}\leq C s^{|\delta|}\left(1+s|\xi|\right).
\end{equation}

\paragraph{\underline{Bounds for $[\partial^{\beta-\delta}_\xi\left(e^{i\xi\cdot\left(\Lambda_1(x,s)-\Lambda_2(x,s)\right)}-1\right)]_{\alpha}$}}
The $L^\infty$ and the $C^\alpha$ bounds are obtained similarly as in the proof of Proposition \ref{propb1:b2}, just by substituting $\Lambda$ by $\Lambda_1-\Lambda_2$. Then, we conclude that 
\begin{equation}\label{eq:alpha:theta3:diff}
 \left\|\partial^{\beta-\delta}_\xi\left(e^{i\xi\cdot\left(\Lambda_1(x,s)-\Lambda_2(x,s)\right)}-1\right)\right\|_{C^{1,\alpha}}\leq C\left(|\xi|s\right)^{|\beta-\delta|}|\xi|^{-|\beta-\delta|}\|\Lambda_1-\Lambda_2\|_{C^{1,\alpha}}.
 \end{equation}

Hence, gathering bounds \eqref{eq:alpha:theta1:diff}-\eqref{eq:alpha:theta3:diff} and using the H\"older semi-norm estimate \eqref{inequality:calpha} we infer that
\begin{equation}
\left[\partial_\xi^\gamma \left(\frac{d_{11}}{|\xi|}\right)\right]_{\alpha}\leq C\sum_{\delta\leq \beta\leq\gamma}\sum_{\ell}\|\Lambda_1-\Lambda_2\|_{C^{1,\alpha}} |\xi|^{-|\beta|}\int_0^L s^{k^{\gamma-\beta}_{\ell}}p^{\gamma-\beta}_{\ell} (\xi) \left(s|\xi|\right)^{|\delta|}(1+s|\xi|)\left(|\xi|s\right)^{|\beta-\delta|}|ds.
\end{equation}
Since  $|p_{\ell}|\lesssim |\xi|^{k_\ell} $, we obtain that
\begin{equation}\label{final:calphad11}
 \left\|\partial_\xi^\gamma \left(\frac{d_{11}}{|\xi|}\right)\right\|_{C^\alpha}\leq C |\xi|^{-|\gamma|-1}\|\Lambda_1-\Lambda_2\|_{C^{1,\alpha}},
\end{equation}
and therefore, $d_{11}\in S^{0}(\alpha)$.

 To conclude that $d_{1}\in S^{0}(\alpha)$, we have to perform similar estimates for $d_{12}$. Indeed, we first compute the $\partial_\xi^\gamma\left(\frac{d_{12}}{|\xi|}\right)$, namely
\[ 
\partial_\xi^\gamma\left(\frac{d_{12}}{|\xi|}\right)=\sum_{\beta\leq\gamma}{\gamma\choose\beta}\int_0^L \partial_\xi^\beta\left(e^{-s|\xi|}\right)\partial_\xi^{\gamma-\beta}\left(e^{i\xi\cdot\Lambda_2(x,s)}-1\right)\frac{\Theta_2(x,s)-\Theta_1(x,s)}{(1+\Theta_1(x,s))(1+\Theta_2(x,s))}ds.
\]
The $L^{\infty}$ estimate is straightforward. Using the fact that $\Lambda_2(x,0)=0$ and formula \eqref{eq:exp:3} to expand $e^{-s|\xi|}$ we infer that
\begin{align} 
\left|\partial_\xi^\gamma\left(\frac{d_{12}}{|\xi|}\right)\right|&\leq C\|\Theta_1-\Theta_2\|_{\infty}|\xi|^{-|\gamma|}\sum_{\beta\leq \gamma} \sum_{\ell}\int_0^L e^{-s|\xi|} \left(s|\xi|\right)^{k^\beta_{\ell}} \left(s|\xi|\right)^{|\gamma-\beta|}ds\nonumber \\
&\leq C|\xi|^{-|\gamma|-1}\|\Theta_1-\Theta_2\|_{\infty}. \label{d12:theta:Linfty}
\end{align}
For the $C^{\alpha}$, we just notice that
$$\left\|\frac{\Theta_2(x,s)-\Theta_1(x,s)}{(1+\Theta_1(x,s))(1+\Theta_2(x,s))}\right\|_{C^{\alpha}}\leq \|\Theta_2(x,s)-\Theta_1(x,s)\|_{C^{\alpha}}\left\|\frac{1}{(1+\Theta_1(x,s))(1+\Theta_2(x,s))}\right\|_{C^{\alpha}}.$$
The right factor above can be bounded easily by using Assumption \ref{asu2} and estimate
$$\left[\frac{1}{1+\Theta}\right]_\alpha\leq \frac{\|\Theta\|_{C^{\alpha}}}{\left(1-\|\Theta\|_{\infty}\right)^2}.$$
Therefore, we have obtained that
\begin{equation}\label{final:d12:calpha}
\left\|\partial_\xi^\gamma\left(\frac{d_{12}}{|\xi|}\right)\right\|_{C^{\alpha}}\leq C|\xi|^{-|\gamma|-1}\|\Theta_1-\Theta_2\|_{C^{\alpha}}.
\end{equation}

Combining \eqref{final:calphad11} and \eqref{final:d12:calpha} we finally conclude that
\[ \left\|\partial_\xi^\gamma\left(\frac{d_{1}}{|\xi|}\right)\right\|_{C^{\alpha}}\leq C|\xi|^{-|\gamma|-1} \left( \|\Lambda_1-\Lambda_2\|_{C^{1,\alpha}}+\|\Theta_1-\Theta_2\|_{C^{\alpha}} \right), \]
showing that $d_{1}/|\xi|\in S^{-1}(\alpha)$ and hence $d_{1}\in S^{0}(\alpha)$.

It is not difficult to check that we can mimic the same arguments derived to provide the final bound for $d_{1}$ for the symbol $d_{2}$. More precisely, one can derive the estimate
\[ \left\|\partial_\xi^\gamma\left(\frac{d_{2}}{|\xi|}\right)\right\|_{C^{\alpha}}\leq C|\xi|^{-|\gamma|-1} \left(\|\Theta_1-\Theta_2\|_{C^{\alpha}} \right), \]
which shows that $d_{2}/|\xi|\in S^{-1}(\alpha)$ and hence $d_{2}\in S^{0}(\alpha)$.
\end{proof}
Using the same ideas as in Corollary \ref{corob1:b2}, we can derive its counterpart for the difference operators.
\begin{coro}\label{prop:c1c2}
Under the same hypothesis of Proposition \ref{prop:d1d2}, the symbols $c_{1}, c_{2}$ extend to symbols in $S^0(\alpha)$. More precisely, we have that
\begin{align}
\left\|\partial_\xi^\gamma\left(c_{1}(\cdot,\xi)\right)\right\|_{C^{\alpha}} &\leq C|\xi|^{-|\gamma|} \left( \|\Lambda_1-\Lambda_2\|_{C^{1,\alpha}}+\|\Theta_1-\Theta_2\|_{C^{\alpha}} \right),  \label{c1:estimatesymbol} \\
\left\|\partial_\xi^\gamma\left(c_{2}(\cdot,\xi)\right)\right\|_{C^{\alpha}}&\leq C|\xi|^{-|\gamma|} \left(\|\Theta_1-\Theta_2\|_{C^{\alpha}} \right).
\end{align}
\end{coro}

Next, we provide in the same spirit as in Proposition \ref{prop:a3:a4}, the regularity result for the smoothing symbols $c_{3}, c_{4}$ given in \eqref{symbol:c3}, \eqref{symbol:c4} respectively that reads
\begin{prop}\label{prop:c3c4}
Let Assumptions \ref{asu1} and \ref{asu2} hold. Then, the symbols $c_{3},c_{4}$ given in \eqref{symbol:c3}-\eqref{symbol:c4} respectively, extend to symbols in $S^{-\infty}(\alpha)$. In particular, we have that
\begin{align}
\left\|\partial_\xi^\gamma\left(c_{3}(\cdot,\xi)\right)\right\|_{C^{\alpha}} &\leq C|\xi|^{-|\gamma|} \left( \|\Lambda_1-\Lambda_2\|_{C^{1,\alpha}}+\|\Theta_1-\Theta_2\|_{C^{\alpha}} \right),  \\
\left\|\partial_\xi^\gamma\left(c_{4}(\cdot,\xi)\right)\right\|_{C^{\alpha}}&\leq C|\xi|^{-|\gamma|} \left(\|\Theta_1-\Theta_2\|_{C^{\alpha}} \right). \label{c4:estimatesymbol}
\end{align}
\end{prop}

The precise statement of the main result for the difference of the operators $\mathsf{T}_{\kappa}$ is the following
\begin{teor}\label{diferencia}
Let $\Lambda_{1},\Lambda_{2}$ satisfy Assumption \ref{asu1} and $\Theta_{1},\Theta_{2}$  satisfy Assumption \ref{asu2}. Then, for $j_{0}^{\ell}\in C^{1,\alpha}(\mathbb{T}^{2})$, $\ell=1,2,$ the operators $\textsf{T}^{d}_\kappa$, \ $\kappa=1,\ldots, 4$ given in \eqref{difference:op:13}-\eqref{difference:op:24} satisfy
 \begin{align}
  \|\textsf{T}^{d}_{1}j_{0}^{\ell}\|_{C^{\alpha}(\mathbb{T}^{2})}+ \|\textsf{T}^{d}_{3}j^{\ell}_{0}\|_{C^{\alpha}(\mathbb{T}^{2})}&\leq C (\|\Lambda_1-\Lambda_2\|_{C^{1,\alpha}(\Omega)}+\|\Theta_1-\Theta_2\|_{C^{\alpha}(\Omega)}), \label{estimate:T1T3:diff} \\
  \|\textsf{T}^{d}_{2}j_0^{\ell}\|_{C^{\alpha}(\mathbb{T}^{2})}+\|\textsf{T}^{d}_{4}j^{\ell}_{0}\|_{C^{\alpha}(\mathbb{T}^{2})}&\leq C (\|\Theta_1-\Theta_2\|_{C^{\alpha}(\Omega)}), \label{estimate:T2T4:diff}
\end{align}
 \end{teor}
 
 \begin{proof}[Proof of Theorem \ref{diferencia}]
 Due to Proposition \ref{prop:d1d2}, Corollary \ref{prop:c1c2}
 and Proposition \ref{prop:c3c4} we infer that the symbols $c_{\kappa}(x,\xi)\in S^{0}(\alpha)$, for $\kappa=1,\ldots, 4$. Therefore, 
$\textsf{T}^d_{\kappa}|_{C^{1,\alpha}}$ is the restriction of a pseudo-differential operator given by a symbol in $S^0(\alpha)$. Therefore, $\textsf{T}^d_{\kappa}\in \mathcal{L}(C^{\alpha})$ and, in particular, due to the estimates \eqref{c1:estimatesymbol}- \eqref{c4:estimatesymbol} we have that the bounds \eqref{estimate:T1T3:diff} and \eqref{estimate:T2T4:diff} hold.
 \end{proof}

\subsection{H\"older estimates on the functions $\textsf{G}_\ell$}\label{subsec:5:3}
In this section, we will provide the H\"older estimates for the operators 
\[ \textsf{G}_\ell=\mathcal{T}_{0}^{-1}\tilde{g}_{\ell}, \ \ell=1,2, \] 
defined in \eqref{varios:def:operators}. Let us also recall that $\tilde{g}_{\ell}=g_{\ell}(r)-\mathcal{Z}_{\ell}(r)$ and $\mathcal{Z}_{\ell}$ are defined in \eqref{freakz1:computation} and \eqref{freakz2:computation} respectively. It is convenient to recall the precise the definitions of $\mathcal{Z}_{\ell}$, namely, 
\begin{align}
\z_{1}(r)&= -\sum_{\xi\in \Z^2}\reallywidehat{h}_1^{-}(\xi)\left(\frac{mn}{|\xi|^2}\chi_{\{\xi\neq0\}}\right)\frac{|\xi|\cosh(|\xi|L)}{\sinh(|\xi|L)}\foul +\sum_{\xi\in \Z^2}\reallywidehat{h}_1^{+}(\xi)\left(\frac{mn}{|\xi|^2}\chi_{\{\xi\neq0\}}\right)\frac{|\xi|}{\sinh(|\xi|L)}\foul\nonumber \\
&-\sum_{\xi\in \Z^2}\reallywidehat{h}_2^{-}(\xi)\left(\frac{n^2}{|\xi|^2}\chi_{\{\xi\neq0\}}-1\right)\frac{|\xi|\cosh(|\xi|L)}{\sinh(|\xi|L)}\foul +\sum_{\xi\in \Z^2}\reallywidehat{h}_2^{+}(\xi)\left(\frac{n^2}{|\xi|^2}\chi_{\{\xi\neq0\}}-1\right)\frac{|\xi|}{\sinh(|\xi|L)}\foul \label{freakz1:computation:new}, \\
\z_{2}(r)&=\sum_{\xi\in \Z^2}\reallywidehat{h}_1^{-}(\xi)\left(\frac{m^2}{|\xi|^2}\chi_{\{\xi\neq0\}}-1\right)\frac{|\xi|\cosh(|\xi|L)}{\sinh(|\xi|L)}\foul -\sum_{\xi\in \Z^2}\reallywidehat{h}_1^{+}(\xi)\left(\frac{m^2}{|\xi|^2}\chi_{\{\xi\neq0\}}-1\right)\frac{|\xi|}{\sinh(|\xi|L)}\foul \nonumber \\
&+\sum_{\xi\in \Z^2}\reallywidehat{h}_2^{-}(\xi)\left(\frac{mn}{|\xi|^2}\chi_{\{\xi\neq0\}}\right)\frac{|\xi|\cosh(|\xi|L)}{\sinh(|\xi|L)}\foul-\sum_{\xi\in \Z^2}\reallywidehat{h}_2^{+}(\xi)\left(\frac{mn}{|\xi|^2}\chi_{\{\xi\neq0\}}\right)\frac{|\xi|}{\sinh(|\xi|L)}\foul. \label{freakz2:computation:new}
\end{align}
Furthermore, the functions $h^{\pm}_\ell$ ($\ell=1,2$) are defined as 
\begin{equation} \label{boundary:f:h:new}
\begin{split}
h_2^-(x,y)&=\int_0^x f(s,y,0)\,ds-\frac{x}{2\pi}\int_0^{2\pi}f(s,y,0)\,ds,\\
h_1^{-}(x,y)&=-\frac{1}{2\pi}\int_0^{2\pi}\int_0^yf(s,t,0)\,dtds+\frac{y}{(2\pi)^2}\int_0^{2\pi}\int_0^{2\pi}f(s,t,0)\,dtds,\\
h_2^+(x,y)&=\int_0^xf(s,y,L)\,ds-\frac{x}{2\pi}\int_0^{2\pi}f(s,y,L)\,ds,\\
h_1^{+}(x,y)&=-\frac{1}{2\pi}\int_0^{2\pi}\int_0^yf(s,t,L)\,dtds+\frac{y}{(2\pi)^2}\int_0^{2\pi}\int_0^{2\pi}f(s,t,L)\,dtds.
\end{split}
\end{equation}
We need to give a precise sense of the different operators listed above. In the following Lemma we derive common estimates for the different $\mathcal{Z}_\ell$ in terms of the $C^{2,\alpha}$ norm of the boundary value $f$, in spite of the different level of regularity of the $h_\ell^{\pm}$.

\begin{lema}\label{lema:z1z2}
Let $f\in C^{2,\alpha}(\partial\Omega)$. Then, $\mathcal{Z}_1\in C^{2,\alpha}(\mathbb{T}^{2})$ and $\mathcal{Z}_{2}\in C^{2,\alpha}(\mathbb{T}^{2})$. Moreover, there exists a constant $C>0$ such that
\begin{equation}\label{estimate:Z1Z2}
\|\mathcal{Z}_{\ell}\|_{C^{2,\alpha}}\leq C\|f\|_{C^{2,\alpha}}, \quad \ell=1,2,
\end{equation}
holds.
\end{lema}

\begin{proof}[Proof of Lemma \ref{lema:z1z2}]
We have to understand the functions $\mathcal{Z}_{\ell}$, $\ell=1,2,$ as periodic distributions with Fourier coefficients given by those in expressions \eqref{freakz1:computation:new} and \eqref{freakz2:computation:new}. The operators defined by $f\mapsto \mathcal{Z}_{\ell}$ are Fourier multipliers, $\ell=1,2$. Since both functions are the same, we will just focus on $\mathcal{Z}_1$. This can clearly be written as the sum of four summands that we denote $\mathcal{Z}_{11}$, $\mathcal{Z}_{12}$, $\mathcal{Z}_{13}$ and $\mathcal{Z}_{14}$. 
The multiplier in $\mathcal{Z}_{11}$ is given by 
\begin{equation}\label{mult:1}
\left(\frac{mn}{|\xi|^2}\chi_{\{\xi\neq 0\}}\right)\frac{|\xi|\cosh(|\xi|L)}{\sinh(|\xi|L)},
\end{equation}
that corresponds with the consecutive application of the multipliers 
$\frac{m}{|\xi|}$ and $\frac{n}{|\xi|}.$
Notice that both are zero order Mikhlin-Hörmander multipliers. On the other hand, we have that
\[ \frac{|\xi|\cosh(|\xi|L)}{\sinh(|\xi|L)}=|\xi| + \mathsf{r}(\xi),\]
this is Mikhlin-Hörmander multiplier $|\xi|$ plus a smoothing multiplier $\mathsf{r}(\xi)$, similarly as in \eqref{smoothing}. Combining the previous ideas, we conclude that \eqref{mult:1} equals is Mikhlin-Hörmander multiplier of order 1 plus a smoothing operator. Theferore, taking into account that $h_{1}^{-}\in C^{3,\alpha}$ we find that
$$\left\| \mathcal{Z}_{11}\right\|_{C^{2,\alpha}}\leq C\|h_1^-\|_{C^{3,\alpha}}\leq C\|f\|_{C^{2,\alpha}}.$$
The second and fourth term in $\mathcal{Z}_1$ are straightforward since the involved multiplier 
$\frac{|\xi|}{\sinh(|\xi|L)}$ decays exponentially in $\xi$ and hence $\mathcal{Z}_{12}, \mathcal{Z}_{14}$ are
smoothing operator. Moreover, we have the estimates 
\[ \|\mathcal{Z}_{12}\|_{C^{2,\alpha}},\,\|\mathcal{Z}_{14}\|_{C^{2,\alpha}}\leq C\|f\|_{C^{2,\alpha}}.\]

The remaining term $\mathcal{Z}_{13}$ is more subtle, because $h_2^-\notin C^{3,\alpha}$. Therefore, to obtain the desired estimate, we need to modify the previous argument. We notice that although $h_2^-\notin C^{3,\alpha}$ in both variables, it has more differentiability in the $x$-variable. Thus, we have rearrange the multipliers in the definition of $\mathcal{Z}_{12}$ to obtain the correct estimate. More precisely, notice that 
\[ \partial_1h^-_2=f(x,y,0)-\frac{1}{2\pi}\int_0^{2\pi}f(s,y,0) \ ds,\]
which shows that
 \[ \|\partial_1 h_2^-\|_{C^{2,\alpha}}\lesssim \|f\|_{C^{2,\alpha}}.\]
Furthermore, recalling that $|\xi|^{2}=n^2+m^2$ we infer that
    \begin{equation}
    \begin{split}
    \reallywidehat{h}_2^{-}(\xi)|\xi|\left(\frac{n^2}{|\xi|^2}\chi_{\{\xi\neq0\}}-1\right)\frac{\cosh(|\xi|L)}{\sinh(|\xi|L)}&=-\reallywidehat{h}_2^{-}(\xi)|\xi|\frac{m^2}{|\xi|^2}\frac{\cosh(|\xi|L)}{\sinh(|\xi|L)}\chi_{\{\xi\neq0\}}-\reallywidehat{h}^-_2(0)\frac{\chi_{\{\xi=0\}}}{L}\\
    &=i\reallywidehat{\partial_1h^{-}_{2}}(\xi)\frac{m}{|\xi|}\frac{\cosh(|\xi|L)}{\sinh(|\xi|L)}\chi_{\{\xi\neq0\}}-\reallywidehat{h}^-_2(0)\frac{\chi_{\{\xi=0\}}}{L}.
    \end{split}
    \end{equation}
As a consequence, $\mathcal{Z}_{13}$ is given by the map
\begin{equation}\label{map:tercer:termino} f \to \sum_{\xi\in \Z^2}\reallywidehat{h}_2^{-}(\xi)\left(\frac{n^2}{|\xi|^2}\chi_{\{\xi\neq0\}}-1\right)\frac{|\xi|\cosh(|\xi|L)}{\sinh(|\xi|L)}\foul,
\end{equation}
and can be understood as the multiplier 
\[ \frac{im}{|\xi|}\frac{\cosh(|\xi|L)}{\sinh(|\xi|L)}\chi_{\{\xi\neq 0\}}-\frac{\chi_{\{\xi=0\}}}{L},\]
applied to the function $\partial_1 h_2^-$. It is easy to check that the multiplier 
\[ \frac{i m}{|\xi|}\frac{\cosh(|\xi|L)}{\sinh(|\xi|L)}= \frac{i m}{|\xi|}+\mathsf{r}(\xi),\]
is a Mihklin-Hörmander multiplier in $\R^{2}$ of degree $0$ plus a smoothing multiplier $\mathsf{r}(\xi)$, so the map \eqref{map:tercer:termino} is bounded from $C^{2,\alpha}(\T^2)\to C^{2,\alpha}(\T^2)$. As a consequence, we obtain the bound 
$$\|\mathcal{Z}_{13}\|_{C^{2,\alpha}}\leq C\|\partial_1 h^-_2\|_{C^{2,\alpha}}\leq C\|f\|_{C^{2,\alpha}}.$$
Combining the derived bounds for $\mathcal{Z}_{1\kappa}$ for $\kappa=1,\ldots, 4,$ shows the desired estimate for $\mathcal{Z}_{1}$ concluding the proof.
\end{proof}
Therefore, using the previous lemma we can easily show the following bounds for the functions $\textsf{G}_\ell$, $\ell=1,2$.
 \begin{coro}\label{coro:estimate:G}
Let $f\in C^{2,\alpha}(\partial\Omega), g\in C^{2,\alpha}(\partial\Omega_{-})$. Then, $\textsf{G}_\ell= 
\mathcal{T}_{0}^{-1}\left(g_{\ell}-\mathcal{Z}_{\ell}\right)\in C^{1,\alpha}(\mathbb{T}^{2})$ for $\ell=1,2$. More precisely, we have that
\begin{equation}\label{estimate:coro:estimate:G}
\norm{\textsf{G}_{\ell}}_{C^{1,\alpha}}\leq C\left( \norm{f}_{C^{2,\alpha}}+\norm{g}_{C^{2,\alpha}} \right), \quad \ell=1,2.
\end{equation}
 \end{coro}

 \begin{proof}
 The proof follows easily by combining estimate \eqref{estimate:Z1Z2} and recalling that
 \[ \widehat{\mathcal{T}_{0}^{-1} f}(\xi)=\frac{|\xi|\sinh(|\xi|L)}{\cosh(|\xi|L)-1} \widehat{f}(\xi).\]
Indeed, since
\[ \frac{|\xi|\sinh(|\xi|L)}{\cosh(|\xi|L)-1}=|\xi|+\mathsf{r}(\xi),\]
is a Mikhlin-Hörmander multiplier in $\R^{2}$ of degree $1$ plus a smoothing multiplier $\mathsf{r}(\xi)$ we conclude that
\begin{align}
\norm{\textsf{G}_{\ell}}_{C^{1,\alpha}}&=\norm{\mathcal{T}_{0}^{-1}\left(g_{\ell}-\mathcal{Z}_{\ell}\right)}_{C^{1,\alpha}} \leq C \norm{g_{\ell}-\mathcal{Z}_{\ell}}_{C^{2,\alpha}} \leq C \left( \norm{g}_{C^{2,\alpha}}+\norm{f}_{C^{2,\alpha}}\right), \quad \ell=1,2.
\end{align}

 \end{proof}

\subsection{H\"older estimates for the operator $\mathsf{H}_{\ell}$} \label{subsec:5:4}
In this subsection, we will derive H\"older estimates for the operator $\mathsf{H}_{\ell}$, $\ell=1,2$ given by 
\begin{align}
\mathsf{H}_{1}&=\mathcal{B}_y \left(b_1\partial_1j^3_0+b_2\partial_2j^3_0-\partial_{\ell}b_{3}j^\ell_0-b_3\partial_1j_0^1-b_3\partial_2j_0^2\right), \label{new:H1} \\
\mathsf{H}_{2}&=\mathcal{B}_x \left(b_1\partial_1j^3_0+b_2\partial_2j^3_0-\partial_{\ell}b_{3}j^\ell_0-b_3\partial_1j_0^1-b_3\partial_2j_0^2\right).\label{new:H2}
\end{align}
We recall that $\mathcal{B}_x, \mathcal{B}_y$ are defined via the multipliers 
\[ \widehat{\mathcal{B}_{x}f}(\xi)=-i\frac{m}{|\xi|^{2}}\widehat{f}(\xi), \ \widehat{\mathcal{B}_{y}f}(\xi)=-i\frac{n}{|\xi|^{2}}\widehat{f}(\xi).\]
\begin{prop}\label{Prop:H:estimadas:ori} 
Let $b\in C^{2,\alpha}(\Omega)$ and $j_0\in C^{1,\alpha}(\partial\Omega_{-})$. Then, $\textsf{H}_{\ell}\in C^{1,\alpha}(\mathbb{T}^{2})$ for $\ell=1,2$. More precisely, we have that
\begin{equation}\label{estimada:Horiginal}
\norm{\textsf{H}_{\ell}}_{C^{1,\alpha}}\leq C\left( \norm{b}_{C^{2,\alpha}}\norm{j_0}_{C^{1,\alpha}} \right), \quad \ell=1,2.
\end{equation}
Furthermore, we have that
\begin{equation}\label{diff:estim:H:orig}
\norm{\textsf{H}_{\ell}[b_{1}]-{\textsf{H}_{\ell}[b_{2}]}}_{C^{\alpha}}\leq C\left( \norm{b_{1}-b_{2}}_{C^{1,\alpha}}\norm{j_0}_{C^{\alpha}} \right), \quad \ell=1,2.
\end{equation}
\end{prop}
\begin{proof}
The proof follows easily by noticing that the multipliers associated to $\mathcal{B}_x, \mathcal{B}_y$ are Mikhlin-Hörmander multipliers in $\R^{2}$ of degree $-1$. Indeed, using the previous assertion we readily check that
\begin{align}
\norm{\textsf{H}_{\ell}}_{C^{1,\alpha}} &\leq C\norm{\left(b_1\partial_1j^3_0+b_2\partial_2j^3_0-\partial_{\ell}b_{3}j^\ell_0-b_3\partial_1j_0^1-b_3\partial_2j_0^2\right)}_{C^{\alpha}}\nonumber \\
&\leq C \norm{b}_{C^{2,\alpha}}\norm{j_0}_{C^{1,\alpha}}, \quad \ell=1,2.
\end{align}
The difference estimate \eqref{diff:estim:H:orig} follows directly by the linearity of the operator.
\end{proof}
\subsection{H\"older estimates for the operator $\mathsf{S}_{\ell}$}\label{subsec:5:5}
In this subsection, we established the estimates for the operator  $\mathsf{S}_{\ell}$, is defined as 
$\mathsf{S}_{\ell}= \mathcal{T}_0^{-1}\mathcal{S}_{\ell}$,  $\ell=1,2$ where
\begin{align}
\mathcal{S}_{1}&=\mathcal{R}_x\mathcal{R}_y\mathcal{A}(\delta j^1)+\left(\mathcal{R}^{2}_y-\text{id}\right)\mathcal{A}(\delta j^2), \label{new:S1} \\
\mathcal{S}_{2}&=\left(\mathcal{R}^{2}_x-\text{id}\right)\mathcal{A}(\delta j^1)+\mathcal{R}_x\mathcal{R}_y\mathcal{A}(\delta j^2). \label{new:S2}
\end{align}
Therefore, in order to give a precise meaning and derive the H\"older bounds for the operator $\mathsf{S}_{\ell}$, we first need to analyze the operator $\mathcal{A}$ introduced in \eqref{operator:A:general}. Although the expression \eqref{operator:A:general} seems complicated, the regularity of the operator can be understood in simpler way.  For a generic function $f\in C^{1,\alpha}(\Omega)$, the function  $\mathcal{A}(f)$ equals $\partial_{3}g(r)$ where $g\in C^{3,\alpha}(\Omega)$ is the unique solution to 
$$\left\{\begin{array}{ll}
  \Delta g=f,   & \text{in }\Omega, \\
  g=0,   & \text{on }\partial\Omega. 
\end{array}\right. $$
More precisely, we can show the following regularity result.
\begin{prop}\label{prop:estimate:S}
Let $\delta j^{\ell}\in C^{1,\alpha}(\Omega)$, for $\ell=1,2$. Then, $\mathsf{S}_{\ell}= \mathcal{T}_0^{-1}\mathcal{S}_{\ell}\in C^{1,\alpha}(\Omega)$. Moreover, the bound
\begin{equation}\label{estimate:S:PropS}
\norm{\mathsf{S}_{\ell}}_{C^{1,\alpha}(\Omega)}\leq C \left( \norm{\delta j^{1}}_{C^{1,\alpha}}+\norm{\delta j^{2}}_{C^{1,\alpha}}\right), \quad \ell=1,2,
\end{equation}
holds true.
\end{prop}
\begin{proof}
First, similarly as in the proof of Corollary \ref{coro:estimate:G}, we notice that the multiplier associated with $\mathcal{T}_0^{-1}$ 
is a Mikhlin-Hörmander multiplier in $\R^{2}$ of degree $1$ plus a smoothing multiplier.  Furthermore,  $\mathcal{R}_{x},\mathcal{R}_{y}$ are zero order operators with Mikhlin-Hörmander multipliers in $\R^{2}$ of degree $0$.
Thus,
\begin{equation}
\norm{\mathsf{S}_{\ell}}_{C^{1,\alpha}(\Omega)}\leq C\norm{\mathcal{S}_{\ell}}_{C^{2,\alpha}(\Omega)}\leq C \left(\norm{\mathcal{A}(\delta j^{1})}_{C^{2,\alpha}}+\norm{\mathcal{A}(\delta j^{2})}_{C^{2,\alpha}}\right), \ \ell=1,2. 
\end{equation}
To conclude, we notice that the operator $\mathcal{A}$ understood as above coincides with the Dirichlet to Neumann operator and therefore
\[ \norm{\mathcal{A}(\delta j^{\ell})}_{C^{2,\alpha}(\Omega)}\leq C \norm{\delta j^{\ell}}_{C^{1,\alpha}(\Omega)}, \ \ell=1,2.\]
Combining both bounds we show the desired estimate. 
\end{proof}
A straightforward consequence of Proposition \ref{prop:estimate:S} combined with the H\"older estimates for the transport problem collected in Corollary \ref{direct:coro:proptransport} reads as follows
\begin{coro}\label{teoremon2}
There exists $M_0$ small enough so that for every $b\in C^{2,\alpha}(\Omega;\R^3)$ with $\|b\|_{C^{2,\alpha}}\leq M<M_0$, there exists a constat $C=C(\alpha,L)$ such that 
$$\norm{\mathsf{S}_{\ell}}_{C^{1,\alpha}(\Omega)}\leq CM \|j_0\|_{C^{1,\alpha}}.$$
 Moreover, we have that 
$$\norm{\mathsf{S}_{\ell}[b_{1}]-\mathsf{S}_{\ell}[b_{2}]}_{C^{\alpha}}\leq C\|b_1-b_2\|_{C^{1,\alpha}}.$$
\end{coro}

To conclude this subsection, let us first show a result regarding H\"older estimates of the functions $\Lambda$ and $\Theta$. 
\begin{prop}\label{estimations:Lambda:Theta}
Let $M_0$ be sufficiently small and let $b\in C^{2,\alpha}(\Omega;\R^3)$ with $\|b\|_{C^{2,\alpha}}\leq M<M_0$. Moreover,  define the functions
\begin{equation}\label{def:lambda:theta}
\Lambda(r,z)=\Psi_z(r,z)-r\qquad \text{and} \qquad \Theta(r,z)=|J_g(r,z)|-1.
\end{equation}
Here, $\Psi_{z}$ is the unique solution of the system of ODEs \eqref{ode} with $v=b$ and $J_g$ denotes the  determinant of the Jacobian of $\Psi_z$. The we have that
$\Lambda \in C^{2,\alpha}(\Omega),  \Theta\in C^{1,\alpha}(\Omega)$ and
\begin{equation}\label{estimate:1:lambda:theta}
\|\Lambda\|_{C^{2,\alpha}}\leq CM, \quad \|\Theta\|_{C^{1,\alpha}}\leq CM.
\end{equation}
Moreover, for $b_{1},b_{2}\in C^{2,\alpha}(\Omega;\R^3)$ with $\|b_{1},b_{2}\|_{C^{2,\alpha}}\leq M<M_0$ we have that
\begin{align}
 \|\Lambda[b_1]-\Lambda[b_2]\|_{C^{1,\alpha}}&\leq C\|b_1-b_2\|_{C^{1,\alpha}}, \label{estimate:lambda:diff} \\
 \|\Theta[b_1]-\Theta[b_2]\|_{C^{\alpha}}&\leq C\|b_1-b_2\|_{C^{\alpha}}. \label{estimate:theta:diff}
\end{align}
\end{prop}

\begin{proof}
The proof follows standard arguments to compute the dependence of the solution of an ODE in their parameters. More precisely, it should be notice that $\Lambda$ basically satisfies the ODE \eqref{ode} of Proposition \ref{prop:regularidad:odes} and therefore we can repeat the same ideas as the one developed there. In order to do so, we have to control incremental quotients combined with Gr\"onwall type estimates. A bound similar to \eqref{estimate:1:lambda:theta} is shown in \cite[Lemma 3.7]{Alo-Velaz-2021} in the two-dimensional case and for $C^{1,\alpha}$ regularity instead of $C^{2,\alpha}$ regularity. Moreover, in order to show \eqref{estimate:lambda:diff}-\eqref{estimate:theta:diff} we calculate the ODE satisfied by the differences of the solutions which defined $\Lambda,\Theta$ (cf. \eqref{ode} and \eqref{def:lambda:theta}) for $v=b_{1}, v=b_{2}$.
\end{proof}

\subsection{Solution to the integral equation for $j_0$}\label{subsec:5:6}
In this subsection we are interested in studying the existence of a solution $j_0\in C^{1,\alpha}(\partial \Omega_{-})$ to the integral equation
\begin{align}
   j^{1}_{0}&=(\mathcal{R}^{2}_x-\text{id})\sum_{\kappa=1}^4 \widetilde{\textsf{T}_{\kappa} j_0^1}+\mathcal{R}_x\mathcal{R}_y\sum_{\kappa=1}^4\textsf{T}_{\kappa} j_0^2 + \mathsf{S}_{2}[\widetilde{\delta j^{1}},\delta j^{2}]-\mathsf{H}_{2}+\left(\text{id}-\mathcal{R}_x^{2}\right)\widetilde{\textsf{G}_2}+\mathcal{R}_x\mathcal{R}_y\textsf{G}_1 \nonumber \\
    &\quad \quad -\langle(\partial_1 g_2-\partial_2 g_1)g_1\rangle-\langle j^1_0 f\rangle, \label{alm:new4:final2} \\
    j^{2}_{0}&=\left(\mathcal{R}^{2}_y-\text{id}\right)\sum_{\kappa=1}^4\widetilde{\textsf{T}_\kappa j_0^2}+\mathcal{R}_x\mathcal{R}_y\sum_{\kappa=1}^4\textsf{T}_\kappa j^1_0- \mathsf{S}_{1}[\delta j^{1},\widetilde{\delta j^{2}}]+\mathsf{H}_{1}  -\left(\text{id}-\mathcal{R}_y^{2}\right)\widetilde{\textsf{G}_1}-\mathcal{R}_x\mathcal{R}_y\textsf{G}_2  \nonumber \\
    &\quad \quad  -\langle(\partial_1 g_2-\partial_2 g_1)g_2\rangle-\langle j^2_0 f\rangle.  \label{alm:new5:final2}
\end{align}
Recall that in the previous subsections, namely in Subsection \ref{subsec:52}-Subsection \ref{subsec:5:5}, we have provided a rigorous approach to define all the operators contained in \eqref{alm:new4:final2}-\eqref{alm:new5:final2}. Moreover, before solving the integral equation, it is convenient to introduce the modified operators $\mathsf{H}^{\star}_{\ell}$, $\ell=1,2$ given by
\begin{align}
\mathsf{H}^{\star}_{1}&=\mathsf{H}_{1}-\mathcal{B}_y \left(b_1\partial_1j^3_0+b_2\partial_2j^3_0-\partial_{3}b_{3}j^3_0 \right), \label{H1:Star1} \\
\mathsf{H}^{\star}_{2}&=\mathsf{H}_{2}-\mathcal{B}_{x} \left(b_1\partial_1j^3_0+b_2\partial_2j^3_0-\partial_{3}b_{3}j^3_0 \right). \label{H2:Star2}
\end{align}
The modified operators $\mathsf{H}^{\star}_{\ell}$ do not depend on $j_0^{3}$ which is already given in terms of the boundary data $g$. Indeed, recalling \eqref{alm:new3:final} we find that
\[ j_0^{3}=\partial_1 g_2-\partial_2 g_{1}, \ \mbox{on } \partial\Omega_{-}. \]
Hence, in order to solve the integral equation \eqref{alm:new4:final2}-\eqref{alm:new5:final2}
we introduce the following operator $\Upsilon_{\ell}: C^{1,\alpha}(\T^2)\times C^{1,\alpha}(\T^2)\longrightarrow C^{1,\alpha}(\T^2)$, $\ell=1,2$ given by
\begin{align}
    \Upsilon_{1}(j^{1}_{0},j^{2}_{0})&=(\mathcal{R}^{2}_x-\text{id})\sum_{\kappa=1}^4 \widetilde{\textsf{T}_{\kappa} j_0^1}+\mathcal{R}_x\mathcal{R}_y\sum_{\kappa=1}^4\textsf{T}_{\kappa} j_0^2 + \mathsf{S}_{2}[\widetilde{\delta j^{1}},\delta j^{2}]-\mathsf{H}_{2}^{\star} -\langle j^1_0 f\rangle, \label{alm:new4:final3} \\
     \Upsilon_{2}(j^{1}_{0},j^{2}_{0})&=\left(\mathcal{R}^{2}_y-\text{id}\right)\sum_{\kappa=1}^4\widetilde{\textsf{T}_\kappa j_0^2}+\mathcal{R}_x\mathcal{R}_y\sum_{\kappa=1}^4\textsf{T}_\kappa j^1_0- \mathsf{S}_{1}[\delta j^{1},\widetilde{\delta j^{2}}]+\mathsf{H}_{1}^{\star}-\langle j^2_0 f\rangle.  \label{alm:new5:final3}
\end{align}
Therefore, if $\Upsilon=(\Upsilon_1,\Upsilon_2)$, one could argue naively that the solution of the integral equations \eqref{alm:new4:final2}-\eqref{alm:new5:final2} should be given by 
\begin{equation}\label{equation:j_0:modified}
(j_0^1,j_0^2)=\left(\textbf{1}+\Upsilon\right)^{-1}\left(\mathscr{G}_1,\mathscr{G}_2\right),
\end{equation}
where 
\begin{align}
  \mathscr{G}_1&=\left(\text{id}-\mathcal{R}_x^{2}\right)\widetilde{\textsf{G}_2}+\mathcal{R}_x\mathcal{R}_y\textsf{G}_1 -\langle(\partial_1 g_2-\partial_2 g_1)g_1\rangle+\mathcal{B}_{x}\left(b_1\partial_1j^3_0+b_2\partial_2j^3_0-\partial_{3}b_{3}j^3_0 \right), \label{G1:modified} \\
        \mathscr{G}_2&= -\left(\text{id}-\mathcal{R}_y^{2}\right)\widetilde{\textsf{G}_1}-\mathcal{R}_x\mathcal{R}_y\textsf{G}_2 -\langle(\partial_1 g_2-\partial_2 g_1)g_2\rangle-\mathcal{B}_{y} \left(b_1\partial_1j^3_0+b_2\partial_2j^3_0-\partial_{3}b_{3}j^3_0 \right).\label{G2:modified}
\end{align}
Let us a provide a rigorous proof of the previous naively approach. 
\begin{prop}\label{eqintegral}
Let $M_0$ be sufficiently small (in particular, $M_{0}\leq \displaystyle{\delta_{0},\delta_{1}}$) and let $b\in C^{2,\alpha}(\Omega;\R^3)$ with $\|b\|_{C^{2,\alpha}}\leq M<M_0$.  
   Moreover, assume that 
   \begin{equation}\label{smallness:eqintegral}
   \max\left\{\|f\|_{C^{2,\alpha}}\,\|g\|_{C^{2,\alpha}} \right\}\leq M.
   \end{equation}
Then $\left(\textbf{1}+\Upsilon\right)$ is invertible in $C^{1,\alpha}$. More precisely, there exists an operator $\Pi=\left(\textbf{1}+\Upsilon\right)^{-1}$ satisfying the bound
\begin{equation}\label{Pi:estimate:prop}
\norm{\Pi}_{C^{1,\alpha}}\leq C \norm{b}_{C^{1,\alpha}}.
\end{equation}
Furthermore, for any two magnetic fields $b_1$ and $b_2$ satisfying \eqref{smallness:eqintegral} we have that
\begin{equation}\label{estimate:2:eq:integral}
\|\Pi[b_1]-\Pi[b_2]\|_{\mathcal{L}(C^{\alpha})}\leq C\|b_1-b_2\|_{C^{1,\alpha}}.
\end{equation}
\end{prop}
\begin{proof}
Recalling \eqref{alm:new4:final3}-\eqref{alm:new5:final3} and using the fact that the operators $\mathcal{R}_{x},\mathcal{R}_{y}$ are Mikhlin-H\"ormander mutipliers in $\mathbb{R}^{2}$ of degree $0$, we readily check that
\begin{align*}
\norm{\Upsilon_{1}}_{C^{1,\alpha}} &\leq C\left( \sum_{\kappa=1}^4 \norm{\widetilde{\textsf{T}_{\kappa} j_0^1}}_{C^{1,\alpha}} + \sum_{\kappa=1}^4\norm{\textsf{T}_{\kappa} j_0^2}_{C^{1,\alpha}} +\norm{\mathsf{S}_{2}[\widetilde{\delta j^{1}},\delta j^{2}]}_{C^{1,\alpha}} + \norm{\mathsf{H}_{2}^{\star}}_{C^{1,\alpha}} + \norm{\langle j^1_0 f\rangle}_{C^{1,\alpha}}  \right), \\
\norm{\Upsilon_{2}}_{C^{1,\alpha}} &\leq C\left( \sum_{\kappa=1}^4 \norm{\widetilde{\textsf{T}_{\kappa} j_0^2}}_{C^{1,\alpha}} + \sum_{\kappa=1}^4\norm{\textsf{T}_{\kappa} j_0^1}_{C^{1,\alpha}} +\norm{\mathsf{S}_{1}[\delta j^{1},\widetilde{\delta j^{2}}]}_{C^{1,\alpha}} + \norm{\mathsf{H}_{1}^{\star}}_{C^{1,\alpha}} + \norm{\langle j^2_0 f\rangle}_{C^{1,\alpha}}  \right).
\end{align*}
Applying estimates \eqref{estimada:teoremon1}, \eqref{estimada:Horiginal} and \eqref{estimate:S:PropS} we find that for $\ell=1,2$
\begin{align*}
\norm{\Upsilon_{{\ell}}}_{C^{1,\alpha}} &\leq C\left( (\|\Lambda\|_{C^{1,\alpha}}+\|\Theta\|_{C^{1,\alpha}}+\norm{b}_{C^{2,\alpha}}+\norm{f}_{C^{2,\alpha}})(\norm{j_0^{1}}_{C^{1,\alpha}}+\norm{j_0^{2}}_{C^{1,\alpha}}) + \norm{\delta j^{1}}_{C^{1,\alpha}}+ \norm{\delta j^{1}}_{C^{1,\alpha}} \right).
\end{align*}
Moreover, using Corollary \ref{teoremon1}, Proposition \ref{estimations:Lambda:Theta} and the smallness assumption \eqref{smallness:eqintegral} we obtain that
\begin{align*}
\norm{\Upsilon_{{\ell}}}_{C^{1,\alpha}} &\leq CM (\norm{j_0^{1}}_{C^{1,\alpha}}+\norm{j_0^{2}}_{C^{1,\alpha}}), \ \ell=1,2,
\end{align*}
and thus $\|\Upsilon\|_{\mathcal{L}(C^{1,\alpha})}$.  Since $M$ is small enough, $\|\Upsilon\|_{\mathcal{L}(C^{1,\alpha})}<1$ and hence, by means of a classical Neumann series argument (c.f. \cite{Alo-Velaz-2022,ReedsSimon}), we conclude that $\textbf{1}+\Upsilon$ is invertible. In addition, we also obtain the simple estimate 
\begin{equation}
\norm{\Pi}_{C^{1,\alpha}}=\norm{(\textbf{1}+\Upsilon)^{-1}}_{C^{1,\alpha}}=\norm{\displaystyle\sum_{n=0}^{\infty} (-1)^{n}\Upsilon^{n}}_{C^{1,\alpha}}\leq \frac{1}{1-CM}(\norm{j_0^{1}}_{C^{1,\alpha}}+\norm{j_0^{2}}_{C^{1,\alpha}}).
\end{equation}
Here, we have used that $\|\Upsilon\|_{\mathcal{L}(C^{1,\alpha})}$. Moreover, for $n=0$, we have that $\Upsilon^{0}=\mathbf{1}$.
For the second assertion, we expand $\Pi=(\textbf{1}+\Upsilon)^{-1}$ in its Neumann series expansion, namely, 
\[ \Pi[b_1]-\Pi[b_2]=\sum_{n=1}^\infty (-1)^{n}\left(\Upsilon[b_1]^n-\Upsilon[b_2]^n\right) 
\]
Next, noting that $ A^n-B^n=\sum_{i=1}^n A^{n-i}(A-B)B^{i-1}$ we have that
\[ \|A^n-B^n\|_{\mathcal{L}(C^{\alpha})}\leq \|A-B\|_{\mathcal{L}(C^{\alpha})}\sum_{i=1}^n\|A\|_{\mathcal{L}(C^{\alpha})}^{n-i}\|B\|^{i-1}_{\mathcal{L}(C^{\alpha})}.\]
As a result, we conclude that 
\[
 \|\Pi[b_1]-\Pi[b_2]\|_{\mathcal{L}(C^{\alpha})}\leq \|\Upsilon[b_1]-\Upsilon[b_2]\|_{\mathcal{L}(C^{\alpha})}\sum_{n=1}^\infty \left(\|\Upsilon[b_1]\|_{\mathcal{L}(C^{\alpha})}+\|\Upsilon[b_2]\|_{\mathcal{L}(C^{\alpha})}\right)^n.
\]
Estimate \eqref{estimate:2:eq:integral} follows by recalling the precise expression of the operator $\Upsilon$ given in \eqref{alm:new4:final3}-\eqref{alm:new5:final3} combined with smallness assumption \eqref{smallness:eqintegral}, the difference estimates for $\textsf{T}^d_{\kappa}$, $\mathsf{S}_{\ell}$, $\mathsf{H}_{\ell}$ and $\Lambda, \Theta$ in Proposition \ref{diferencia}, Proposition \ref{teoremon2}, Proposition \ref{Prop:H:estimadas:ori}, Proposition \ref{estimations:Lambda:Theta}, respectively. 
\end{proof}

 \begin{coro}\label{coro:estimate:j0}
 Let the hypothesis of Proposition \ref{eqintegral} hold. Then there exists a solution $(j_0^{1},j_0^{2})\in C^{1,\alpha}$ to \eqref{equation:j_0:modified} given by 
 \begin{equation}\label{j_0:solution:coro}
 (j_0^1,j_0^2)=\Pi\left(\mathscr{G}_1,\mathscr{G}_2\right),
 \end{equation}
with $\mathscr{G}_1$ and $\mathscr{G}_2$ given in \eqref{G1:modified} and \eqref{G2:modified}, respectively. Furthermore,
\begin{equation}\label{bound:finalj0}
\norm{(j_0^1,j_0^2)}_{C^{1,\alpha}}\leq C\left(  \norm{g}_{C^{2,\alpha}}+\norm{f}_{C^{2,\alpha}}  \right).
 \end{equation}
 \end{coro}
 \begin{proof}
The fact that $j_0$ as given in \eqref{j_0:solution:coro} is a solution to \eqref{equation:j_0:modified} is a direct consequence of the definition of the operator $\Pi$ in Proposition \ref{eqintegral}. Moreover, recalling the multipliers associated to $\mathcal{R}_{x},\mathcal{R}_{y}$ and $\mathcal{B}_x, \mathcal{B}_y$ are Mikhlin-Hörmander multipliers in $\R^{2}$ of degree $0$ and $-1$, respectively we find that
\begin{align}
\norm{\mathscr{G}_{\ell}}_{C^{1,\alpha}} &\leq C\left( \norm{\widetilde{\mathsf{G}}_{\ell}}_{C^{1,\alpha}} +\norm{\mathsf{G}_{\ell}}_{C^{1,\alpha}} +\norm{g}^{2}_{C^{2,\alpha}} + \norm{b}\norm{j_0^3}_{C^{1,\alpha}} \right), \nonumber \\
&\leq C\left( \norm{g}_{C^{2,\alpha}}+\norm{f}_{C^{2,\alpha}} +\norm{g}^{2}_{C^{2,\alpha}} + \norm{b}\norm{g}_{C^{2,\alpha}} \right), 
\end{align}
where in the second inequality we have used estimate \eqref{estimate:coro:estimate:G} of Proposition \ref{coro:estimate:G} and the fact that $j_0^{3}=\partial_1 g_2-\partial_2 g_{1}, \ \mbox{on } \partial\Omega_{-}.$ Therefore, this estimate combined with \eqref{Pi:estimate:prop} shows bound \eqref{bound:finalj0}.
 \end{proof}

To conclude this section we outlining the reason why the solution of $(j_0^1,j_0^2)=\Pi(\mathscr{G}_1,\mathscr{G}_2)$ yields the unique choice of $(j_0^1,j_0^2)$ that ensures that the new magnetic field $W$ satisfies the correct boundary conditions. Note that in Theorem \ref{wellposedness}, we proved the existence and uniqueness of the div-curl problem given $j$ divergence-free, $f\in C^{2,\alpha}(\partial\Omega)$ and $J_1,J_2\in\R$. In such construction, we employed a combination of a scalar and a vector potential for the magnetic field, while in Section \ref{sec:linearized} and \ref{sec:4} we employed an a priori different construction. In the following, we show why this solution is the unique solution of the div-curl system, by using the uniqueness result we showed in Proposition \ref{wellposedness}. To that end, we take the vector field $A$ defined as the unique solution for the following elliptic problem
\begin{equation}\label{problemaA:elliptic}
    \left\{\begin{array}{rl}
        -\Delta A=j & \text{ in }\Omega \\
         A_1=h_1 & \text{ on }\partial\Omega\\
         A_2=h_2 & \text{ on }\partial\Omega\\
         \partial_3 A_3=-\partial_1 A_1-\partial_2 A_2 & \text{ on }\partial\Omega
    \end{array}\right. ,
\end{equation}
where $h_1$ and $h_2$ are defined in $\partial\Omega_-$ and $\partial\Omega_+$ in \eqref{aux}. Notice that the functions $h_1$ and $h_2$ only gain extra regularity in one of the variables, so we cannot expect to directly obtain $C^{3,\alpha}$  bounds for $A$. However, we can separate \eqref{problemaA:elliptic} into 
\begin{equation}
    \left\{\begin{array}{rl}
        -\Delta A^1=0 & \text{ in }\Omega \\
         A^1_1=h_1 & \text{ on }\partial\Omega\\
         A^1_2=h_2 & \text{ on }\partial\Omega\\
         \partial_3 A^1_3=-\partial_1 A_1-\partial_2 A_2 & \text{ on }\partial\Omega
    \end{array}\right. ,\quad \text{and}\quad 
    \left\{\begin{array}{rl}
        -\Delta A^2=j & \text{ in }\Omega \\
         A^2_1=0 & \text{ on }\partial\Omega\\
         A^2_2=0 & \text{ on }\partial\Omega\\
         \partial_3 A^2_3=0 & \text{ on }\partial\Omega 
    \end{array}\right. .
\end{equation}
The left-hand side problem yields an harmonic $A^1$ function which can be extended to a $C^{2,\alpha}$ function in $\Omega$. Notice that, since it is harmonic in the interior, it must be smooth. On the other hand $A^2$ is clearly in $C^{3,\alpha}$. The resulting $A$ needs to have zero divergence. Indeed, since $A$ is $C^{3}$ in the interior of our domain, we can permute the Laplacian with the divergence, and conclude that $\text{div}\,A$ is harmonic and with zero boundary values. Thus, it is zero. Therefore, and since the maps $h_1$ and $h_2$ are constructed to this end, we conclude that $\nabla\times A$ yields the unique solution for \eqref{divcurlotravez}. Now, once we have seen that $\nabla\times A$ yields the unique $C^{2,\alpha}$ solution of the div-curl system, we can take Fourier coefficients to yield the relations that $(j_0^1,j_0^2)$ need to satisfy which, together with the divergence free condition, are equivalent to the integral equations \eqref{alm:new4:final2}-\eqref{alm:new5:final2}, whose existence and uniqueness of solutions have been proved in the previous result, cf. Corollary \ref{coro:estimate:j0}.

\section{The fixed point argument and proof of Theorem \ref{mainteor}}\label{sec:6}
In this final section, we will provide the fixed point argument, this is, we will defined an adequate operator $\Gamma$ which has a fixed point such that $B=(0,0,1)+b$ is a solution satisfying the boundary value problem \eqref{mhs2}, thus proving Theorem \ref{mainteor}.

So far, we have constructed a map that takes an initial perturbation $b\in C^{2,\alpha}$, and maps it to a current $j=\Pi[b]\in C^{1,\alpha}$ so that the corresponding solution of the div-curl system, 

\begin{equation}\label{pesao}
    \left\{\begin{array}{rl}
        \nabla \times W=j & \text{ in }\Omega \\
        \text{div}\,W=0 & \text{ in }\Omega\\
        W\cdot n =f & \text{ on }\partial\Omega\\
    \end{array}\right.
    \,\text{ and }\,
    \left\{\begin{array}{l}
        \int_{\{x=0\}}W\cdot d\vec{S}=J_1[j_0,f,g,b]\\
        \int_{\{y=0\}}W\cdot d\vec{S}=J_2[j_0,f,g,b]
    \end{array}\right.,
\end{equation}
satisfies (formally) the boundary condition $W_\tau=g$. Denote $\mathscr{B}[J,f,j]=W[b,j]$ such solution. In this section we will prove that this resulting map
\begin{equation}\label{gama}
b\mapsto \Gamma[b]=\mathscr{B}\left[J,f,\Pi[b]\right]
\end{equation}
attains a unique fixed point in an adequate space, that corresponds to the solution of our problem. 

Due to Proposition \ref{wellposedness}, the problem \eqref{pesao} is well-posed in $C^{2,\alpha}$, so that the map $\Gamma$ is well defined.  Next, we will use all the previous results we have obtained in the previous section to prove that \eqref{gama}  admits a unique fixed point in some particular subset of $C^{1,\alpha}(\Omega)$. This fixed point will be a solution to our problem.

\subsection{Fixed point for the map $\Gamma$ }

In order to prove that a fixed point for the map $\Gamma$ exists, we must make precise where such fixed point is attained, as well as a precise description on how the map $\Gamma$ is constructed. We begin with this second task.

First of all, we need to fix a (small) $M>0$, so that $M<M_0$, and $M_0$ is given by Proposition \ref{existenciatransporte}. Then, given an initial perturbation $b\in C^{2,\alpha}$ satisfying $\|b\|_{C^{2,\alpha}(\Omega)}\leq M$, we can construct a map 

\begin{equation*}
    \begin{array}{rcl}
    T[b]:C^{1,\alpha}(\T^2)\times C^{1,\alpha}(\T^2)   & \longrightarrow & C^{1,\alpha}(\Omega;\R^3)  \\
      (j_0^1,j_0^2)   & \mapsto &(j^1,j^2,j^3) 
    \end{array},
\end{equation*}
where $j=(j^1,j^2,j^3)$ is the unique solution for the equation \eqref{transporte:1} with initial condition $j_0=(j^1_0,j^2_0,j_0^3)$, and $j_0^3=\partial_2 g_1-\partial_1 g_2$. Due to Proposition \ref{existenciatransporte}, this map is well defined and is bounded. 

Now, after making $M$ smaller, as in Proposition \ref{eqintegral}, if we have boundary data $f\in C^{2,\alpha}(\partial\Omega)$ and $g\in C^{2,\alpha}(\partial\Omega_-)$ satisfying $\|f\|_{C^{2,\alpha}(\partial\Omega)},\|g\|_{C^{2,\alpha}(\partial\Omega_-)}\leq M$, we can take $(j_0^1,j_0^2)$ to be equal to 

$$(j_0^1,j_0^2)=\Pi[b](\mathscr{G}_1,\mathscr{G}_2)=\left(\textbf{1}+\Upsilon[b]\right)^{-1}(\mathscr{G}_1,\mathscr{G}_2),$$
where $\mathscr{G}_1$ and $\mathscr{G}_2$ are given by \eqref{G1:modified} and \eqref{G2:modified}. Note that these functions are well defined and belong to $C^{1,\alpha}(\T^2)$, so the formula for $(j_0^1,j_0^2)$ is well defined. 

It is worth noticing that the resulting current $j=T[b](\Pi[b](\mathscr{G}_1,\mathscr{G}_2))$ is divergence free by construction. Therefore, we can solve the div-curl problem \eqref{divcurlotravez}, where we take $(J_1,J_2)$ as in \eqref{j1}-\eqref{j2}. We then conclude that the map $\Gamma$ is given by the following composition
$$\Gamma[b]=\mathscr{B}\left[f,J,T[b]\left(\Pi[b](\mathscr{G}_1,\mathscr{G}_2)\right)\right].$$

We next prove that, in some subspace of $C^{1,\alpha}$, this operator admits a fixed point. 

\begin{teor}
    Let $f\in C^{2,\alpha}(\partial\Omega)$ satisfying the integrability condition \eqref{integrability:condition} and $g\in C^{2,\alpha}(\partial\Omega_-)$. There exists $\varepsilon_0>0$ and $M=M(L,\alpha)$ small enough such that if
    $$\|f\|_{C^{2,\alpha}(\partial\Omega)}+\|g\|_{C^{2,\alpha}(\partial\Omega_-)}\leq M\varepsilon_0 ,$$
    then $\Gamma$ maps $B_M(C^{2,\alpha})$ into itself. Furthermore, the operator $\Gamma$ admits a unique fixed point in $B_M(C^{2,\alpha})$.
\end{teor}
\begin{proof}

For the first step, we use that all of the operators involved in the definition of $\Gamma$ satisfy suitable bounds, so that 

\begin{equation*}
    \begin{split}
        \|\Gamma[b]\|_{C^{2,\alpha} (\Omega)}&\leq C\left(\|f\|_{C^{2,\alpha}(\partial\Omega)}+|J|+\|T[b]\left(\Pi[b](\mathscr{G}_1,\mathscr{G}_2)\right)\|_{C^{1,\alpha}}\right)\\
        &\leq C\left(\|f\|_{C^{2,\alpha}(\partial\Omega)}+|J|+\|g\|_{C^{2,\alpha}(\partial\Omega_-)}+\|\Pi[b](\mathscr{G}_1,\mathscr{G}_2)\|_{C^{1,\alpha}(\partial\Omega_-)}\right)\\
        &\leq C \left(\|f\|_{C^{2,\alpha}(\partial\Omega)}+\|g\|_{C^{2,\alpha}(\partial\Omega_-)}\right)\leq CM\epsilon_{0}.
    \end{split}
\end{equation*}

In the first line, we have used the estimates for $\mathscr{B}$ from Theorem \ref{wellposedness}, in the second line we have used boundedness of the operator $T[b]$, which arises from Theorem \ref{existenciatransporte}, and the last line comes from writing the explicit expresion for $J$, the fact that $\|\Pi\|_{C^{1,\alpha}}$ is bounded, and the bounds for the $\mathscr{G}_{\ell}$ defined in \eqref{G1:modified} and \eqref{G2:modified}. Thus, if we take $\varepsilon_0<1/(C+1)$, we get that $\Gamma$ maps $B_M(C^{2,\alpha})$ into itself. 

In order to prove that a fixed point is attained, we use Banach's fixed point. However, we will not equip $B_M(C^{2,\alpha})$ with the $C^{2,\alpha}$ norm, but with the $C^{1,\alpha}$ one. The reason behind this is that all the estimates we have made in terms of differences are given in terms of $C^{1,\alpha}$ norms. We can use Banach's fixed point theorem, since by Arzelà-Ascoli the closed balls of $C^{2,\alpha}$ are compact in $C^{1,\alpha}$, c.f. \cite{Alo-Velaz-2022}. Then, we can see that

{\small
\begin{equation*} 
    \begin{split} 
        \|\Gamma[b_1]-\Gamma[b_2]\|_{C^{1,\alpha}}&\leq C\|b_1-b_2\|_{C^{1,\alpha}}\\
        &=\|\mathscr{B}[f,J^1,T[b_1]\left(\Pi[b_1](\mathscr{G}_1,\mathscr{G}_2)\right)]-\mathscr{B}[f,J^2,T[b_2]\left(\Pi[b_2](\mathscr{G}_1,\mathscr{G}_2)\right)]\|_{C^{1,\alpha}}\\
        &\leq C\left(|J^1-J^2|+\|T[b_1]\left(\Pi[b_1](\mathscr{G}_1,\mathscr{G}_2)\right)-T[b_2]\left(\Pi[b_2](\mathscr{G}_1,\mathscr{G}_2)\right)\|_{C^{\alpha}(\Omega)}\right)\\
        &\leq C\left(|J^1-J^2|+\|(\Pi[b_1]-\Pi[b_2])(\mathscr{G}_1,\mathscr{G}_2)\|_{C^{\alpha}}+\|\Pi[b_1](\mathscr{G}_1,\mathscr{G}_2)\|_{C^{1,\alpha}}\|b_1-b_2\|_{C^{1,\alpha}}\right).
    \end{split}
\end{equation*}
}

Invoking Proposition \ref{eqintegral}, we have that $\|\Pi[b_1]-\Pi[b_2]\|_{\mathcal{L}(C^{\alpha})}\leq C\|b_1-b_2\|_{C^{1,\alpha}}$. Therefore, 

$$\|\Gamma[b_1]-\Gamma[b_2]\|_{C^{1,\alpha}}\leq C\left(\|f\|_{C^{2,\alpha}(\partial\Omega)}+\|g\|_{C^{2,\alpha}(\partial\Omega_-}\right)\|b_1-b_2\|_{C^{1,\alpha}}\leq CM\varepsilon_0\|b_1-b_2\|_{C^{1,\alpha}}$$

Taking $\varepsilon_0\leq 1/(CM+1)$, we conclude that $\Gamma$ is a contraction mapping in the complete metric space $X=(B_M(C^{2,\alpha}),\|\cdot\|_{C^{1,\alpha}})$, so by Banach's fixed point, we conclude that it admits a unique fixed point. 
\end{proof}

\subsection{Proof of Theorem \ref{mainteor}} During this work, we have constructed a map $\Gamma$ which attains a unique fixed point $b$ in the metric space $(B_M(C^{2,\alpha}),\|\cdot\|_{C^{1,\alpha}})$ that solves the div-curl problem \eqref{pesao}. The last step remaining to prove Theorem \ref{mainteor} is checking that this fixed point is indeed the solution for \eqref{mhs1} with our set of boundary conditions. Notice that, by construction, $B=(0,0,1)+b$ satisfies the normal boundary condition, namely, $B\cdot n=1+f$ with $n$ being the outer normal vector in $\partial\Omega_+$ and the inner normal vector in $\partial\Omega_-$. Furthermore, since in the very definition of $\Gamma$ we have ensured that $\Gamma[b]$ satisfies the tangential boundary condition (c.f. check the last digression in Section \ref{sec:5}), we just need to check that there exists a unique well defined pressure function $p\in C^{2,\alpha}(\Omega)$ such that 
\[ j\times B=\nabla p, \ \mbox{in } \Omega,\]
with $j=\nabla\times B$.
One might think that the pressure can be written simply as 
\begin{equation}\label{presion}
p=\int_0^{\textbf{\textit{x}}}j\times B,\ \mbox{in } \Omega,
\end{equation}
where the integral is taken between any curve that connects $(0,0,0)$ and $\textbf{\textit{x}}=(x,y,z)$. However, for this to make sense, we need to ensure the following two conditions:
\begin{itemize}
    \item There cannot be any dependence on the chosen curve. At least near zero, this is achieved only if $j\times B$ has zero curl. This is indeed the case, since
    $$\nabla\times (j\times B)=(B\cdot\nabla)-j(\nabla\cdot A)+j(\nabla\cdot B)-(j\cdot\nabla)B.$$
    As $j$ and $B$ are divergence free by construction, the last two items vanish. On the other hand, since $b$ is a fixed point of $\Gamma$, it satisfies $(j\cdot\nabla)B-(B\cdot\nabla)j=0$, so we get the result.
    \item In our domain the points $(x,y,z)$ and $(x+2\pi l_1,y+2\pi l_2,z)$ both represent the same point for every $l_1,l_2\in\Z$. Thus, in order for $p$ to be a well defined function in $\T^2\times[0,L]$, we need 
\[ \int_0^{2\pi}\left(j\times B\right)_1(x,y,0) \ dx=0, \ \forall\,y\in\R, \quad  \int_0^{2\pi}\left(j\times B\right)_{2}(x,y,0) \, dy=0, \ \forall\,x\in\R. \]
Using the fact that $\nabla\times(j\times B)=0$ and the periodicity of the integrand, we infer that the integrals above are constant functions of $y$ and $x$, respectively. Hence, we can then integrate on the other variable, so that these two conditions are equivalent to 

    \begin{equation}\label{j1}
        \f\int_0^{2\pi}\int_0^{2\pi}\left(j\times B\right)_1(x,y,0) \ dxdy=\langle j_0^2\rangle+\langle j_0^2f\rangle-\left\langle\left(\partial_1 g_2-\partial_2 g_1\right)g_2\right\rangle=0,
    \end{equation}

    \begin{equation}\label{j2}
        \f\int_0^{2\pi}\int_0^{2\pi}\left(j\times B\right)_2(x,y,0) \ dxdy=-\langle j_0^1\rangle-\langle j_0^1f\rangle+\left\langle\left(\partial_1 g_2-\partial_2 g_1\right)g_1\right\rangle=0.
    \end{equation}
Note that these equations are satisfied due to our particular tailor-made choice of $J_1$ and $J_2$ in $\eqref{J1}-\eqref{J2}$. This also justifies the selection of the constants $J_{1},J_{2}$, which might have looked unjustified and arbitrary at the beginning.
\end{itemize}

Summarizing, we have defined a solution pair $(B,p)\in C^{2,\alpha}(\Omega)$ that satisfies the MHS equations with Grad-Rubin conditions, i.e., system \eqref{mhs2}. Furthermore, due to our construction, and the uniqueness granted by the Banach's fixed point, we conclude that the solutions is the only one satisfying 
$$\|B-(0,0,1)\|_{C^{1,\alpha}}\leq M.$$
Thus, Theorem \ref{mainteor} follows.

\begin{appendix}
\section{Complementary results Besov spaces and pseudo-differential operators}\label{appendix1}
In this appendix we include for the sake of completeness a detailed proof of Theorem  \ref{boundedness:Besov} which can be traced back to \cite[Lemma 4.5]{Jurgen-87}.

\subsection{Proof of Theorem \ref{boundedness:Besov}}
The main step of this proof is based on a delicate decomposition of $a$ into building blocks of the form $a_{jk}=\widecheck{\varphi}_j*a(\cdot,\xi)\varphi_k(\xi)$. Note that, for any $x\in \R^d$ the new symbols $a_{jk}$ are $C^\infty_c(\R^d)$ in the $\xi $ variable. Therefore, the inverse Fourier Transform with respect to $\xi$ yields a kernel $K_{jk}(x,z)$ which is integrable with respect to the $z$ variable. Indeed, 

    \begin{equation*} \begin{split}\int_{\R^d}|K_{jk}(x,z)|dz&=\frac{1}{(2\pi)^d}\int_{\R^d}\left|\int_{\R^d}a_{jk}(x,\xi)e^{i\xi\cdot z}\right|dz\\
    &=\frac{1}{(2\pi)^d}\int_{\R^d}\frac{1}{(1+|z|^2)}\left|\int_{\R^d}a_{jk}(x,\xi)\left(1-\Delta_\xi\right)e^{i\xi\cdot z}d\xi\right|dz\\
    &=\frac{1}{(2\pi)^d}\int_{\R^d}\frac{1}{(1+|z|^2)}\left|\int_{\R^d}\left(1-\Delta_\xi\right)a_{jk}(x,\xi)e^{i\xi\cdot z}d\xi\right|dz.
    \end{split}
    \end{equation*}
Iterating this procedure, we find that $K_{jk}$ is integrable with respect to the $z$ variable. Moreover, $\p (a_{jk})u=\p (a_{jk})(u_{k-1}+u_k+u_{k+1})$, with $u_k=\varphi_k(D)u$. This justifies the use of the dyadic partition of unity in the $\xi$ variable, as we inmediately obtain information about the $u_k$, which are the main ingredients of the Besov norm. 

Next, we compute the decay of $z^\alpha K_{jk}$ by usual estimates that involve the Fourier Transform. We see that, since the $\varphi_j$ are supported in an annulus, then $\partial^\alpha \varphi_j(0)=0$ for any multiindex $\alpha$, which means that the integral of $x^\alpha\widecheck{\varphi}$ vanishes for every $\alpha$. By means of Taylor's theorem, we have that for every $x,y\in\R^n$, there is some $c\in [-y,0]$ such that 
\[ \partial^\gamma _\xi a(x-y,\xi)=\sum_{|\alpha|\leq \lfloor s\rfloor}\frac{1}{\alpha!}\partial^\gamma _\xi \partial_x^\alpha a(x,\xi)(-y)^\alpha +\sum_{|\alpha|=\lfloor s\rfloor}\frac{1}{\alpha!}\partial^\gamma _\xi \partial_x^\alpha a(x+c,\xi)(-y)^\alpha.\]
Therefore, convoluting the previous equality agains $\widecheck{\varphi}$, we have the bound
\begin{equation*} 
\begin{split} 
|\widecheck{\varphi}_j*\partial^\gamma _\xi a(\cdot,\xi)|&\leq \sum_{|\alpha|=\lfloor s\rfloor}\frac{1}{\alpha!}\|\partial^\gamma _\xi a(\cdot,\xi)\|_{C^s}\int|y|^s|2^{jd}\widecheck{\varphi}\left(2^j y\right)dy\\
&\leq 2^{-js}\|\partial^\gamma _\xi a(\cdot,\xi)\|_{C^s} \sum_{|\alpha|=\lfloor s\rfloor}\frac{1}{\alpha!}\int_{\R^d}|y|^s|\widecheck{\varphi}\left(y\right)|dy.
\end{split}
\end{equation*}
By means of integration by parts we find that 
\begin{align*}
    z^\alpha K_{jk}(x,z)=\frac{z^\alpha}{(2\pi)^d}\int_{\R^d}\widecheck{\varphi}_j*a(\cdot,\xi)\varphi_k(\xi)e^{i\xi\cdot z}d\xi&=\frac{1}{(2\pi)^d}\int_{\R^d}\widecheck{\varphi}_j*a(\cdot,\xi)\varphi_k(\xi)\frac{\partial_\xi^\alpha}{i^\alpha}e^{i\xi\cdot z}d\xi\\
    &=\frac{i^\alpha}{(2\pi)^d}\int_{\R^d}\partial_\xi^\alpha\left[\widecheck{\varphi}_j*a(\cdot,\xi)\varphi_k(\xi)\right]e^{i\xi\cdot z}d\xi\\
    &=\frac{i^\alpha}{(2\pi)^d}\sum_{\beta\leq\alpha}{\alpha\choose\gamma}\int_{\R^d} \widecheck{\varphi}_j*\partial_\xi^\beta a(\cdot,\xi)\cdot\partial_{\xi}^{\alpha-\beta}\varphi_k(\xi)e^{i\xi\cdot z}d\xi.
    \end{align*}
 Hence, using the fact that    $\varphi_k$ is supported in the ball of radius $2^{k}$, and $\varphi_k=\varphi(2^{-k}\cdot)$, we invoke the decay if $\|\partial^\gamma _\xi a(\cdot,\xi)\|_{C^s}$ to obtain the estimate
    \begin{align*}
    |z^\alpha K_{jk}(x,z)|&\leq \sum_{\beta\leq\alpha}{\alpha\choose\gamma}\|\widecheck{\varphi}_j*\partial_\xi^\beta a(\cdot,\xi)\cdot \partial_\xi^{\alpha-\beta}\varphi_k\|_{L^1_\xi}\\
    &\hspace{-1cm}\leq \frac{C_{\alpha,s}}{\alpha!} \sum_{\beta\leq\alpha}{\alpha\choose\gamma} 2^{-js}2^{-k|\alpha-\beta|}\int_{3/4\cdot2^{k-1}<|\xi|<8/3\cdot 2^{k+1}}(1+|\xi|)^{m-|\beta|}d\xi \leq C\cdot \|a\|_{m,s,0} \cdot 2^{-js+mk+dk-k|\alpha|},
    \end{align*}
where the constant $C$ does not depend on $a$ nor on $j$,$k$ or $\textit{x}$. Next, fix $|\alpha|>n$ and define the function 
    $$\omega_{k}(y)=\left\{\begin{array}{ll}
        2^{-kd}, & |y|\leq 2^{-k} \\
        2^{-kd+2|\alpha|}|y|^{2|\alpha|}, & |y|\geq 2^{-k}
    \end{array}\right. .$$
Due to the above computations, we have that
\[ \|\omega_{k}K_{jk}\|_{\infty}\leq C\cdot \|a\|_{m,s,2|\alpha|}\cdot 2^{-js+km}.\]
Furthermore, noticing that 
\[ \int_{\mathbb{R}^{d}}\frac{1}{\omega_{k}}dy<C, \]
with $C$ independent of $j$ and $k$, leads to the estimate
    \begin{align*}
    \|\p(a_{jk})u\|_{L^1}=\|\p(a_{jk})u_k\|_{L^1}&=\int_{\R^d} \left|\int_{\R^d} K_{jk}(x,x-z)u(z) dz\right|dx\\
    &\leq \|\omega_{k}K_{jk}\|_{\infty}\int_{\R^d} \frac{1}{\omega_{k}}dy\int_{\R^d} |u_k(y)|dy\leq C\cdot \|a\|_{m,s,2|\alpha|}\cdot 2^{-js+km}\|u_k\|_{L^1}.
    \end{align*}
  
We have derived proper estimates on the norms of $\p(a)$, so we can divide $a$ into adequate combinations of these building blocks. Notice that the convolution with $\widecheck{\varphi}_j$ gives us the contribution of the Fourier modes (with respect to the first variable) contained in the annulus $2^{j-1}<|\zeta|<2^{j+1}$. Thus, it is convenient to divide $a$ in pieces where such Fourier modes are smaller than $\xi$ and greater than $\xi$, namely,
\[ a=\sum_{k=-1}^{\infty}\sum_{j=-1}^{k-5}a_{jk}+\sum_{k=-1}^{\infty}\sum_{j=k-4}^{\infty}a_{jk}=a_1+a_2.\]
Taking $u\in \mathcal{S}(\R^n)$, let use separate the different computations for $a_1$ and $a_2$. \\

\paragraph{\underline{Estimates for $\p(a_{1})$}}  
This is the simplest case, as the spectrum of $a_1$ is well controlled by $\xi$. To that purpose, we consider the symbol  $b_{1k}=\sum_{j=-1}^{k-5}a_{jk}$. Then, by the same arguments as above, the kernel $K_{k}$ satisfies
\[ \|K_k w_k\|_{L^1}\leq C\cdot \|a\|_{m,s,2|\alpha|}\cdot 2^{km}\Longrightarrow \|\p(b_{1k})u\|_{L^1}\leq C\cdot \|a\|_{m,s,2|\alpha|}\cdot 2^{km} \|u_k\|_{L^1}.\]
Making use of the fact that the spectrum of $\p(b_{1k})u$ is well controlled due to the construction of the building blocks, we find testing againts a given Schwartz function $v$ that
    \begin{align*}
        \langle\reallywidehat{\p(b_{1k})u_{k'}},v\rangle=\langle\p(b_{1,k})u_{k'},\reallywidehat{v}\rangle &=\frac{1}{(2\pi)^d}\int_{\R^d}\left(\int_{\R^d}b_{1k}(x,\xi)\reallywidehat{u}_{k'}(\xi)e^{i\xi\cdot x}d\xi\right)\reallywidehat{v}(x)dx\\
        &=\frac{1}{(2\pi)^d}\int_{\R^d}\reallywidehat{u}_{k'}(\xi)\left(\int_{\R^d}b_{1k}(x,\xi)e^{i\xi\cdot x}\reallywidehat{v}(x)dx\right)d\xi\\
        &=\frac{1}{(2\pi)^d}\int_{\R^d}\reallywidehat{u}_{k'}(\xi)\left(\int_{\R^d}b_{1k}(x,\xi)\reallywidehat{v(\cdot-\xi)}dx\right)d\xi\\
        &=\frac{1}{(2\pi)^d}\int_{\R^d}\reallywidehat{u}_{k'}(\xi)\tau_{\xi}\reallywidehat{b}_{1,k}(\cdot,\xi)(v)d\xi,
    \end{align*}
where in the last line, $\reallywidehat{b}_{1k}$ is the Fourier transform (as a distribution) with respect to the $x$ variable of $b_{1k}$. Now, $b_{1k}=\chi_{k-5}*a\cdot \varphi_k$ and then its Fourier Transform with respect to $x$ is supported in the ball of radius $C2^{k-5}$. Futhermore, since the functions $\reallywidehat{u}_{k'}$ are supported in an annulus of radius $c2^{k'}\leq |\xi|\leq c2^{k'+1}$, the distribution $\p(b_{1k})u$ has a Fourier transform which is supported in an annulus given by $C2^k\leq |\xi|\leq C2^{k+1}$, where $C$ is independent of $u$ or $k$. Moreover, 
\begin{equation*}
        \begin{split}
            \sum_{k=-1}^\infty 2^{-ks}\|\p(b_{1k})u\|_{L^1}\leq C\|a\|_{m,s,2|\alpha|}\sum_{k=-1}^\infty \sum_{|k-k'|\leq 1} 2^{-k(s-m)}\|u_{k'}\|_{L^1}&\leq C\|a\|_{m,s,2|\alpha|}\sum_{k=-1}^\infty 2^{k(m-s)}\|u_k\|_{L^1}\\
            &=C\|a\|_{m,s,2|\alpha|}\|u\|_{B^{(m-s)_{1,1}}}.
        \end{split}
\end{equation*}
\paragraph{\underline{Estimates for $\p(a_{2})$}}  
Here we can use our previous computations directly, so that 
\begin{align*}
    \|\p(a_2)u\|_{L^1}\leq \sum_{k=-1}^\infty\sum_{j=k-4}^{\infty}\|\p(a_{jk})u\|_{L^1}&\leq C\cdot \|a\|_{m,s,2|\alpha|}\cdot\sum_{k=1}^\infty \sum_{j=k-4}^\infty 2^{-js+km}\|u_k\|_{L^1}\\
    &\leq C\cdot \|a\|_{m,s,2|\alpha|} \cdot  \sum_{k=1}^\infty 2^{k(m-s)}\|u_k\|_{L^1}.
 \end{align*}
Therefore, due to the density of Schwartz functions, and the fact that  $L^1$ embeds continuously into $B^{-s}_{1,1}$, the result follows.
\end{appendix}

\vspace{.2in}
\subsection*{Acknowledgments.}
D. Alonso-Or\'{a}n is supported by the Spanish MINECO through Juan de la Cierva Fellowship FJC2020-046032-I and by the project “An\'alisis Matem\'atico Aplicado y Ecuaciones Diferenciales” Grant PID2022-141187NB-I00 funded by MCIN/ AEIand acronym “AMAED”. Daniel Sánchez-Simón del Pino is funded by the Deutsche Forschungsgemeinschaft (DFG, German Research Foundation) under Germany's Excellence Strategy - GZ 2047/1, Projekt-ID 390685813 and by the Bonn International Graduate School of Mathematics (BIGS) at the Hausdorff Center for Mathematics. J. J. L. Vel\'azquez gratefully acknowledge the support by the Deutsche Forschungsgemeinschaft (DFG) through the collaborative research centre The mathematics of emerging effects (CRC 1060, Project-ID 211504053) and the DFG under Germany's Excellence Strategy-EXC2047/1-390685813.

\bibliographystyle{plain}

\end{document}